\documentclass[a4paper,reqno]{amsart}

\usepackage[T1]{fontenc}
\usepackage[utf8x]{inputenc}
\usepackage[english]{babel}
\usepackage{yfonts}
\usepackage{dsfont}
\usepackage{amscd,amssymb,amsmath,amsthm,amsfonts}
\usepackage{mathtools}
\usepackage{accents}
\usepackage{tensor}
\usepackage{graphicx}
\usepackage{tikz}
\usetikzlibrary{calc,matrix,arrows,decorations.pathmorphing}
\usepackage{calc}
\usepackage[cal=boondox,scr=boondoxo]{mathalfa}
\usepackage{enumitem}
\usepackage{marginnote}
\usepackage{booktabs}
\usepackage{scalerel}[2016/12/29]
\usepackage{url}
\usepackage{contour}
\usepackage[normalem]{ulem}

\usepackage{hyperref}
\hypersetup{colorlinks=true,pageanchor=false,
linkcolor=blue,citecolor=red,urlcolor=red}
\usepackage[all]{hypcap}

\allowdisplaybreaks

\newtheorem{theorem}{Theorem}[section]
\newtheorem{introthm}{Theorem}

\newtheorem{corollary}[theorem]{Corollary}
\newtheorem{proposition}[theorem]{Proposition}
\newtheorem{lemma}[theorem]{Lemma}

\theoremstyle{definition}
\newtheorem{definition}[theorem]{Definition}
\newtheorem{example}[theorem]{Example}

\theoremstyle{remark}
\newtheorem{remark}[theorem]{Remark}

\newcommand{\N}{\mathbb{N}}
\newcommand{\Z}{\mathbb{Z}}

\newcommand{\R}{\mathbb{R}}

\newcommand{\rmt}{\mathrm{t}}

\newcommand{\rmL}{\mathrm{L}}

\newcommand{\rmT}{\mathrm{T}}

\newcommand{\rmV}{\mathrm{V}}

\newcommand{\bfh}{\boldsymbol{h}}

\newcommand{\bfA}{\boldsymbol{A}}
\newcommand{\bfhatA}{\boldsymbol{\hat{A}}}
\newcommand{\bfB}{\boldsymbol{B}}
\newcommand{\bfC}{\boldsymbol{C}}
\newcommand{\bfD}{\boldsymbol{D}}

\newcommand{\bfF}{\boldsymbol{F}}
\newcommand{\bfhatF}{\boldsymbol{\hat{F}}}

\newcommand{\bfM}{\boldsymbol{M}}

\newcommand{\bfP}{\boldsymbol{P}}

\newcommand{\bfS}{\boldsymbol{S}}
\newcommand{\bfT}{\boldsymbol{T}}

\newcommand{\bfW}{\boldsymbol{W}}

\newcommand{\bfalpha}{\boldsymbol{\alpha}}

\newcommand{\bfeta}{\boldsymbol{\eta}}
\newcommand{\bfhateta}{\boldsymbol{\hat{\eta}}}

\newcommand{\bfmu}{\boldsymbol{\mu}}
\newcommand{\bfhatmu}{\boldsymbol{\hat{\mu}}}

\newcommand{\bfsigma}{\boldsymbol{\sigma}}
\newcommand{\bfhatsigma}{\boldsymbol{\hat{\sigma}}}
\newcommand{\bftau}{\boldsymbol{\tau}}

\newcommand{\bfchi}{\boldsymbol{\chi}}
\newcommand{\bfGamma}{\boldsymbol{\Gamma}}

\newcommand{\bfSigma}{\boldsymbol{\Sigma}}

\newcommand{\calA}{\mathcal{A}}
\newcommand{\calB}{\mathcal{B}}
\newcommand{\calC}{\mathcal{C}}
\newcommand{\calD}{\mathcal{D}}
\newcommand{\calE}{\mathcal{E}}

\newcommand{\calG}{\mathcal{G}}
\newcommand{\calH}{\mathcal{H}}

\newcommand{\calK}{\mathcal{K}}
\newcommand{\calL}{\mathcal{L}}

\newcommand{\calR}{\mathcal{R}}
\newcommand{\calS}{\mathcal{S}}

\newcommand{\calV}{\mathcal{V}}

\newcommand{\bcalC}{\mathbfcal{C}}

\renewcommand{\epsilon}{\varepsilon}
\renewcommand{\theta}{\vartheta}
\renewcommand{\phi}{\varphi}
\renewcommand{\Delta}{\varDelta}
\renewcommand{\Gamma}{\varGamma}
\renewcommand{\Sigma}{\varSigma}
\renewcommand{\Lambda}{\varLambda}

\newcommand{\id}{\mathrm{id}}

\newcommand{\lrad}{\operatorname{rad_L}}
\newcommand{\rrad}{\operatorname{rad_R}}

\newcommand{\lev}{\smash{\stackrel{\leftarrow}{\mathrm{ev}}}}
\newcommand{\lcoev}{\smash{\stackrel{\longleftarrow}{\mathrm{coev}}}}
\newcommand{\rev}{\smash{\stackrel{\rightarrow}{\mathrm{ev}}}}
\newcommand{\rcoev}{\smash{\stackrel{\longrightarrow}{\mathrm{coev}}}}

\newcommand{\ladun}{\smash{\stackrel{\leftarrow}{\eta}}}
\newcommand{\ladcoun}{\smash{\stackrel{\leftarrow}{\epsilon}}}
\newcommand{\radun}{\smash{\stackrel{\rightarrow}{\eta}}}
\newcommand{\radcoun}{\smash{\stackrel{\rightarrow}{\epsilon}}}

\DeclareMathOperator{\Forall}{\forall}

\DeclareMathOperator{\sqtimes}{\scaleobj{0.8}{\boxtimes}}

\DeclareMathOperator{\disjun}{\sqcup}

\DeclareMathOperator{\bcs}{\natural}
\newcommand{\colim}{\qopname\relax m{colim}}

\newcommand{\leqs}{\leqslant}
\newcommand{\geqs}{\geqslant}

\newcommand{\mods}[1]{\operatorname{\mathnormal{#1}-mod}}

\newcommand{\SL}{\mathrm{SL}}

\newcommand{\Hom}{\mathrm{Hom}}
\newcommand{\bfHom}{\mathbf{Hom}}

\newcommand{\End}{\mathrm{End}}

\newcommand{\Vect}{\mathrm{Vect}}
\newcommand{\FVect}{\mathrm{FVect}}

\newcommand{\op}{\mathrm{op}}

\DeclareRobustCommand{\one}{\mathbin{\text{\includegraphics[height=\heightof{$\mathbf{1}$}]{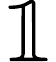}}}}

\newcommand{\Cob}{\mathrm{Cob}}
\newcommand{\bfCob}{\mathbf{Cob}}
\newcommand{\adCob}{\check{\mathrm{C}}\mathrm{ob}}

\newcommand{\bfadCob}{\check{\mathbf{C}}\mathbf{ob}}
\newcommand{\ACob}{\mathrm{ACob}}
\newcommand{\bfACob}{\mathbf{ACob}}
\newcommand{\NCob}{\mathrm{NCob}}
\newcommand{\bfNCob}{\mathbf{NCob}}
\newcommand{\CCob}{\mathrm{CCob}}
\newcommand{\SCob}{\mathrm{SCob}}

\newcommand{\TS}{\bfT \bfS}
\newcommand{\CTS}{\bfC \bfT \bfS}
\newcommand{\CTB}{\bfC \bfT \bfB}
\newcommand{\Lagr}{\calK}
\newcommand{\sig}{n}

\newcommand{\KTan}{\mathrm{KTan}}
\newcommand{\TTan}{\mathrm{TTan}}
\newcommand{\BTan}{\mathrm{BTan}}
\newcommand{\bfCat}{\mathbf{Cat}}
\newcommand{\coCat}{\mathbf{FCat}}

\newcommand{\SCat}{\mathbf{SCat}}

\newcommand{\FCat}{\mathbf{FCat}}
\newcommand{\Lex}{\mathbf{Lex}}

\newcommand{\SLex}{\mathbf{SLex}}

\newcommand{\bfLex}{\mathbf{Lex}}
\newcommand{\Nat}{\mathrm{Nat}}

\newcommand{\cmpl}{\bfC}

\newcommand{\Proj}{\mathrm{Proj}}

\makeatletter
\newcommand{\subalign}[1]{
  \vcenter{
    \Let@ \restore@math@cr \default@tag
    \baselineskip\fontdimen10 \scriptfont\tw@
    \advance\baselineskip\fontdimen12 \scriptfont\tw@
    \lineskip\thr@@\fontdimen8 \scriptfont\thr@@
    \lineskiplimit\lineskip
    \ialign{\hfil$\m@th\scriptstyle##$&$\m@th\scriptstyle{}##$\crcr
      #1\crcr
    }
  }
}
\makeatother

\def\clap#1{\hbox to 0pt{\hss#1\hss}}

\newcommand{\pic}[2][0]{\raisebox{-0.5\height + 2.5pt + #1pt}{\includegraphics{Pictures/#2.pdf}}}

\newcommand{\eend}{\mathcal{E}}
\newcommand{\coend}{\mathcal{L}}

\newcommand{\intL}{\Lambda}

\contourlength{1pt}

\DeclareRobustCommand{\myuline}[1]{
 \ifmmode \text{\uline{$#1$}}
 \else \uline{#1} \fi
}

\newcommand{\dmnsnl}[1]{$#1$-di\-men\-sion\-al}
\newcommand{\hndl}[1]{$#1$-han\-dle}

\newcommand{\mnfld}[1]{$#1$-man\-i\-fold}

\newcommand{\mrphsm}[1]{$#1$-mor\-phism}
\newcommand{\mrphsms}[1]{$#1$-mor\-phisms}
\newcommand{\ctgr}[1]{$#1$-cat\-e\-go\-ry}
\newcommand{\ctgrs}[1]{$#1$-cat\-e\-go\-ries}
\newcommand{\fnctr}[1]{$#1$-func\-tor}
\newcommand{\fnctrs}[1]{$#1$-func\-tors}
\newcommand{\trnsfrmtn}[1]{$#1$-trans\-for\-ma\-tion}
\newcommand{\trnsfrmtns}[1]{$#1$-trans\-for\-ma\-tions}
\newcommand{\mdfctn}[1]{$#1$-mod\-i\-fi\-ca\-tion}
\newcommand{\mdfctns}[1]{$#1$-mod\-i\-fi\-ca\-tions}
\newcommand{\dcrtd}[1]{$#1$-dec\-o\-rat\-ed}
\newcommand{\dcrtn}[1]{$#1$-dec\-o\-ra\-tion}
\newcommand{\dcrtns}[1]{$#1$-dec\-o\-ra\-tions}

\newcommand{\KL}{Ker\-ler--Lyu\-ba\-shen\-ko}
\newcommand{\KLRT}{Ker\-ler--Lyu\-ba\-shen\-ko--Re\-she\-ti\-khin--Tu\-ra\-ev}
\newcommand{\LRT}{Lyu\-ba\-shen\-ko--Re\-she\-ti\-khin--Tu\-ra\-ev}
\newcommand{\RT}{Re\-she\-ti\-khin--Tu\-ra\-ev}

\newcommand{\arxiv}[2]{\href{http://arXiv.org/abs/#1}{\texttt{arXiv:#1} #2}}
\newcommand{\doi}[2]{\href{http://doi.org/#1}{#2}}

\begin{document}

\raggedbottom

\title{On the Relation Between Non-Semisimple TQFTs}

\author[M. De Renzi]{Marco De Renzi}
\address{IMAG, Université de Montpellier, Place Eugène Bataillon, 34090 Montpellier, France}
\email{marco.de-renzi@umontpellier.fr}

\begin{abstract}
 The goal of this paper is to clarify the relation between the Kerler--Lyubashenko TQFT $J_\mathcal{E}$ associated with the adjoint end $\mathcal{E}$ of a (not necessarily semisimple) modular category $\mathcal{C}$ and its renormalized version $V_\mathcal{E}$ based on the modified trace supported by the ideal $\mathrm{Proj}(\mathcal{C})$ of projective objects of $\mathcal{C}$. More precisely, we construct a $3$-dimensional ETQFT $\boldsymbol{\hat{A}}_\mathcal{E}$ that contains both $J_\mathcal{E}$ and $V_\mathcal{E}$. The ETQFT $\boldsymbol{\hat{A}}_\mathcal{E}$ is given by a $2$-functor whose source is the $2$-category of admissible $3$-dimensional cobordisms, and whose target is the $2$-category of finitely complete linear categories. We improve on earlier versions of the construction by dropping the admissibility condition for surfaces. By doing so, we obtain an ETQFT whose circle category $\boldsymbol{\hat{A}}_\mathcal{E}(\boldsymbol{S}^1)$ is equivalent to the category $\mathcal{C}$, as opposed to the ideal $\mathrm{Proj}(\mathcal{C})$.
\end{abstract}

\maketitle
\setcounter{tocdepth}{2}

\section{Introduction}

TQFTs (Topological Quantum Field Theories) are a central notion in low-di\-men\-sion\-al topology. Ever since their axiomatization by Atiyah \cite{A88}, many have contributed to the development of a general framework for their construction. Several approaches are possible, according to whether cobordisms are presented by surgery, by triangulation, or by handle attachment. In this paper, we will focus on those TQFT constructions that are based on surgery presentations.

In their non-semisimple incarnation, TQFTs have existed for more than 25 years. Their first construction is due to Kerler and Lyubashenko \cite{KL01}, who deeply generalized and upgraded results of Hennings \cite{H96} and of Kauffman and Radford \cite{KR94}. More recently, a new wave of constructions has integrated the theory of modified traces, developed by Geer and Patureau with many collaborators, to yield extensions of the original \KL{} TQFTs that will be referred to, here, as renormalized \KL{} TQFTs \cite{DGGPR19}. The extent to which these deserve to be called extensions, however, has not yet been clearly explained. Indeed, there is currently no explicit discussion of the precise relation between the original construction and the renormalized one in the literature. This paper aims at providing one.

The main difficulty in comparing the two TQFTs is that both their sources and targets differ. Indeed, the \KL{} TQFT\footnote{Here, $\eend \in \calC$ denotes the adjoint end of $\calC$, that is, the end of the functor $\_ \otimes \_^* : \calC \times \calC^\op \to \calC$, see Section~\ref{S:ribbon_categories}.} $J_\eend$ associated with a modular category $\calC$ is a braided monoidal functor defined over the category $\CCob$ of \textit{connected} \dmnsnl{3} cobordisms between connected surfaces with (non-empty) connected boundary (first studied by Crane, Yetter, and Kerler, see Definition~\ref{D:connected_cobordisms_cat}), and taking values in the modular category $\calC$ itself. By contrast, the renormalized \KL{} TQFT $V_\eend$ associated with $\calC$ is a symmetric monoidal functor defined over the category of \textit{admissible} $\calC$-decorated cobordisms $\ACob_\calC$ (closer to Atiyah's framework, see Definition~\ref{D:admissible_cobordisms_cat}), and taking values in the category $\FVect_\Bbbk$ of finite-dimensional vector spaces. While it is possible to relate the two functors through natural transformations, like in Corollary~\ref{C:natural_isomorphism}, these isomorphisms can be understood as instances of an overarching construction that encompasses the two TQFTs. Indeed, we define in this paper an ETQFT (Extended TQFT) $\bfhatA_\eend$ that contains both $J_\eend$ and $V_\eend$. As a \fnctr{2}, $\bfhatA_\eend$ is defined over the \ctgr{2} $\bfACob_\calC$ of admissible $\calC$-decorated cobordisms (see Definition~\ref{D:admissible_cobordisms_2-cat}), and takes values in the \ctgr{2} $\coCat_\Bbbk$ of finitely complete linear categories (see Definition~\ref{D:sym_mon_2-cat_of_co_lin_cat}). The source $\bfACob_\calC$ contains both $\CCob$, as (a subcategory of) the category of morphisms from $\varnothing$ to $\bfS^1$, and $\ACob_\calC$, as the category of endomorphisms of $\varnothing$. Similarly, the target $\coCat_\Bbbk$ contains both $\calC$, which is abelian and thus finitely complete (and cocomplete), and $\FVect_\Bbbk$, which is the tensor unit for Kelly's box tensor product of finitely complete linear categories. Our main result, then, can be formulated as follows, by combining Theorem~\ref{T:symmetric_monoidality} and Propositions~\ref{P:KL_TQFT_inside_ETQFT} and \ref{P:renormalized_KL_TQFT_inside_ETQFT}.

\begin{introthm}\label{T:main_theorem}
 If $\calC$ is a modular category with adjoint end $\eend$ over an algebraically closed field $\Bbbk$, then there exists an ETQFT $\bfhatA_\eend : \bfACob_\calC \to \coCat_\Bbbk$ whose circle category satisfies
 \[
  \bfhatA_\eend(\bfS^1) \cong \calC.
 \]
 Furthermore, the restriction of $\bfhatA_\eend$ to $\CCob$ is naturally isomorphic to the \KL{} TQFT $J_\eend : \CCob \to \calC$, while its restriction to $\ACob_\calC$ is naturally isomorphic to the renormalized \KL{} TQFT $V_\eend : \ACob_\calC \to \FVect_\Bbbk$.
\end{introthm}

The statement of Theorem~\ref{T:main_theorem} closely resembles that of \cite[Theorem~1.1]{D21}. Notice however that the source \ctgr{2} $\bfACob_\calC$ considered here is actually larger than the one considered in \cite{D21}, since the admissibility condition is no longer required at the level of \mrphsms{1}. This makes it possible to recover $\calC$ as the circle category of $\bfhatA_\eend$, as opposed to $\Proj(\calC)$. In particular, Theorem~\ref{T:main_theorem} provides a positive answer to a question raised in the introduction of \cite{D21}. Furthermore, the whole renormalized \KL{} TQFT $V_\eend$ is now entirely contained in $\bfhatA_\eend$. At the same time, though, dropping the admissibility condition for \mrphsms{1} forces the use of finitely complete linear categories as target, instead of Cauchy complete ones, in order to ensure monoidality.

A similar result was already obtained by Lagiotis in \cite{L25} for the \ctgr{2} $\bfNCob$ of non-compact cobordisms, which is precisely the sub-\ctgr{2} of $\bfACob_\calC$ whose morphisms have empty decorations. It should be noted, however, that the definition of Lagiotis relies on a conjectural algebraic presentation of $\bfNCob$, which is deduced from the one of \cite{BDSV15}, whose complete proof is yet to appear, at the time of writing. Therefore, the main result of this paper can be seen as a concrete construction of an ETQFT $\bfhatA_\eend : \bfNCob \to \coCat_\Bbbk$ that is conjecturally equivalent to that of \cite[Theorem~1.1]{L25}.

\subsection{Structure of the paper}

Let us outline how the paper is organized into different sections, to help the reader navigate them.

In Section~\ref{S:algebraic_preliminaries}, we recall the algebraic setup for the construction of ETQFTs, and we fix our notation for braided monoidal categories, finitely complete linear categories, and factorizable ribbon categories. In particular, Section~\ref{S:complete_lin_cat} introduces the symmetric monoidal \ctgr{2} $\coCat_\Bbbk$ of finitely complete linear categories, which provides the target for the ETQFT of Theorem~\ref{T:main_theorem} (see Definition~\ref{D:sym_mon_2-cat_of_co_lin_cat}), and Section~\ref{S:completion} explains how to obtain a finitely complete linear category out of an arbitrary linear category by a finite completion procedure.

In Section~\ref{S:topological_preliminaries}, we describe the topological framework for the construction of ETQFTs, and we illustrate how to represent diagrammatically connected cobordisms and bichrome graphs inside them. In particular, Section~\ref{S:cobordisms_with_corners} introduces the symmetric monoidal \ctgr{2} $\bfACob_\calC$ of admissible $\calC$-decorated cobordisms associated with a finite ribbon category $\calC$, which provides the source for the ETQFT of Theorem~\ref{T:main_theorem} (see Definition~\ref{D:admissible_cobordisms_2-cat}), and Sections~\ref{S:Atiyah_cobordisms} and \ref{S:connected_cobordisms} recall several subcategories that support the different TQFTs we want to compare.

In Section~\ref{S:TQFTs}, we recall the construction of several known functorial invariants associated with a modular category $\calC$. We start from the \KL{} TQFT $J_\eend : \CCob \to \calC$ for connected cobordisms, in Section~\ref{S:KL_TQFTs}. Then, we move on to the \KLRT{} functor $F_\eend : \calB_\calC \to \calC$ for $\calC$-labeled bichrome graphs, in Section~\ref{S:KLRT_functors}. We finish with the renormalized \KL{} TQFT $V_\eend : \ACob_\calC \to \FVect_\Bbbk$ for admissible $\calC$-decorated cobordisms, in Section~\ref{S:renormalized_KL_TQFTs}. We also recall several important features of these invariants, that will be later used in the construction of the ETQFT of Theorem~\ref{T:main_theorem}. We point out that the \KL{} TQFT constructed here does not appear, as is, in the work of Kerler and Lyubashenko. Indeed, \cite{KL01} provides the construction of a double functor from a double category of connected cobordisms (between connected surfaces) that only implicitly contains the functor presented here, the framework being significantly more involved. What we refer to as the \KL{} TQFT appeared instead in \cite{BD21}, where the construction is formulated starting from an algebraic presentation of the category of connected cobordisms that first appeared in \cite{BP11}, and that is recalled here in Section~\ref{S:BPHK_Hopf_algebras}. However, since the definition of $J_\eend$ does not actually require knowledge of such an algebraic presentation, we reformulate the construction here in terms of the universal property defining the adjoint end, and prove invariance in Appendix~\ref{A:proof_KL_TQFT}.

In Section~\ref{S:ETQFTs}, we build the ETQFT $\bfhatA_\eend : \bfACob_\calC \to \coCat_\Bbbk$ through the extended universal construction, prove its symmetric monoidality in Theorem~\ref{T:symmetric_monoidality}, and show that it contains both the \KL{} TQFT and its renormalized version, in Propositions~\ref{P:KL_TQFT_inside_ETQFT} and \ref{P:renormalized_KL_TQFT_inside_ETQFT}, respectively. As a consequence, we obtain a family of natural isomorphisms parametrized by objects of the underlying modular category, in Corollary~\ref{C:natural_isomorphism}. Along the way, we recover an explicit description of all the images of generating \mrphsms{1} of the \ctgr{2} of admissible decorated cobordisms that generalizes the one of \cite[Sections~8.2 \& 8.3]{D21}, and that agrees with the one of \cite[Table~1]{L25}. In particular, up to natural isomorphisms:
\begin{enumerate}
 \item A disc $\bfD^2_X : \varnothing \to \bfS^1$ decorated with a single positive framed point labeled by $X \in \calC$ is sent to the linear functor
  \[
   \_ \otimes X : \FVect_\Bbbk \to \calC
  \]
  sending every $V \in \FVect_\Bbbk$ to the copower $V \otimes X \in \calC$ in the orbit of $X$ under the canonical action of $\FVect_\Bbbk$ on $\calC$, see Proposition~\ref{P:unit};
 \item A disc $(\bfD^2_X)^\dagger : \bfS^1 \to \varnothing$ with opposite orientation decorated with a single negative framed point labeled by $X \in \calC$ is sent to the representable linear functor
  \[
   \calC(X,\_) : \calC \to \FVect_\Bbbk
  \]
  sending every $Y \in \calC$ to the morphism vector space $\calC(X,Y)$, see Proposition~\ref{P:counit};
 \item A pair of pants $\bfP^2 : \bfS^1 \disjun \bfS^1 \to \bfS^1$ is sent to the tensor product linear functor
  \[
   \otimes : \calC \sqtimes \calC \to \calC
  \]
  sending every $X \sqtimes X' \in \calC \sqtimes \calC$ to $X \otimes X' \in \calC$, see Proposition~\ref{P:product};
 \item A pair of pants $(\bfP^2)^\dagger : \bfS^1 \to \bfS^1 \disjun \bfS^1$ with opposite orientation is sent to the adjoint linear functor
  \[
   \eend_{[1]} \sqtimes {(\eend_{[2]}^* \otimes \_)} : \calC \to \calC \sqtimes \calC
   \vspace*{6.66pt}
  \]
  sending every $Y \in \calC$ to $\displaystyle \eend_{[1]} \sqtimes {(\eend_{[2]}^* \otimes Y)} = \smash{\int_{X \in \calC}} X \sqtimes {(X^* \otimes Y)} \in \calC \sqtimes \calC$, see Proposition~\ref{P:coproduct}.
\end{enumerate}

\subsection{Acknowledgments}

The author is grateful to Benjamin Haïoun, Aaron Ho\-fer, and Cris Negron for helpful discussions during the preparation of this paper, and used ChatGPT and Claude for finding and fixing mistakes in the statement and proof of Propositions~\ref{P:fc_equivalence} and \ref{P:end}, as well as for references and proofreading.

\section{Algebraic preliminaries}\label{S:algebraic_preliminaries}

In this section, we fix our notations and conventions for the linear categories we will use to perform our constructions, and for the \ctgrs{2} we will choose as a target for our linear representations of cobordisms. We will systematically replace essentially small categories by small equivalent categories, whenever necessary. We will work over an algebraically closed field $\Bbbk$, and the terms linear, vector space, and algebra will be used as shorthand for $\Bbbk$-linear, $\Bbbk$-vector space, and $\Bbbk$-algebra, respectively. Calligraphic, roman, and Greek characters used to denote categories, functors, and natural transformations (such as $\calC$, $F$, and $\alpha$) will appear in their boldface variants to denote \ctgrs{2}, \fnctrs{2}, and \trnsfrmtns{2} (such as $\bcalC$, $\bfF$, and $\bfalpha$). Horizontal and vertical composition will be denoted by $\circ$ and by $\ast$, respectively, while left and right whiskering will be denoted by $\triangleright$ and by $\triangleleft$, respectively. Furthermore, the following prefixes will be used: S will stand for \textit{separate}, as in the \ctgr{2} of separate linear categories $\SCat_\Bbbk$, or in the category of separately left exact linear functors $\SLex_\Bbbk(\myuline{A};\myuline{A'})$ between separate linear categories $\myuline{A}$ and $\myuline{A'}$; F will stand for \textit{finite}, as in the \ctgr{2} of finitely complete linear categories $\FCat_\Bbbk$, or in the category of finite-dimensional vector spaces $\FVect_\Bbbk$.

\subsection{Braided monoidal categories and braided monoidal functors}\label{S:braided_monoidal}

Let us start by recalling some definitions. A (\textit{strict}) \textit{monoidal category} is a category $\calC$ equipped with a \textit{tensor product} $\otimes : \calC \times \calC \to \calC$ and a \textit{tensor unit} $\one \in \calC$ satisfying:
\begin{itemize}
 \item $(X \otimes Y ) \otimes Z = X \otimes (Y \otimes Z)$ for all $X,Y,Z \in \calC$;
 \item $\one \otimes X = X = X \otimes \one$ for every $X \in \calC$.
\end{itemize}
See \cite[Section~2.1]{EGNO15} for the more general notion of a (not necessarily strict) monoidal category, and \cite[Section~2.4]{EGNO15} for an explanation of why we can (and will tacitly) assume monoidal structures to be strict.

A \textit{braided monoidal category} is a monoidal category $\calC$ equipped with a natural family of \textit{braiding} isomorphisms $c_{X,Y} : X \otimes Y \to Y \otimes X$ for all $X,Y \in \calC$ satisfying:
\begin{itemize}
 \item $c_{X \otimes Y,Z} = (c_{X,Z} \otimes \id_Y) \circ (\id_X \otimes c_{Y,Z})$ for all $X,Y,Z \in \calC$;
 \item $c_{X,Y \otimes Z} = (\id_Y \otimes c_{X,Z}) \circ (c_{X,Y} \otimes \id_Z)$ for all $X,Y,Z \in \calC$.
\end{itemize}
A braided monoidal category $\calC$ is \textit{symmetric} if
\[
 c_{Y,X} \circ c_{X,Y} = \id_{X \otimes Y}
\]
for all $X,Y \in \calC$. See \cite[Section~8.1]{EGNO15} for more details.

A \textit{monoidal functor} between monoidal categories $\calC$ and $\calC'$, with tensor products $\otimes : \calC \times \calC \to \calC$ and $\otimes' : \calC' \times \calC' \to \calC'$ and with tensor units $\one \in \calC$ and $\one' \in \calC'$, is a functor $F : \calC \to \calC'$ equipped with a natural family of isomorphisms $\chi_{X,Y} : F(X) \otimes' F(Y) \to F(X \otimes Y)$ for all $X,Y \in \calC$ and with an isomorphism $\iota : \one' \to F(\one)$ satisfying:
\begin{itemize}
 \item $\chi_{X,Y \otimes Z} \circ (\id_{F(X)} \otimes' \chi_{Y,Z}) = \chi_{X \otimes Y,Z} \circ (\chi_{X,Y} \otimes' \id_{F(Z)})$ for all $X,Y,Z \in \calC$;
 \item $\chi_{\one,X} \circ (\iota \otimes' \id_{F(X)}) = \id_{F(X)} = \chi_{X,\one} \circ (\id_{F(X)} \otimes' \iota)$ for every $X \in \calC$.
\end{itemize}
A \textit{braided monoidal functor} between braided monoidal categories $\calC$ and $\calC'$, with braidings $c_{X,Y} : X \otimes Y \to Y \otimes X$ for all $X,Y \in \calC$ and $c'_{X',Y'} : X' \otimes' Y' \to Y' \otimes' X'$ for all $X',Y' \in \calC'$, is a monoidal functor $F : \calC \to \calC'$ satisfying:
\begin{itemize}
 \item $\chi_{Y,X} \circ c'_{F(X),F(Y)} = F(c_{X,Y}) \circ \chi_{X,Y}$ for all $X,Y \in \calC$.
\end{itemize}
A \textit{symmetric monoidal functor} is a braided monoidal functor $F : \calC \to \calC'$ between symmetric monoidal categories $\calC$ and $\calC'$.

Although we will not recall the specifics here, all these notions can be extended to the setting of \ctgrs{2}, see \cite[Appendix~D]{D17} for a summary, or \cite[Appendices~A \& C]{S11} for all details.

\subsection{Finitely complete linear categories}\label{S:complete_lin_cat}

Next, we recall the definition of the symmetric monoidal \ctgr{2} $\FCat_\Bbbk$ of finitely complete small linear categories, left exact linear functors, and natural transformations, which will provide the target \ctgr{2} for our ETQFT. This represents a change with respect to \cite[Section~4]{D21}, where the target was chosen to be the symmetric monoidal \ctgr{2} of Cauchy complete (meaning additive and idempotent complete) small linear categories. A reference for all the relevant definitions and notations concerning symmetric monoidal \ctgrs{2}, which were first introduced in \cite{M00}, is provided by \cite[Appendix~D]{D17}, which summarizes \cite[Appendices~A \& C]{S11}.

We denote by $\Vect_\Bbbk$ the symmetric monoidal category of vector spaces, and we use the notation $\Hom_\Bbbk(V,V')$ for the vector space $\Vect_\Bbbk(V,V')$ of linear maps between vector spaces $V,V' \in \Vect_\Bbbk$.

First, let us recall the definition of the \textit{symmetric monoidal \ctgr{2} $\bfCat_\Bbbk$ of small linear categories}.
An object of $\bfCat_\Bbbk$ is a small linear category $A$, that is, a small $\Vect_\Bbbk$-enriched category. A \mrphsm{1} of $\bfCat_\Bbbk$ between small linear categories $A$ and $A'$ is a linear functor $F : A \rightarrow A'$, that is, a $\Vect_\Bbbk$-enriched functor. A \mrphsm{2} of $\bfCat_\Bbbk$ between linear functors $F,F' : A \rightarrow A'$ is a natural transformation $\alpha : F \Rightarrow F'$. The horizontal composition of $\bfCat_\Bbbk$ is denoted $\circ$, and it is given by the standard composition of functors and by the horizontal composition of natural transformations. Similarly, the vertical composition of $\bfCat_\Bbbk$ is denoted $\ast$, and it is given by the standard vertical composition of natural transformations. We use the notation $\bfHom_\Bbbk(A,A')$ for the category $\bfCat_\Bbbk(A,A')$ of linear functors between small linear categories $A,A' \in \bfCat_\Bbbk$, and the notation $\Nat(F,F')$ for the vector space $\bfHom_\Bbbk(A,A')(F,F')$ of natural transformations between linear functors $F,F' \in \bfHom_\Bbbk(A,A')$. The tensor product of $\bfCat_\Bbbk$ is denoted $\otimes$, and it is given by the standard $\Vect_\Bbbk$-en\-riched tensor product of linear categories, linear functors, and natural transformations. The tensor unit of $\bfCat_\Bbbk$ is denoted $\Bbbk$, and it is given by any linear category having a single object, and such that its single endomorphism space is isomorphic to the base field $\Bbbk$. The symmetric braiding of $\bfCat_{\Bbbk}$ is given by the transposition of tensor factors.

Next, a \textit{separate linear category} is a finite sequence $(A_1,\ldots,A_n)$ of linear categories,
and its \textit{underlying linear category} is $A_1 \otimes \ldots \otimes A_n$. The underlying linear category of the empty sequence $\varnothing$ is defined to be $\Bbbk$. In the following, we will abusively use the notation $\myuline{A}$ to denote both a separate linear category $(A_1,\ldots,A_n)$ and its underlying linear category $A_1 \otimes \ldots \otimes A_n$.

Then, let us fix a separate linear category $\myuline{A} = (A_1,\ldots,A_n)$. For every choice of
\begin{itemize}
 \item a category $J$,
 \item an integer $1 \leqs i \leqs n$,
 \item an object $X_k \in A_k$ for every integer $1 \leqs k \leqs n$ with $i \neq k$,
 \item a functor $D_i : J \to A_i$,
\end{itemize}
we refer to the functor $\myuline{D} : J \to \myuline{A}$ given by
\[
 \myuline{D}(j) := X_1 \otimes \ldots \otimes X_{i-1} \otimes D_i(j) \otimes X_{i+1} \otimes \ldots \otimes X_n
\]
for every $j \in J$ as a \textit{separate diagram in $\myuline{A}$}. A separate diagram $\myuline{D} : J \to \myuline{A}$ is \textit{finite} if $J$ has finite set of objects and finite sets of morphisms. A (co)limit of a separate (finite) diagram in $\myuline{A}$ is called a \textit{separate} 	(\textit{finite co})\textit{limit}. We say that $\myuline{A}$ is \textit{separately} (\textit{finitely co})\textit{complete} if it admits all separate (finite co)limits. When $n=1$, separate (finite co)limits coincide with (finite co)limits, and a separately (finitely co)complete linear category is simply (\textit{finitely co})\textit{complete}.

Now, let us fix separate linear categories $\myuline{A} = (A_1,\ldots,A_n)$, $\myuline{A'} = (A'_1,\ldots,A'_{n'})$. A \textit{separate linear functor} $\myuline{F} : \myuline{A} \to \myuline{A'}$ is a linear functor that preserves separate diagrams, meaning that, if $\myuline{D}$ is a separate diagram in $\myuline{A}$, then $\myuline{F} \circ \myuline{D}$ is a separate diagram in $\myuline{A'}$. When $n'=1$, a separate linear functor is simply a linear functor. We say a separate linear functor $\myuline{F} : \myuline{A} \to \myuline{A'}$ is \textit{separately left exact} if it preserves separate finite limits, meaning that, whenever a separate finite diagram $\myuline{D}$ in $\myuline{A}$ admits a limit in $\myuline{A}$, then its image under $\myuline{F}$ is the limit of the separate finite diagram $\myuline{F} \circ \myuline{D}$ in $\myuline{A'}$. Notice that $\myuline{A}$ is not required to admit all separate finite limits, since the condition only applies to those separate finite limits that do exist in $\myuline{A}$. The analogous definition for separate finite colimits yields the notion of a \textit{separately right exact} separate linear functor, and a separate linear functor that is simultaneously separately left and right exact is \textit{separately exact}. When $n=n'=1$, separate linear functors coincide with linear functors, and separately (left/right) exact linear functors are simply (\textit{left}/\textit{right}) \textit{exact}.

\begin{remark}\label{R:adjoints_and_(co)limits}
 Let us recall a few standard examples of separately left/right exact linear functors.
 \begin{enumerate}
  \item A cone $L$ over a diagram $D : J \to A$ in a linear category $A$ is a limit if and only if, for every object $X$ of $A$, the cone $A(X,L)$ over the diagram $A(X,\_) \circ D : J \to \Vect_\Bbbk$ is a limit in $\Vect_\Bbbk$. In other words, representable linear functors $A(X,\_) : A \to \Vect_\Bbbk$ preserve and jointly reflect limits for all $X \in A$, see for instance \cite[Section~3.2]{K82}. In particular, if $\myuline{X} \in \myuline{A}$ is an object of a separate linear category $\myuline{A}$, then the representable linear functor $\myuline{A}(\myuline{X},\_) : \myuline{A} \to \Vect_\Bbbk$ is separately left exact.
  \item If $R : A \to A'$ and $L : A' \to A$ are adjoint linear functors between linear categories $A$ and $A'$, then the right adjoint $R$ preserves all limits, and the left adjoint $L$ preserves all colimits, see for instance \cite[Section~3.2]{K82}. In particular, if $\myuline{R} : \myuline{A} \to \myuline{A'}$ and $\myuline{L} : \myuline{A'} \to \myuline{A}$ are adjoint separate linear functors between separate linear categories $\myuline{A}$ and $\myuline{A'}$, then the right adjoint $\myuline{R}$ is separately left exact, and the left adjoint $\myuline{L}$ is separately right exact.
 \end{enumerate}
\end{remark}

We will focus on separate finite limits and separate left exactness, from now on, although we will sometimes make use of the dual notions too. For a pair of separate linear categories $\myuline{A} = (A_1,\ldots,A_n)$ and $\myuline{A'} = (A'_1,\ldots,A'_{n'})$, we denote by $\SLex_\Bbbk(\myuline{A};\myuline{A'}) = \SLex_\Bbbk(A_1,\ldots,A_n;A'_1,\ldots,A'_{n'})$ the category of separately left exact separate linear functors from $\myuline{A}$ to $\myuline{A'}$ and their natural transformations, which is a full subcategory of $\bfHom_\Bbbk(\myuline{A},\myuline{A'}) = \bfHom_\Bbbk(A_1 \otimes \ldots \otimes A_n,A'_1 \otimes \ldots \otimes A'_{n'})$. When $n=n'=1$, we denote the same category by $\bfLex_\Bbbk(A,A')$ instead.

The \textit{symmetric monoidal \ctgr{2} of separate linear categories} is the symmetric monoidal \ctgr{2} $\SCat_\Bbbk$ whose objects are small separate linear categories $\myuline{A}$, whose categories of morphisms $\SCat_\Bbbk(\myuline{A},\myuline{A'})$ are given by $\SLex_\Bbbk(\myuline{A};\myuline{A'})$ for all $\myuline{A}, \myuline{A'} \in \SCat_\Bbbk$, whose tensor product is given by concatenation at the level of objects and by the restriction of the tensor product of $\bfCat_\Bbbk$ at the level of morphisms, whose tensor unit is given by the empty object, and whose symmetric braiding is given by the restriction of the symmetric braiding of $\bfCat_\Bbbk$.

If $A$ and $A'$ are small finitely complete linear categories, then their \textit{box tensor product} is a small finitely complete linear category $A \sqtimes A'$ equipped with a separately left exact fully faithful linear functor $K_{(A,A')} : (A,A') \to A \sqtimes A'$ that satisfies the following universal property: for every finitely complete linear category $A''$, restriction along $K_{(A,A')}$ induces an equivalence
\begin{align*}
 \_ \circ K_{(A,A')} : \Lex_\Bbbk(A \sqtimes A',A'') &\to \SLex_\Bbbk(A,A';A'') \\*
 F &\mapsto F \circ K_{(A,A')}.
\end{align*}
This universal property determines $A \sqtimes A'$ up to (a unique up to isomorphism) equivalence. We recall a proof of its existence in Section~\ref{S:completion}.

\begin{definition}\label{D:sym_mon_2-cat_of_co_lin_cat}
 The \textit{symmetric monoidal \ctgr{2} of finitely complete small linear categories} is obtained from the sub-\ctgr{2} $\FCat_\Bbbk$ of $\bfCat_\Bbbk$ whose objects are finitely complete small linear categories $A$ and whose categories of morphisms $\FCat_\Bbbk(A,A')$ are given by $\bfLex_\Bbbk(A,A')$ for all $A,A' \in \FCat_\Bbbk$, by specifying the box tensor product $\sqtimes$ as a tensor product, by specifying the small linear category $\FVect_\Bbbk$ of finite-dimensional vector spaces as a tensor unit, and by specifying the transposition of tensor factors as a symmetric braiding.
\end{definition}

For every finitely complete linear category $A$, the \textit{canonical action of $\FVect_\Bbbk$ on $A$} is the linear functor $\rho : \FVect_\Bbbk \to \bfLex_\Bbbk(A,A)$ sending every $V \in \FVect_\Bbbk$ to the copower functor $V \otimes \_ : A \to A$ defined by the linear isomorphism
\begin{equation}\label{E:canonical_Vect-action}
 A(V \otimes X,Y) \cong \Hom_\Bbbk(V,A(X,Y))
\end{equation}
of \cite[Equation~(3.44)]{K82}, which is natural in $V \in \FVect_\Bbbk$ and $X,Y \in A$.
Every linear functor $F : A \to A'$ between finitely complete linear categories $A$ and $A'$ intertwines the canonical actions of $\FVect_\Bbbk$ on $A$ and $A'$, in the sense that we have an isomorphism
\begin{equation}\label{E:canonical_Vect-action_intertwiner}
 F(V \otimes X) \cong V \otimes F(X)
\end{equation}
in $A'$ that is natural in $V \in \FVect_\Bbbk$ and $X \in A$.

\subsection{Finite completion \texorpdfstring{$2$-functor}{2-functor}}\label{S:completion}

Let us move on to explain a finite completion procedure that preserves all the existing separate finite limits, and let us give a criterion for equivalence under finite completion.

First of all, every separate linear category $\myuline{A} = (A_1,\ldots,A_n)$ can be completed by adding all the missing separate finite limits without changing the existing ones. Indeed, we define the \textit{finite completion $\smash{\myuline{\hat{A}}}$ of $\myuline{A}$} to be the smallest replete full linear subcategory of $\SLex_\Bbbk(\myuline{A};\Vect_\Bbbk)^\op$ that is closed under finite limits and that contains all the representable linear functors $\myuline{A} (\myuline{X},\_)$ for $\myuline{X} \in \myuline{A}$. Notice that $\myuline{\hat{A}}$ is a (single) small linear category, even when $n > 1$. Then, thanks to \cite[Theorem~10.2]{K80}, the linear functor
\begin{align*}
 K_{\myuline{A}} : \myuline{A} &\to \myuline{\hat{A}} \\*
 \myuline{X} &\mapsto \myuline{A} (\myuline{X},\_)
\end{align*}
is separately left exact and fully faithful. Notice that the definition of $K_{\myuline{A}}$ explains the use of the opposite in the definition of $\smash{\myuline{\hat{A}}}$, since a morphism $\myuline{h} \in \myuline{A}(\myuline{X},\myuline{Y})$ induces a natural transformation $\myuline{A}(\myuline{h},\_) : \myuline{A}(\myuline{Y},\_) \Rightarrow \myuline{A}(\myuline{X},\_)$. Furthermore, $K_{\myuline{A}}$ satisfies the following universal property: for every finitely complete linear category $A'$, restriction along $K_{\myuline{A}}$ induces an equivalence
\begin{align*}
 \_ \circ K_{\myuline{A}} : \Lex_\Bbbk(\myuline{\hat{A}},A') &\to \SLex_\Bbbk(\myuline{A};A') \\*
 F &\mapsto F \circ K_{\myuline{A}}.
\end{align*}
This universal property determines $\myuline{\hat{A}}$ up to (a unique up to isomorphism) equivalence. If $\myuline{F} : \myuline{A} \to \myuline{A'}$ is a separately left exact separate linear functor between separate linear categories $\myuline{A}$ and $\myuline{A'}$, then we denote by $\myuline{\hat{F}} : \myuline{\hat{A}} \to \myuline{\hat{A}'}$ the unique (up to isomorphism) left exact linear functor that admits a natural isomorphism $\theta_{\myuline{F}} : \myuline{\hat{F}} \circ K_{\myuline{A}} \Rightarrow K_{\myuline{A'}} \circ \myuline{F}$. Similarly, if $\myuline{\alpha} : \myuline{F} \Rightarrow \myuline{F'}$ is a natural transformation between separately left exact separate linear functors $\myuline{F},\myuline{F'} : \myuline{A} \to \myuline{A'}$ between separate linear categories $\myuline{A}$ and $\myuline{A'}$, then we denote by $\myuline{\hat{\alpha}} : \myuline{\hat{F}} \Rightarrow \myuline{\hat{F}'}$ the unique natural transformation satisfying $\theta_{\myuline{F'}} \ast (\myuline{\hat{\alpha}} \triangleleft K_{\myuline{A}}) = (K_{\myuline{A'}} \triangleright \myuline{\alpha}) \ast \theta_{\myuline{F}}$. Then, the operation of finite completion defines a \fnctr{2} $\cmpl : \SCat_{\Bbbk} \rightarrow \coCat_\Bbbk$.

Then, the box tensor product $A \sqtimes A'$ of small finitely complete linear categories $A$ and $A'$ can be constructed as the finite completion of the pair $(A,A')$, and the tensor unit for $\sqtimes$ can be constructed as the finite completion of the empty sequence $\varnothing$, which can naturally be identified with $\FVect_\Bbbk$. For a proof of the axioms of a symmetric monoidal \ctgr{2}, compare with \cite[Section~6.5]{K82}, where the free cocompletion operation is considered instead. Then, the next statement follows directly from the definitions.

\begin{proposition}\label{P:completion}
 The \fnctr{2} $\cmpl : \SCat_\Bbbk \to \FCat_\Bbbk$ is symmetric mo\-noi\-dal.
\end{proposition}

The operation of finite completion can be directly extended to \fnctrs{2} with target $\SCat_\Bbbk$, as well as to \trnsfrmtns{2} between them. Indeed, for every \ctgr{2} $\bcalC$, every \fnctr{2} $\bfF : \bcalC \rightarrow \SCat_\Bbbk$ automatically induces a \fnctr{2} $\bfhatF : \bcalC \rightarrow \FCat_\Bbbk$, which we call the \textit{finite completion} of $\bfF$, and which is simply defined as the composition $\cmpl \circ \bfF$. Similarly, every \trnsfrmtn{2} $\bfsigma : \bfF \Rightarrow \bfF'$ between \fnctrs{2} $\bfF,\bfF' : \bcalC \rightarrow \SCat_\Bbbk$ induces a \trnsfrmtn{2} $\bfhatsigma : \bfhatF \Rightarrow \bfhatF'$, which we call the \textit{finite completion} of $\bfsigma$, and which is simply defined as the left whiskering $\cmpl \triangleright \bfsigma$.

The finite completion \fnctr{2} $\cmpl$ allows us to work with objects of $\SCat_{\Bbbk}$, provided we consider them to be equivalent whenever their finite completions are equivalent in $\FCat_{\Bbbk}$. This motivates the next definition.

We say two separate linear categories $\myuline{A}$ and $\myuline{A'}$ are \textit{finite completion equivalent} if their finite completions $\myuline{\hat{A}}$ and $\myuline{\hat{A}'}$ are equivalent, and we say a separately left exact separate linear functor $\myuline{F} : \myuline{A} \to \myuline{A'}$ is a \textit{finite completion equivalence} if its finite completion $\myuline{\hat{F}} : \myuline{\hat{A}} \to \myuline{\hat{A}'}$ is an equivalence. Then, let us give a finite completion equivalence criterion.

We say a full linear subcategory $\myuline{B}$ of a separate linear category $\myuline{A} = (A_1,\ldots,A_n)$ is \textit{separate} (\textit{finite co})\textit{limit dense} if the smallest replete full linear subcategory of $\myuline{A}$ that is closed under separate (finite co)limits and that contains $\myuline{B}$ is $\myuline{A}$. When $n=1$, separate (finite co)limit dense subcategories are simply (\textit{finite co})\textit{limit dense}.

\begin{example}
 We say a full linear subcategory $B$ of a linear category $A$ is (\textit{finitely}) \textit{additively dense} if the smallest replete full linear subcategory of $A$ that is closed under (finite) direct sums and that contains $B$ is $A$. If $B$ is (finitely) additively dense in $A$, then it is both (finite) limit dense and (finite) colimit dense in $A$, and its objects are (\textit{finite}) \textit{additive generators} of $A$.
\end{example}

If $\myuline{A'} = (A'_1,\ldots,A'_{n'})$ is a separate linear category, we say a full separate linear functor $\myuline{F} : \myuline{A} \to \myuline{A'}$ is \textit{separate} (\textit{finite co})\textit{limit dense} if its image is a separate (finite co)limit dense linear subcategory of $\myuline{A'}$. When $n=1$, a separate linear functor is simply a linear functor, and separate (finite co)limit dense functors are simply (\textit{finite co})\textit{limit dense}.

\begin{proposition}\label{P:fc_equivalence}
 Let $\myuline{F} : \myuline{A} \to \myuline{A'}$ be a separately left exact separate linear functor between small separate linear categories $\myuline{A}$ and $\myuline{A'}$. If $\myuline{F}$ is fully faithful and separate finite limit dense, and if there exists a separate finite limit dense full linear subcategory $\myuline{B}$ of $\myuline{A}$ such that, for every $\myuline{Y} \in \myuline{B}$, the linear functor $\myuline{A'}(\_,\myuline{F}(\myuline{Y})) : (\myuline{A'})^\op \to \Vect_\Bbbk$
 is separately right exact, then $\myuline{F}$ is a finite completion equivalence.
\end{proposition}

We postpone the proof of Proposition~\ref{P:fc_equivalence} to Appendix~\ref{A:fc_equivalence}.

\subsection{Factorizable ribbon categories}\label{S:ribbon_categories}

We finish this section by recalling the definition of modular categories, ends, and traces on ideals, which will provide the fundamental algebraic tools for our construction of ETQFTs.

A monoidal category $\calC$ is \textit{left rigid} if every object $X \in \calC$ admits a \textit{left dual} $X^* \in \calC$, which can be equipped with \textit{left evaluation} and \textit{coevaluation} morphisms $\lev_X : X^* \otimes X \to \one$ and $\lcoev_X : \one \to X \otimes X^*$ satisfying:
\begin{itemize}
 \item $(\id_X \otimes \lev_X) \circ (\lcoev_X \otimes \id_X) = \id_X$;
 \item $(\lev_X \otimes \id_{X^*}) \circ (\id_{X^*} \otimes \lcoev_X) = \id_{X^*}$.
\end{itemize}
See \cite[Section~2.10]{EGNO15} for more details. A \textit{ribbon category} is a left rigid braided monoidal category $\calC$ equipped with a natural family of \textit{twist} isomorphisms $\vartheta_X : X \to X$ for every $X \in \calC$ satisfying:
\begin{itemize}
 \item $\vartheta_{X \otimes Y} = c_{Y,X} \circ c_{X,Y} \circ (\vartheta_X \otimes \vartheta_Y)$ for all $X,Y \in \calC$;
 \item $(\vartheta_X)^* = \vartheta_{X^*}$ for every $X \in \calC$.
\end{itemize}
See \cite[Section~8.10]{EGNO15} for more details. A ribbon category $\calC$ is automatically \textit{rigid}, meaning that the left dual $X^* \in \calC$ of an object $X \in \calC$ is also a \textit{right dual}, which can be equipped with \textit{right evaluation} and \textit{coevaluation} morphisms $\rev_X : X \otimes X^* \to \one$ and $\rcoev_X : \one \to X^* \otimes X$ defined as:
\begin{itemize}
 \item $\rev_X = \lev_X \circ c_{X,X^*} \circ (\vartheta_X \otimes \id_{X^*})$;
 \item $\rcoev_X = (\id_{X^*} \otimes \vartheta_X^{-1}) \circ c_{X^*,X}^{-1} \circ \lcoev_X$.
\end{itemize}
Morphisms in a ribbon category $\calC$ can be represented graphically using Penrose diagrams, with structure morphisms represented as shown for all $X,Y \in \calC$.
\[
 \pic{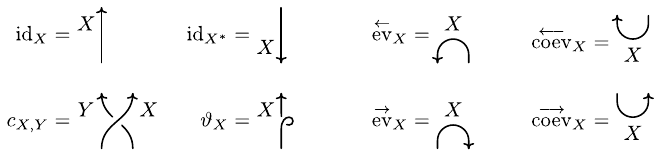}
\]
Penrose diagrams are read from bottom to top, with composition given by vertical stacking, and tensor product given by horizontal juxtaposition. The axioms satisfied by the ribbon structure of $\calC$ imply the invariance of this diagrammatic calculus under planar isotopy and framed Reidemeister moves.

Following \cite[Definition~1.8.5]{EGNO15}, a \textit{finite category} is a linear category $\calC$ that is equivalent to the category $\mods{A}$ of finite-dimensional left $A$-modules for a finite-dimensional algebra $A$. An equivalent and more explicit characterization is provided by \cite[Definition~1.8.6]{EGNO15}. In a finite category $\calC$, every object $X \in \calC$ admits a projective cover $P_X \in \calC$ equipped with an epimorphism $\varepsilon_X : P_X \to X$. A projective generator of $\calC$ is a projective object $G \in \calC$ that covers every simple object $X \in \calC$. In particular, if $I(\calC)$ is a set of representatives of isomorphism classes of simple objects of $\calC$, then
\begin{equation}\label{E:projective_generator}
 G := \bigoplus_{X \in I(\calC)} P_X
\end{equation}
is a projective generator of $\calC$.

A finite category $\calC$ is \textit{unimodular} if it is rigid and $P_{\one}^* \cong P_{\one}$, compare with \cite[Definition~6.5.7]{EGNO15}. An object $X \in \calC$ of a braided monoidal category is said to be transparent if it belongs to the Müger center $M(\calC)$ of $\calC$, which means its double braiding satisfies $c_{Y,X} \circ c_{X,Y} = \id_{X \otimes Y}$ for every object $Y \in \calC$. A unimodular ribbon category $\calC$ is \textit{factorizable} if all its transparent objects are trivial, meaning that they are isomorphic to direct sums of copies of the tensor unit $\one$. A factorizable ribbon category $\calC$ is sometimes also called a \textit{modular} category. The definition given above is equivalent to several others, as proved in \cite[Theorem~1.1]{S16}.

If $\calC$ is a finite rigid monoidal category, then the \textit{adjoint end $\calE \in \calC$} and the \textit{coadjoint coend $\calL \in \calC$} are defined as
\begin{align*}
 \eend &= \int_{X \in \calC} X \otimes X^*, &
 \coend &= \int^{X \in \calC} X^* \otimes X,
\end{align*}
that is, as the end of the functor $\_ \otimes \_^* : \calC \times \calC^\op \to \calC$ and as the coend of the functor $\_^* \otimes \_ : \calC^\op \times \calC \to \calC$, with structure morphisms
\begin{align*}
 \{ i_X : \eend \to X \otimes X^* &\mid X \in \calC \}, &
 \{ j_X : X^* \otimes X \to \coend &\mid X \in \calC \},
\end{align*}
respectively, compare with \cite[Deﬁnition~3.2]{S15} for the terminology. It follows from \cite[Corollary~5.1.8]{KL01} that, if $G$ is a projective generator of $\calC$, and $\calB$ is a basis of $\End_\calC(G)$, then $\eend$ is the kernel of the morphism
\[
 (f \otimes \id_{G^*} - \id_G \otimes f^*)_{f \in \calB} : G \otimes G^* \to \bigoplus_{f \in \calB} G \otimes G^*,
\]
while $\calL$ is the cokernel of the morphism
\[
 (f^* \otimes \id_G - \id_{G^*} \otimes f)_{f \in \calB} : \bigoplus_{f \in \calB} G^* \otimes G \to G^* \otimes G.
\]

Let $\Proj(\calC)$ denote the full subcategory of projective objects of $\calC$, which is the smallest non-zero ideal in $\calC$, in the sense that it is closed under retracts and absorbent under tensor products, see \cite[Section~4.4]{GKP10}. A \textit{trace} on $\Proj(\calC)$ is a family of linear maps
\[
 \rmt := \{ \rmt_X : \End_\calC(X) \to \Bbbk \mid X \in \Proj(\calC) \}
\]
satisfying:
\begin{enumerate}
 \item \textit{Cyclicity}: $\rmt_X \left( \pic{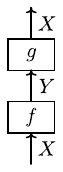} \right) = \rmt_Y \left( \pic{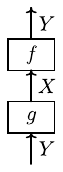} \right)$ \par \bigskip
  for all $X,Y \in \Proj(\calC)$, $f \in \calC(X,Y)$, and $g \in \calC(Y,X)$;
  \bigskip
 \item \textit{Partial trace}: $\rmt_{X \otimes Y} \left( \pic{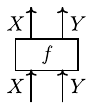} \right) = \rmt_X \left( \pic{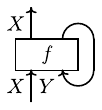} \right)$ \par \bigskip
  for all $X \in \Proj(\calC)$, $Y \in \calC$, and $f \in \End_{\calC}(X \otimes Y)$.
\end{enumerate}
A trace $\rmt$ on $\Proj(\calC)$ is \textit{non-degenerate} if, for all $X \in \Proj(\calC)$ and $Y \in \calC$, the bilinear pairing
\begin{align}
 \rmt_X(\_ \circ \_) : \calC(Y,X) \times \calC(X,Y) &\to \Bbbk \label{E:non-deg_m-trace_pairing} \\
 (g,f) &\mapsto \rmt_X(g \circ f) \nonumber
\end{align}
is non-degenerate. Thanks to \cite[Corollary~5.6]{GKP18}, for a unimodular ribbon category $\calC$ there exists a trace $\rmt$ on $\Proj(\calC)$, which is unique up to scalar and furthermore non-degenerate. By contrast, the standard categorical trace vanishes on $\Proj(\calC)$ as soon as $\calC$ is non-semisimple, since it always vanishes on every proper ideal in $\calC$, see \cite[Lemma~2.6]{BGR23}.

\section{Topological preliminaries}\label{S:topological_preliminaries}

In this section, we fix our notations and conventions for the several categories and \ctgrs{2} of \dmnsnl{3} cobordisms that we will consider in the following. To help the reader navigate the notation, let us highlight that, if $\calC$ is a finite ribbon category, the subscript $\calC$ will signal the presence of \dcrtns{\calC} on cobordisms, while its absence will mean that \dcrtns{\calC} are empty. Furthermore, the following prefixes will be used: A will stand for \textit{admissible}, as in the category of admissible cobordisms $\ACob_\calC$; N will stand for \textit{non-compact}, as in the category of non-compact cobordisms $\NCob$; C will stand for \textit{connected}, as in the category of connected cobordisms $\CCob$.

\subsection{Bichrome graphs}\label{S:bichrome_graphs}

We start by fixing our terminology for decorated graphs, so let us consider a finite ribbon category $\calC$. If $\Sigma$ is a \dmnsnl{2} cobordism from $\Gamma$ to $\Gamma'$, then a \textit{$\calC$-labeled bichrome set $P \subset \Sigma$} is a discrete subset $P$ of $\Sigma$ disjoint from the boundary of $\Sigma$ whose points are equipped with a framing, given by a non-zero tangent vector, and divided into two groups:
\begin{itemize}
 \item \textit{red points} are unoriented and unlabeled;
 \item \textit{blue points} are oriented and labeled by objects of $\calC$.
\end{itemize}
A $\calC$-labeled bichrome set $P \subset \Sigma$ is a \textit{red set} if its set of blue points is empty, and it is a \textit{$\calC$-labeled blue set} if its set of red points is empty. If $M$ is a \dmnsnl{3} cobordism with corners from $\Sigma$ to $\Sigma'$, and if $P \subset \Sigma$ and $P' \subset \Sigma'$ are $\calC$-labeled bichrome sets, then a \textit{generalized $\calC$-labeled bichrome graph $T \subset M$ from $P$ to $P'$} is a ribbon graph $T$ embedded inside $M$ that is disjoint from the vertical boundary of $M$, and that has edges and coupons of two kinds:
\begin{itemize}
 \item \textit{red edges} are unoriented and unlabeled, and they intersect the boundary of $M$ along the sets of red points of $P$ and $P'$;
 \item \textit{blue edges} are oriented and labeled by objects of $\calC$, and they intersect the boundary of $M$ along the sets of blue points of $P$ and $P'$ respecting orientations and labels;
 \item \textit{bichrome coupons} are unlabeled, and can only have the two following configurations
  \[
   \pic{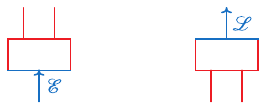}
  \]
 \item \textit{blue coupons} are labeled by morphisms of $\calC$ that are compatible with orientations and labels of edges, and can have arbitrary configurations.
\end{itemize}
A generalized $\calC$-labeled bichrome graph $T \subset M$ between $\calC$-labeled blue sets $P$ and $P'$ is called a \textit{$\calC$-labeled bichrome graph}. Here is an example of a $\calC$-labeled bichrome graph.
\[
 \pic{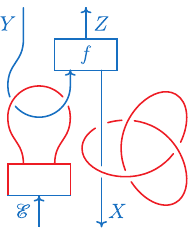}
\]
A $\calC$-labeled bichrome graph $T \subset M$ whose set of red edges is empty is called a \textit{$\calC$-labeled blue graph}, or simply a \textit{$\calC$-labeled ribbon graph}.

A $\calC$-labeled bichrome set $P \subset D^2$ is \textit{regular} if its points are distributed on the line $D^1 \times \{ 0 \} \subset D^2$, which is naturally ordered from left to right. We denote by $\calG_\calC$ the \textit{category of generalized $\calC$-labeled bichrome graphs}, whose objects are regular $\calC$-labeled bichrome sets $P \subset D^2$, and whose morphisms are isotopy classes of generalized $\calC$-labeled bichrome graphs $T \subset D^2 \times I$ relative to the boundary. We denote by $\calB_\calC$ the \textit{category of $\calC$-labeled bichrome graphs}, which is the full subcategory of $\calG_\calC$ whose objects are regular $\calC$-labeled blue sets $P \subset D^2$. Notice that $\calB_\calC$ is denoted $\calR_\Lambda$ in \cite[Section~3.1]{DGGPR19}. We denote by $\calR_\calC$ the \textit{category of $\calC$-labeled ribbon graphs}, which is the subcategory of $\calB_\calC$ whose objects coincide with those of $\calB_\calC$, and whose morphisms are isotopy classes of $\calC$-labeled ribbon graphs $T \subset D^2 \times I$ relative to the boundary.

\subsection{Cobordism \texorpdfstring{$2$-categories}{2-categories}}\label{S:cobordisms_with_corners}

We are now ready to define the symmetric mo\-noi\-dal \ctgr{2} $\bfACob_\calC$ of admissible $\calC$-decorated cobordisms, which will provide the source \ctgr{2} for our ETQFT. This represents a change with respect to \cite[Section~3]{D21}, where a stronger admissibility condition was imposed.

It is useful to start by recalling the definition of the
\textit{\ctgr{2} $\bfCob_\calC$ of \dcrtd{\calC}} (\textit{extended}\footnote{The term \textit{extended} for denoting surfaces equipped with Lagrangian subgroups of their first homology and cobordisms equipped with the signature of a \mnfld{4} goes back to Walker \cite[Section~1]{W91} and Turaev \cite[Chapter~IV, Section~6]{T94}.}) \textit{cobordisms}, which contains $\bfACob_\calC$ as a sub-\ctgr{2}, and which is defined as follows:
\begin{itemize}
 \item An object of $\bfCob_\calC$ is a closed \dmnsnl{1} manifold $\bfGamma = \Gamma$.
 \item A \mrphsm{1} of $\bfCob_\calC$ from $\bfGamma$ to $\bfGamma'$ is a triple $\bfSigma = (\Sigma,P,\Lagr)$ where:
  \begin{itemize}
   \item $\Sigma$ is a \dmnsnl{2} cobordism from $\Gamma$ to $\Gamma'$;
   \item $P \subset \Sigma$ is a $\calC$-labeled blue set;
   \item $\Lagr \subset H_1(\Sigma)$ is a Lagrangian subgroup.
  \end{itemize}
  We will refer to $\Sigma$ as the support of $\bfSigma$, and to $P \subset \Sigma$ as the \dcrtn{\calC} of $\bfSigma$. We recall that a Lagrangian subgroup of $H_1(\Sigma)$ is a maximal isotropic subgroup with respect to the transverse intersection form, which is non-degenerate if $\Gamma = \Gamma' = \varnothing$.
 \item A \mrphsm{2} of $\bfCob_\calC$ from $\bfSigma : \bfGamma \to \bfGamma'$ to $\bfSigma' : \bfGamma \to \bfGamma'$ is an equivalence class of triples $\bfM = (M,T,\sig)$, where:
  \begin{itemize}
   \item $M$ is a \dmnsnl{3} cobordism with corners from $\Sigma$ to $\Sigma'$;
   \item $T \subset M$ is a $\calC$-labeled bichrome graph from $P$ to $P'$;
   \item $\sig$ is an integer, called the signature.
  \end{itemize}
  Two triples $(M,T,\sig)$ and $(M',T',\sig')$ are defined to be equivalent if there exists an isomorphism of cobordisms with corners $f : M \to M'$ that satisfies $f(T) = T'$, and if $\sig = \sig'$. We will refer to $M$ as the support of $\bfM$, and to $T \subset M$ as the \dcrtn{\calC} of $\bfM$.
 \item The horizontal composition
  \[
   \bfSigma' \circ \bfSigma : \bfGamma \to \bfGamma''
  \]
  of \mrphsms{1} $\bfSigma : \bfGamma \to \bfGamma'$ and $\bfSigma' : \bfGamma' \to \bfGamma''$ of $\bfCob_\calC$ is
  \[
   \bfSigma' \circ \bfSigma := \left( \Sigma' \circ \Sigma,P' \cup P,(i_{\Sigma'})_*(\Lagr') + (i_\Sigma)_*(\Lagr) \right),
  \]
  where $\Sigma' \circ \Sigma$ denotes the \dmnsnl{2} cobordism from $\Gamma$ to $\Gamma''$ given by $\Sigma' \cup_{\Gamma'} \Sigma$, and where $i_\Sigma : \Sigma \hookrightarrow \Sigma' \cup_{\Gamma'} \Sigma$ and $i_{\Sigma'} : \Sigma' \hookrightarrow \Sigma' \cup_{\Gamma'} \Sigma$ denote inclusions.
 \item The horizontal composition
  \[
   \bfM' \circ \bfM : \bfSigma' \circ \bfSigma \Rightarrow \bfSigma''' \circ \bfSigma''
  \]
  of \mrphsms{2} $\bfM : \bfSigma \Rightarrow \bfSigma''$ and $\bfM' : \bfSigma' \Rightarrow \bfSigma'''$ of $\bfCob_\calC$ between \mrphsms{1} $\bfSigma,\bfSigma'' : \bfGamma \to \bfGamma'$ and $\bfSigma',\bfSigma''' : \bfGamma' \to \bfGamma''$ is (the equivalence class of)
  \[
   \bfM' \circ \bfM := \left( M' \circ M,T' \cup T, \sig' + \sig \right).
  \]
  where $M' \circ M$ denotes the \dmnsnl{3} cobordism with corners from $\Sigma' \circ \Sigma$ to $\Sigma''' \circ \Sigma''$ given by $M' \cup_{\Gamma' \times I} M$.
 \item The vertical composition
  \[
   \bfM' \ast \bfM : \bfSigma \Rightarrow \bfSigma''
  \]
  of \mrphsms{2} $\bfM : \bfSigma \Rightarrow \bfSigma'$ and $\bfM' : \bfSigma' \Rightarrow \bfSigma''$ of $\bfCob_\calC$ between \mrphsms{1} $\bfSigma,\bfSigma',\bfSigma'' : \bfGamma \to \bfGamma'$ is (the equivalence class of)
  \[
   \bfM' \ast \bfM := \left( M' \ast M,T' \ast T, \sig' + \sig - \mu(M_*(\Lagr),\Lagr',(M')^*(\Lagr'')) \right),
  \]
  where $M' \ast M$ denotes the \dmnsnl{3} cobordism with corners from $\Sigma$ to $\Sigma''$ given by $M' \cup_{\Sigma'} M$,
  where $T' \ast T$ denotes the $\calC$-labeled bichrome graph from $P$ to $P''$ given by $T' \cup_{P'} T$, where the push-forward $M_*(\Lagr)$ and the pull-back $(M')^*(\Lagr'')$ are
  \begin{align*}
   M_*(\Lagr) &:= \{ x' \in H_1(\Sigma') \mid (j_{\Sigma'})_*(x') \in (j_\Sigma)_* (\Lagr) \}, \\*
   (M')^*(\Lagr'') &:= \{ x' \in H_1(\Sigma') \mid (j'_{\Sigma'})_*(x') \in (j'_{\Sigma''})_*(\Lagr'') \}
  \end{align*}
  for the embeddings
  \[
   j_\Sigma : \Sigma \hookrightarrow M, \quad
   j_{\Sigma'} : \Sigma' \hookrightarrow M, \quad
   j'_{\Sigma'} : \Sigma' \hookrightarrow M', \quad
   j'_{\Sigma''} : \Sigma'' \hookrightarrow M'\phantom{,}
  \]
  induced by horizontal boundary identifications, and where the correction term $\mu(M_*(\Lagr),\Lagr',(M')^*(\Lagr''))$ denotes the Maslov index of the Lagrangian subspaces
  \[
   M_*(\Lagr) \otimes \R, \Lagr' \otimes \R,(M')^*(\Lagr'') \otimes \R \subset H_1(\Sigma') \otimes \R,
  \]
  see \cite[Section~C.3]{D17}.
 \item The identity
  \[
   \id_{\bfGamma} : \bfGamma \to \bfGamma
  \]
  of an object $\bfGamma$ of $\bfCob_\calC$ is
  \[
   \id_{\bfGamma} := \left( \id_\Gamma,\varnothing,H_1(I \times \Gamma) \right),
  \]
  where $\id_\Gamma$ denotes the \dmnsnl{2} cobordism from $\Gamma$ to $\Gamma$ given by $I \times \Gamma$.
 \item The identity
  \[
   \id_{\bfSigma} : \bfSigma \Rightarrow \bfSigma
  \]
  of a \mrphsm{1} $\bfSigma : \bfGamma \to \bfGamma'$ of $\bfCob_\calC$ is (the equivalence class of)
  \[
   \id_{\bfSigma} := \left( \id_\Sigma,\id_P,0 \right),
  \]
  where $\id_\Sigma$ denotes the \dmnsnl{3} cobordism with corners from $\Sigma$ to $\Sigma$ given by $\Sigma \times I$, and where $\id_P \subset \id_\Sigma$ denotes the $\calC$-labeled ribbon graph from $P$ to $P$ given by $P \times I$
\end{itemize}

The \ctgr{2} $\bfCob_\calC$ can be equipped with the following monoidal structure:
\begin{itemize}
 \item The tensor product $\bfGamma \disjun \bfGamma'$ of objects $\bfGamma$ and $\bfGamma'$ of $\bfCob_\calC$ is
  \[
   \bfGamma \disjun \bfGamma' := \Gamma \sqcup \Gamma'.
  \]
 \item The tensor product
  \[
   \bfSigma \disjun \bfSigma' : \bfGamma \disjun \bfGamma' \to \bfGamma'' \disjun \bfGamma'''
  \]
  of \mrphsms{1} $\bfSigma : \bfGamma \to \bfGamma''$ and $\bfSigma' : \bfGamma' \to \bfGamma'''$ of $\bfCob_\calC$ is
  \[
   \bfSigma \disjun \bfSigma' := (\Sigma \sqcup \Sigma',P \sqcup P',\Lagr \oplus \Lagr').
  \]
 \item The tensor product
  \[
   \bfM \disjun \bfM' : \bfSigma \disjun \bfSigma' \Rightarrow \bfSigma'' \disjun \bfSigma'''
  \]
  of \mrphsms{2} $\bfM : \bfSigma \Rightarrow \bfSigma''$ and $\bfM' : \bfSigma' \Rightarrow \bfSigma'''$ of $\bfCob_\calC$ is (the equivalence class of)
  \[
   \bfM \disjun \bfM' := (M \sqcup M',T \sqcup T',\sig + \sig').
  \]
 \item The tensor unit $\varnothing$ of $\bfCob_\calC$ is the empty object.
\end{itemize}

Furthermore, the monoidal \ctgr{2} $\bfCob_\calC$ can be equipped with the following symmetric braiding:
\begin{itemize}
 \item The braiding
  \[
   \bftau_{\bfGamma,\bfGamma'} : \bfGamma \disjun \bfGamma' \to \bfGamma' \disjun \bfGamma
  \]
  of objects $\bfGamma$ and $\bfGamma'$ of $\bfCob_\calC$ is
  \[
   \bftau_{\bfGamma,\bfGamma'} := (\tau_{\Gamma,\Gamma'} \cdot \id_{\Gamma \sqcup \Gamma'},\varnothing,H_1(I \times (\Gamma \sqcup \Gamma'))),
  \]
  where $\tau_{\Gamma,\Gamma'} : \Gamma \sqcup \Gamma' \to \Gamma' \sqcup \Gamma$ denotes the diffeomorphism induced by the transposition, see Equation~\eqref{E:left_action_1-morphism} for the left action of a diffeomorphism on a morphism of $\bfCob_\calC$, and compare with \cite[Definition~B.9]{D17}.
 \item The braiding
  \[
   \bftau_{\bfSigma,\bfSigma'} :
   \bftau_{\bfGamma'',\bfGamma'''} \circ (\bfSigma \disjun \bfSigma')
   \Rightarrow
   (\bfSigma' \disjun \bfSigma) \circ \bftau_{\bfGamma,\bfGamma'}
  \]
  of \mrphsms{1} $\bfSigma : \bfGamma \to \bfGamma''$ and $\bfSigma' : \bfGamma' \to \bfGamma'''$ of $\bfCob_\calC$ is (the equivalence class of)
  \[
   \bftau_{\bfSigma,\bfSigma'} := (\tau_{\Sigma,\Sigma'} \cdot (\tau_{\Gamma,\Gamma'} \cdot (\Sigma \sqcup \Sigma')),\id_{P \sqcup P'},0),
  \]
  where $\tau_{\Sigma,\Sigma'} : \tau_{\Gamma,\Gamma'} \cdot (\Sigma \sqcup \Sigma') \to (\Sigma' \sqcup \Sigma) \cdot \tau_{\Gamma,\Gamma'}$ denotes the compatible diffeomorphism induced by the transposition, see Equations~\eqref{E:left_action_1-morphism}--\eqref{E:left_action_2-morphism} for the left and right actions of a diffeomorphism on a morphism of $\bfCob_\calC$, and compare with \cite[Definition~B.10]{D17}.
\end{itemize}

The \textit{symmetric monoidal \ctgr{2} $\bfCob$ of} (\textit{extended}) \textit{cobordisms} is the symmetric mo\-noi\-dal sub-\ctgr{2} of $\bfCob_\calC$ whose objects coincide with those of $\bfCob_\calC$, and whose morphisms have empty \dcrtns{\calC}. We adopt the notation $(\Sigma,\Lagr) := (\Sigma,\varnothing,\Lagr)$ and $(M,\sig) := (M,\varnothing,\sig)$ for morphisms of $\bfCob$.

Let $\bfSigma = (\Sigma,P,\Lagr)$ be a \mrphsm{1} of $\bfCob_\calC$. We say $\bfSigma$ is \textit{admissible} if every connected component of the support $\Sigma$ disjoint from the outgoing boundary contains a $\calC$-labeled blue subset of the \dcrtn{\calC} $P$ that admits a label in $\Proj(\calC)$. We say $\bfSigma$ is \textit{coadmissible} if its dual \mrphsm{1} $\bfSigma^*$, which is defined by Equation~\eqref{E:dual_1-morphism}, is admissible. We say $\bfSigma$ is \textit{biadmissible} if it is simultaneously admissible and coadmissible. Notice that horizontal compositions of (co)admissible \mrphsms{1} are (co)admissible, and that identities of objects are biadmissible.

Let $\bfM = (M,T,\sig)$ be a \mrphsm{2} of $\bfCob_\calC$. We say $\bfM$ is \textit{admissible} if every connected component of the support $M$ disjoint from the outgoing horizontal boundary contains a $\calC$-labeled bichrome subgraph of the \dcrtn{\calC} $T$ that admits a label in $\Proj(\calC)$. We say $\bfM$ is \textit{coadmissible} if its dual \mrphsm{2} $\bfM^*$, which is defined by Equation~\eqref{E:dual_2-morphism}, is admissible. We say $\bfM$ is \textit{biadmissible} if it is simultaneously admissible and coadmissible. Notice that horizontal and vertical compositions of (co)admissible \mrphsms{2} are (co)admissible, and that identities of \mrphsms{1} are biadmissible.

\begin{definition}\label{D:admissible_cobordisms_2-cat}
 The \textit{symmetric monoidal \ctgr{2} $\bfACob_\calC$ of admissible \dcrtd{\calC}} (\textit{extended}) \textit{cobordisms} is the symmetric monoidal sub-\ctgr{2} of $\bfCob_\calC$ whose objects and \mrphsms{1} coincide with those of $\bfCob_\calC$, and whose \mrphsms{2} are admissible.
\end{definition}

Notice that $\bfACob_\calC$ is obviously very closely related to, but slightly different from, the \ctgr{2} $\bfadCob_\calC$ of \cite[Definition~3.1]{D21}. First of all, the convention about the admissibility condition is reversed, as it applies here to connected components that are disjoint from the outgoing boundary, rather than the incoming one. In other words, \mrphsms{2} of $\bfadCob_\calC$ are coadmissible \mrphsms{2} of $\bfCob_\calC$, in our current terminology. More importantly, though, the admissibility condition is dropped for \mrphsms{1}, here. In other words, for all objects $\bfGamma$ and $\bfGamma'$ of $\bfACob_\calC$, the morphism category $\bfadCob_\calC(\bfGamma,\bfGamma')$ can be naturally identified with a proper full subcategory of $\bfACob_\calC(\bfGamma',\bfGamma)^\op$. In this sense, the construction of this paper will extend the one of \cite{D21}.

A \mrphsm{1} $\bfSigma = (\Sigma,\Lagr)$ of $\bfCob$ is said to be \textit{non-compact} if no connected component of the support $\Sigma$ is disjoint from the outgoing boundary. Notice that horizontal compositions of non-compact \mrphsms{1} are non-compact, and that identities of objects are non-compact.

A \mrphsm{2} $\bfM = (M,\sig)$ of $\bfCob$ is said to be \textit{non-compact}\footnote{The term \textit{non-compact} for denoting cobordisms without connected components disjoint from the outgoing (horizontal) boundary goes back to Lurie \cite[Definition~4.2.10]{L09}. Notice that supports of non-compact cobordisms are still given by compact manifolds (with corners).} if no connected component of the support $M$ is disjoint from the outgoing horizontal boundary. Notice that horizontal and vertical compositions of non-compact \mrphsms{2} are non-compact, and that identities of \mrphsms{1} are non-compact.

\begin{definition}\label{D:non-compact_cobordisms_2-cat}
 The \textit{symmetric monoidal \ctgr{2} $\bfNCob$ of non-compact} (\textit{extended}) \textit{cobordisms} is the symmetric monoidal sub-\ctgr{2} of $\bfCob$ whose objects and \mrphsms{1} coincide with those of $\bfCob$, and whose \mrphsms{2} are non-compact.
\end{definition}

Notice that $\bfNCob = \bfACob_\calC \cap \bfCob$.

\subsection{Atiyah cobordism categories}\label{S:Atiyah_cobordisms}

Next, let us recall the definition of the symmetric monoidal category $\ACob_\calC$ of admissible $\calC$-decorated cobordisms, which provides the source category for the renormalized \KL{} TQFT. As we will see, the absence of the admissibility condition for \mrphsms{1} makes it possible to realize $\ACob_\calC$ as a full symmetric monoidal category of endomorphisms of $\bfACob_\calC$.

The \textit{symmetric monoidal category $\Cob_\calC$ of \dcrtd{\calC}} (\textit{extended}) \textit{cobordisms} is the category $\bfCob_\calC(\varnothing,\varnothing)$ of morphisms $\bfCob_\calC$ from $\varnothing$ to $\varnothing$.

The \textit{symmetric monoidal category $\Cob$ of} (\textit{extended}) \textit{cobordisms} is the category $\bfCob(\varnothing,\varnothing)$ of morphisms $\bfCob$ from $\varnothing$ to $\varnothing$.

\begin{definition}\label{D:admissible_cobordisms_cat}
The \textit{symmetric monoidal category $\ACob_\calC$ of admissible \dcrtd{\calC}} (\textit{extended}) \textit{cobordisms} is the subcategory of $\Cob_\calC$ whose objects coincide with those of $\Cob_\calC$, and whose morphisms are admissible.
\end{definition}

Notice that $\ACob_\calC = \bfACob_\calC(\varnothing,\varnothing)$.

\begin{definition}\label{D:non-compact_cobordisms_cat}
The \textit{symmetric monoidal category $\NCob$ of non-compact} (\textit{extended}) \textit{cobordisms} is the subcategory of $\Cob$ whose objects coincide with those of $\Cob$, and whose morphisms are non-compact.
\end{definition}

Notice that $\NCob = \ACob_\calC \cap \Cob$.

For all the symmetric monoidal categories introduced above, tensor product, tensor unit, and braiding are inherited from those of $\bfCob_\calC$.

\begin{remark}\label{R:end_vs_coend_4}
 The category of admissible cobordisms $\adCob_\calC$ defined in \cite[Section~4.3]{DGGPR19} actually uses the opposite convention for admissibility. In other words, projective labels are required for connected components disjoint from the incoming boundary, rather than the outgoing one. Of course, each of these categories is simply the opposite of the other.
\end{remark}

\subsection{Crane--Yetter--Kerler cobordism categories}\label{S:connected_cobordisms}

We move on to the definition of the braided monoidal category $\CCob$ of connected cobordisms, which provides the source category for the \KL{} TQFT. As we will see, $\CCob$ can be understood as a proper sub-category of morphisms of $\bfNCob$, although not as a monoidal one, since tensor products do not agree.

Let $\bfS^1$ be the object of $\bfCob$ given by $S^1 \subset \R^2$, and let $\bfD^2 : \varnothing \to \bfS^1$ be the \mrphsm{1} of $\bfCob$ defined as
\begin{equation}\label{E:disc}
 \bfD^2 := (D^2,\{ 0 \}),
\end{equation}
where $D^2$ denotes the \dmnsnl{2} cobordism from $\varnothing$ to $S^1$ whose support is given by $D^2 \subset \R^2$.

Let $\bfP^2 : \bfS^1 \disjun \bfS^1 \to \bfS^1$ be the \mrphsm{1} of $\bfCob$ defined as
\begin{equation}\label{E:pant}
 \bfP^2 := (P^2,H_1(P^2)),
\end{equation}
where $P^2$ denotes the \dmnsnl{2} cobordism from $S^1 \sqcup S^1$ to $S^1$ whose support is given by $D^2 \subset \R^2$ minus two open disks of radius $\frac{1}{4}$ and center $(-\frac{1}{2},0)$ and $(\frac{1}{2},0)$ respectively, and whose incoming and outgoing horizontal boundary identifications are induced by $\id_{S^1}$ through rescaling and translation.

Let $\bfP^2 \cdot \tau_{\bfS^1,\bfS^1} : \bfS^1 \disjun \bfS^1 \to \bfS^1$ be the \mrphsm{1} of $\bfCob$ determined by the right action of $\tau_{\bfS^1,\bfS^1}$ on $\bfP^2$, as defined by Equation~\eqref{E:left_action_1-morphism}. Using the terminology introduced in Appendix~\ref{A:dualities_adjunctions}, let $h : \bfP^2 \to \bfP^2 \cdot \tau_{\bfS^1,\bfS^1}$ denote the compatible diffeomorphism given by the composition of the two left-handed half-twists along the two connected components of the incoming boundary of $P^2$ with the right-handed half-twist along the outgoing boundary of $P^2$, represented by the following picture.
\[
 \pic{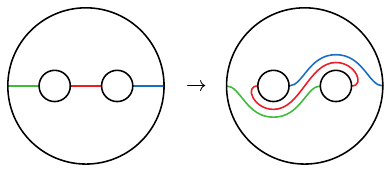}
\]
Let $h \cdot \id_{\bfP^2} : \bfP^2 \to \bfP^2 \cdot \tau_{\bfS^1,\bfS^1}$ denote the \mrphsm{2} of $\bfCob$ determined by the left action of $h$ on $\id_{\bfP^2} \subset \R^3$, as defined by Equation~\eqref{E:left_action_2-morphism}. We represent $h \cdot \id_{\bfP^2}$ graphically through the projection to $\R \times \{ 0 \} \times \R$ as
\[
 \pic{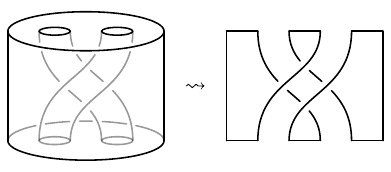}
\]

The category $\bfCob(\varnothing,\bfS^1)$ can be given a braided monoidal structure, with tensor product $\bcs : \bfCob(\varnothing,\bfS^1) \times \bfCob(\varnothing,\bfS^1) \to \bfCob(\varnothing,\bfS^1)$
defined as
\[
 \bcs := \bfP^2_* \circ \disjun = \bfP^2 \circ (\_ \disjun \_),
\]
with tensor unit $\bfD^2 \in \bfCob(\varnothing,\bfS^1)$, and with braiding $\bfh_{\bfSigma,\bfSigma'} : \bfSigma \bcs \bfSigma' \to \bfSigma' \bcs \bfSigma$ defined as
\[
 \bfh_{\bfSigma,\bfSigma'} := (h \cdot \id_{\bfP^2}) \triangleleft (\bfSigma \disjun \bfSigma')
\]
for all objects $\bfSigma$ and $\bfSigma'$ of $\bfCob(\varnothing,\bfS^1)$.

\begin{definition}\label{D:connected_cobordisms_cat}
 The \textit{braided monoidal category $\CCob$ of} (\textit{extended relative}) \textit{connected cobordisms} is obtained from the subcategory of $\bfCob(\varnothing,\bfS^1)$ whose objects and morphisms have connected support by specifying the boundary connected sum $\bcs$ as a tensor product, the disc $\bfD^2$ as a tensor unit, and the mapping cylinder $\bfh$ as a braiding.
\end{definition}

Notice that $\CCob$ is also a subcategory of $\bfNCob(\varnothing,\bfS^1)$, since every morphism of $\CCob$ is non-compact. It was first introduced and studied by Crane and Yetter in \cite{CY94} and by Kerler in \cite{K00}.

\subsection{Kirby tangles in handlebodies}

In order to explain how to represent connected cobordisms diagrammatically, let us recall the definition of the category $\KTan$ of \textit{Kirby tangles in handlebodies}. This is a generalization of the category of bottom tangles in handlebodies introduced by Habiro in \cite{H05}.

For every non-negative integer $g \geqs 0$, we specify a connected \dmnsnl{3} handlebody $M_g \subset \R^3$ of genus $g$, obtained by attaching $g$ copies $B_1, \ldots, B_g$ of the \dmnsnl{3} \hndl{1} $B = D^1 \times D^2$ to the bottom face $D^2 \times \{ 0 \}$ of the cylinder $D^2 \times I$. We represent graphically $M_g$ through the projection to $\R \times \{ 0 \} \times \R$ as
\begin{align*}
 \pic{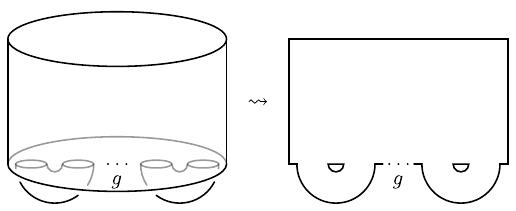}
\end{align*}
We denote by $\Sigma_g$ the connected surface of genus $g$ with one boundary component appearing as the bottom face of $M_g$, which is obtained from the disc $D^2 \times \{ 0 \}$ by performing $1$-surgery along $g$ pairs of discs. Notice that $M_g$ can be interpreted as a \dmnsnl{3} cobordism with corners from $\Sigma_g$ to $\Sigma_0$.

A \textit{Kirby $g'$-tangle in a $g$-handlebody}, sometimes simply called a \textit{Kirby tangle}, is an oriented framed tangle $T$ inside the connected handlebody $M_g$ whose boundary is composed of $2g'$ points distributed on the top line $D^1 \times \{ 0 \} \times \{ 1 \} \subset D^2 \times \{ 1 \} \subset M_g$ and such that, for every $1 \leqs j \leqs g'$, one of its components joins the $(2j)$th boundary point to the $(2j-1)$th one. If $T$ is a Kirby tangle, we denote by $A(T)$ the collection of its arc components, and by $C(T)$ the collection of its closed components, so that $T = A(T) \cup C(T)$. A \textit{top $g'$-tangle in a $g$-handlebody}, sometimes simply called a \textit{top tangle}, is a Kirby tangle $T$ such that $C(T) = \varnothing$. In particular, if $T$ is a Kirby tangle, then $A(T)$ is a top tangle, and $C(T)$ is a framed link.

Here is an example of a Kirby tangle, together with its projection:
\begin{align*}
 \pic{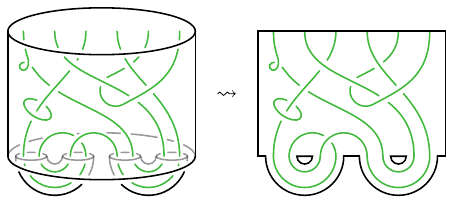}
\end{align*}
Notice that, up to isotopy, we can always represent Kirby tangles using \textit{regular diagrams}, that are diagrams in which every handle $D^1 \times D^2$ of $M_g$ intersects the tangle $T$ in $D^1 \times P \subset D^1 \times D^2$, where $P \subset D^1$ is a finite set of points that remain distinct under the projection.

We consider Kirby tangles up to \textit{signature-preserving Kirby moves} of the following type:
\begin{gather*}
 \pic{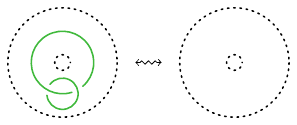} \tag{K$1$}\label{E:K1} \\*
 \pic{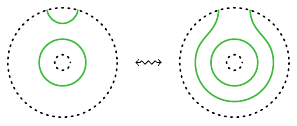} \tag{K$2$}\label{E:K2}
\end{gather*}
These operations can be performed inside any solid torus $S^1 \times D^2$ arbitrarily embedded into $M_g$ (represented here as the annuli contained between the dashed circles on both sides of the operations). Move \eqref{E:K1} is also known as the \textit{slam-dunk}, while move \eqref{E:K2} is also known as the \textit{slide}.

The \textit{category $\KTan$ of Kirby tangles in handlebodies} is the category defined as follows:
\begin{itemize}
 \item An object of $\KTan$ is a natural number $g \geqs 0$.
 \item A morphism of $\KTan$ from $g$ to $g'$ is an equivalence class of Kirby $g'$-tangles $T$ in the $g$-handlebody $M_g$ up to signature-preserving Kirby moves.
 \item The composition
  \[
   T' \ast T : g \to g''
  \]
  of morphisms $T : g \to g'$ and $T' : g' \to g''$ of $\KTan$ is the equivalence class of the Kirby tangle $f(T \cup T')$ obtained by carving out an open tubular neighborhood $N(A(T))$ in $M_g$ of the top tangle $A(T) \subset T$, by gluing $M_g \smallsetminus N(A(T))$ to $M_{g'}$, identifying the top base of $M_g \smallsetminus N(A(T))$ with the bottom base of $M_{g'}$ as prescribed by the top tangle $A(T)$, and by considering a diffeomorphism $f : M_{g'} \cup_{\Sigma_{g'}} (M_g \smallsetminus N(A(T))) \to M_g$ given by vertical rescaling. Here is an example of a composition of Kirby tangles:
  \[
   \pic{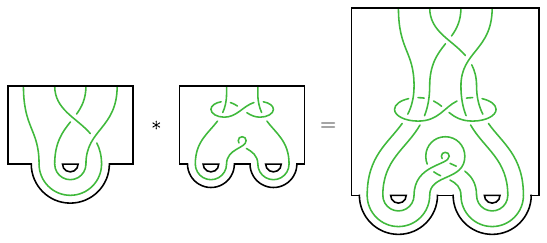}
  \]
 \item The identity $\id_g : g \to g$ of an object $g$ of $\KTan$ is given by
  \[
   \id_g := \pic{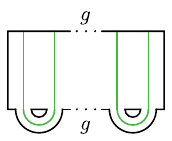}
  \]
 \item The tensor product $g \bcs g'$ of objects $g$ and $g'$ of $\KTan$ is $g+g'$.
 \item The tensor product
  \[
   T \bcs T' : g \bcs g' \to g'' \bcs g'''
  \]
  of morphisms $T : g \to g''$ and $T' : g' \to g'''$ of $\KTan$ is the equivalence class of the Kirby tangle $f(T \cup T')$ obtained by gluing $M_g \sqcup M_{g'}$ to $P^2 \times I$, identifying the vertical boundary of $M_g \sqcup M_{g'}$ with the innermost vertical boundary of $P^2 \times I$ as prescribed by the identity map, and by considering a diffeomorphism $f : (P^2 \times I) \cup_{(S^1 \sqcup S^1) \times I} (M_g \sqcup M_{g'}) \to M_{g+g'}$ given by horizontal rescaling. Here is an example of a tensor product of Kirby tangles:
  \[
   \pic{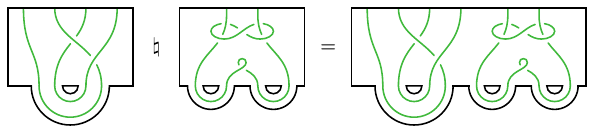}
  \]
 \item The tensor unit $0$ of $\KTan$ is the natural number $0$.
 \item The braiding
  \[
   c_{g,g'} : g \bcs g' \to g' \bcs g
  \]
  of objects $g$ and $g'$ of $\KTan$ is given by
  \[
   c_{g,g'} := \pic{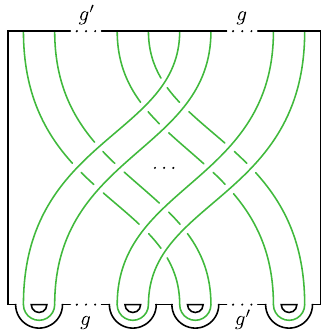}
  \]
\end{itemize}

The \textit{category $\TTan$ of top tangles in handlebodies} is the subcategory of $\KTan$ whose objects coincide with those of $\KTan$, and whose morphisms are equivalence classes of top tangles in handlebodies.

An equivalence between $\KTan$ and $\CCob$ is provided by the braided monoidal functor
\[
 \chi : \KTan \to \CCob
\]
sending every object $g$ of $\KTan$ to the object $\bfSigma_g$ of $\CCob$ given by
\begin{equation}\label{E:surgery_functor_objects}
 \bfSigma_g := (\Sigma_g,\Lagr_g),
\end{equation}
where $\Lagr_g \subset H_1(\Sigma_g)$ is the Lagrangian subgroup given by $\ker \iota_g^*$ for the inclusion $\iota_g : \Sigma_g \hookrightarrow M_g$, and sending every morphism $T : g \to g'$ of $\KTan$ to the morphism $\bfM(T) : \bfSigma_g \to \bfSigma_{g'}$ of $\CCob$ given by
\begin{equation}\label{E:surgery_functor_morphisms}
 \bfM(T) := (M(T),\sigma(C(T))),
\end{equation}
where:
\begin{itemize}
 \item $M(T)$ is the \dmnsnl{3} cobordism with corners obtained from $M_g$ by carving out an open tubular neighborhood $N(A(T))$ in $M_g$ of the top tangle $A(T) \subset T$, and by performing $2$-surgery along the framed link $C(T) \subset T$;
 \item $\sigma(C(T))$ is the signature of the linking matrix of $C(T) \subset M_g \subset \R^3$.
\end{itemize}
See \cite[Proposition~4.2]{BD21} for a proof that $\chi : \KTan \to \CCob$ is a braided monoidal equivalence. Notice that, under the functor $\chi : \KTan \to \CCob$, the subcategory $\TTan \subset \KTan$ of top tangles in handlebodies is sent to the subcategory $\SCob \subset \CCob$ of \textit{special Lagrangian cobordisms}, in the terminology of \cite[Definition~2.6]{CHM07}.

\subsection{BPHK Hopf algebras}\label{S:BPHK_Hopf_algebras}

We finish this section by recalling the Hopf algebra structure supported by the genus $1$ surface $\bfSigma_1$ inside the category $\CCob$ of connected cobordisms. The existence of such an algebraic structure can be seen as a deep, intrinsically topological explanation of the prominence of Hopf algebras in quantum topology. In fact, although we will not actually need it, we will also recall a complete algebraic presentation of $\CCob$, which pinpoints the exact algebraic structure that governs the topology of connected cobordisms.

First, let us fix a braided monoidal category $\calC$, and let us recall the notion of a Hopf algebra in $\calC$.

\begin{definition}\label{D:Hopf_algebra}
 A \textit{Hopf algebra} in $\calC$ is an object $\calH \in \calC$, together with
 \begin{itemize}
  \item a \textit{product} $\mu \in \calC(\calH \otimes \calH,\calH)$ and a \textit{unit} $\eta \in \calC(\one,\calH)$, represented as
   \[
    \pic{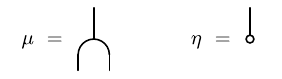}
   \]
  \item a \textit{coproduct} $\Delta \in \calC(\calH,\calH \otimes \calH)$ and a \textit{counit} $\varepsilon \in \calC(\calH,\one)$, represented as
   \[
    \pic{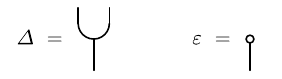}
   \]
  \item an \textit{antipode} and an \textit{inverse antipode} $S,S^{-1} \in \calC(\calH,\calH)$, represented as
   \[
    \pic{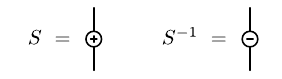}
   \]
 \end{itemize}
 satisfying the following axioms:
 \begin{itemize}
  \item $\mu$ and $\eta$ provide an algebra structure
   \begin{equation}\label{E:Hopf_algebra_axiom}
    \pic{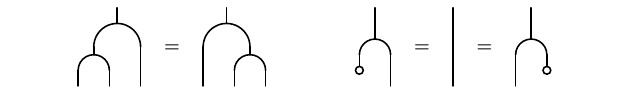}
   \end{equation}
  \item $\Delta$ and $\varepsilon$ provide a coalgebra structure
   \begin{equation}\label{E:Hopf_coalgebra_axiom}
    \pic{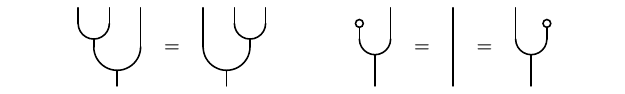}
   \end{equation}
  \item $\mu$, $\eta$, $\Delta$, and $\varepsilon$ provide a bialgebra structure
   \begin{equation}\label{E:Hopf_bialgebra_axiom}
    \pic{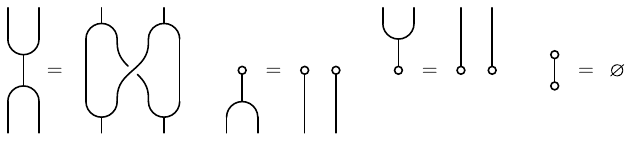}
   \end{equation}
  \item $S$ is the convolution-inverse of the identity and it is invertible
   \begin{equation}\label{E:Hopf_antipode_axiom}
    \pic{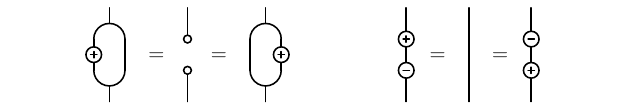}
   \end{equation}
 \end{itemize}
\end{definition}

We will abusively denote Hopf algebras in $\calC$ simply by their underlying objects $\calH$, without explicit mention to their structure morphisms. Notice that Hopf algebras are not always required to have an invertible antipode in the literature. However, the antipode of a Hopf algebra in the category $\FVect_\Bbbk$ of finite-dimensional vector spaces is always automatically invertible, see \cite[Theorem~7.1.14]{R12}.

It was first observed by Crane and Yetter \cite{CY94} that $\bfSigma_1 \in \CCob$ is a Hopf algebra. Its structure morphisms are represented, as morphisms of $\KTan$, by
\begin{align*}
 &\pic{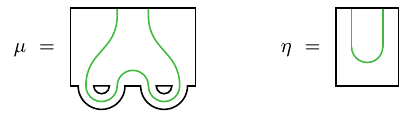} \\*
 &\pic{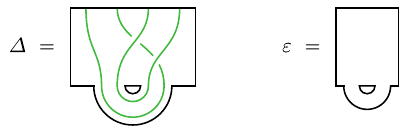} \\*
 &\pic{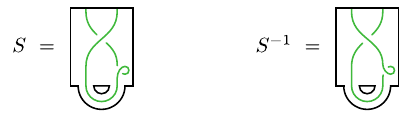}
\end{align*}
It was later noticed by Habiro \cite{A11} that the adjoint action of $\bfSigma_1 \in \CCob$ on itself is \textit{braided cocommutative}, in the sense that it satisfies
\begin{equation}\label{E:Hopf_braided_cocommutative_axiom}
 \pic{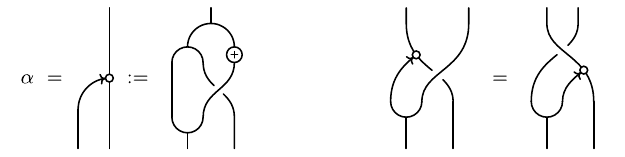}
\end{equation}
The problem of finding a complete algebraic presentation of the category $\CCob$ led to the following notion.

\begin{definition}\label{D:BPHK_Hopf_algebra}
 A \textit{BPHK Hopf algebra} in $\calC$ (short for \textit{Bobtcheva--Piergallini--Habiro--Kerler}) is a Hopf algebra $\calH \in \calC$ whose adjoint action is braided cocommutative, equipped with
 \begin{itemize}
  \item a \textit{ribbon} and an \textit{inverse ribbon element} $v_+,v_- \in \calC(\one,\calH)$, represented as
   \[
    \pic{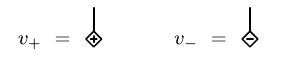}
   \]
  \item an \textit{integral form} $\lambda \in \calC(\calH,\one)$, represented as
   \[
    \pic{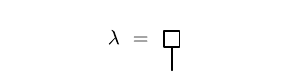}
   \]
 \end{itemize}
 satisfying the following axioms:
 \begin{itemize}
  \item $v_+$ is central, invertible, and antipode-invariant
   \begin{equation}\label{E:Hopf_ribbon_axiom}
    \pic{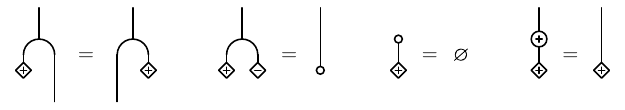}
   \end{equation}
  \item $\Delta$, $v_+$, and $v_-$ induce a Hopf copairing $w$
   \begin{equation}\label{E:Hopf_copairing_axiom}
    \pic{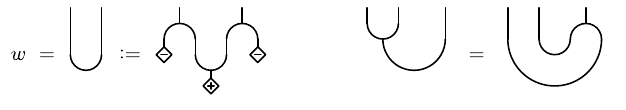}
   \end{equation}
  \item $\lambda$ is a coinvariant form which is antipode-invariant
   \begin{equation}\label{E:Hopf_integral_axiom}
    \pic{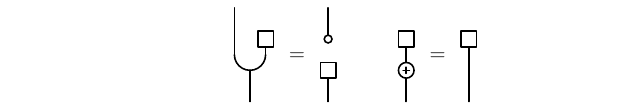}
   \end{equation}
  \item $w$ and $\lambda$ induce an invariant element $\Lambda$ which is antipode-invariant
   \begin{equation}\label{E:Hopf_cointegral_axiom}
    \pic{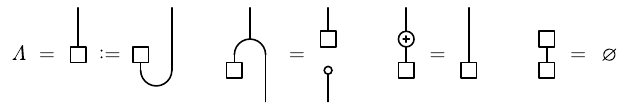}
   \end{equation}
 \end{itemize}
\end{definition}

Notice that the ribbon element of a ribbon Hopf algebra in $\FVect_\Bbbk$ does not satisfy the axioms of a ribbon element of a BPHK Hopf algebra in general, since it might not induce a Hopf copairing, as required by Equation~\eqref{E:Hopf_copairing_axiom}. Similarly, neither the left nor the right integral form of a unimodular Hopf algebra in $\FVect_\Bbbk$ need satisfy the axioms of an integral form of a BPHK Hopf algebra in general, since they might not be antipode-invariant, as required by Equation~\eqref{E:Hopf_integral_axiom}.

As explained in \cite{A11}, it was Habiro who first observed that $\bfSigma_1 \in \CCob$ is a BPHK Hopf algebra, based on the work of Kerler \cite{K00}. Its remaining structure morphisms are represented, as morphisms of $\KTan$, by
\begin{align*}
 &\pic{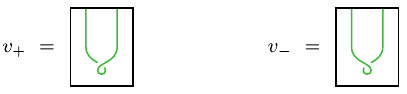} \\*
 &\pic{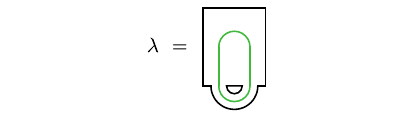}
\end{align*}
The following result was announced by Habiro, as stated in \cite[Theorem~2.3]{A11}, and independently proved by Bobtcheva and Piergallini in \cite[Theorem~5.5.4]{BP11}, see \cite[Theorem~2.6.11]{BBDP23} for the reformulation reported here.

\begin{theorem}\label{T:BP}
 $\CCob$ is equivalent to the free braided mo\-noi\-dal category generated by the BPHK Hopf algebra $\bfSigma_1$.
\end{theorem}

In other words, the list of axioms of a BPHK Hopf algebra, given by Equations~\eqref{E:Hopf_algebra_axiom}--\eqref{E:Hopf_cointegral_axiom}, yields a complete algebraic presentation of the category $\CCob$.

\section{Topological Quantum Field Theories}\label{S:TQFTs}

In this section, we revisit the construction of the \KL{} TQFT $J_\eend : \CCob \to \calC$ for connected cobordisms, of the \KLRT{} functor $F_\eend : \calB_\calC \to \calC$ for $\calC$-labeled bichrome graphs, and of the renormalized \KL{} TQFT $V_\eend : \ACob_\calC \to \FVect_\Bbbk$ for admissible $\calC$-decorated cobordisms. It is useful to start by recalling that every ribbon category induces a ribbon functor $F_\calC : \calR_\calC \to \calC$, called the \textit{\RT{} functor}, on the category $\calR_\calC$ of $\calC$-labeled ribbon graphs, as defined in \cite[Chapter~I, Theorem~2.5]{T94}. If $P$ and $P'$ are objects of $\calR_\calC$, then we say two linear combinations $T$ and $T'$ of morphisms in $\calR_\calC$ from $P$ to $P'$ are \textit{skein equivalent in $D^2 \times I$} if
\[
 F_\calC(T) = F_\calC(T').
\]
More generally, if $M$ is a \dmnsnl{3} cobordism with corners from $\Sigma$ to $\Sigma'$, and if $P \subset \Sigma$ and $P' \subset \Sigma'$ are $\calC$-labeled blue sets, then we say two linear combinations $T$ and $T'$ of $\calC$-labeled ribbon graphs in $M$ from $P$ to $P'$ are \textit{skein equivalent in $M$} if they are related by a finite sequence of skein equivalences in embedded copies of $D^2 \times I$. In this case, we write
\[
 T \doteq T'.
\]

\subsection{Kerler--Lyubashenko TQFT}\label{S:KL_TQFTs}

First, let us consider a modular category $\calC$ with adjoint end $\eend$, and let us recall the construction of a braided monoidal functor $J_\eend : \CCob \to \calC$
known as the \textit{\KL{} TQFT associated with $\eend$}, which is essentially a byproduct of the main result of \cite{KL01}. In order to build it, one possibility is to show that, if $\calC$ is modular, then $\eend$ is a BPHK Hopf algebra, see \cite[Theorem~7.4]{BD21}. However, the construction of $J_\eend : \CCob \to \calC$ does not actually require such a strong result as Theorem~\ref{T:BP}. Let us recall here below how the definition of the \KL{} TQFT can be given in terms of the universal property satisfied by $\eend$. We will do this in two steps: first, we will define a braided monoidal functor $J_\eend : \TTan \to \calC$ for an arbitrary finite ribbon category $\calC$,
and then, we will explain how to extend the definition to a braided monoidal functor $J_\eend : \KTan \to \calC$ when $\calC$ is modular, in Proposition~\ref{P:KL_TQFT}.

Let us start from top tangles. If $T : g \to g'$ is a top tangle, let $\myuline{T}$ denote the portion of (a regular diagram of) $T$ that is contained in the cylinder $D^2 \times I \subset M_g$, as represented by the following picture:
\[
 \pic{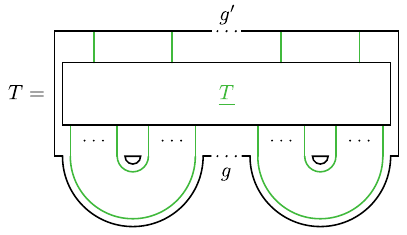}
\]
Notice that $\myuline{T}$ depends on the specific regular diagram that has been chosen to represent $T$. For all $X_1,\ldots,X_{g'} \in \calC$, let $T_{X_1,\ldots,X_{g'}}$ be obtained from $T$ by orienting its $k$th component from right to left and by labeling it $X_k$, and let $\myuline{T}_{X_1,\ldots,X_{g'}}$ denote the $\calC$-labeled framed tangle corresponding to the portion of $T_{X_1,\ldots,X_{g'}}$ that is contained in the cylinder $D^2 \times I \subset M_g$, as represented by the following picture:
\[
 \pic{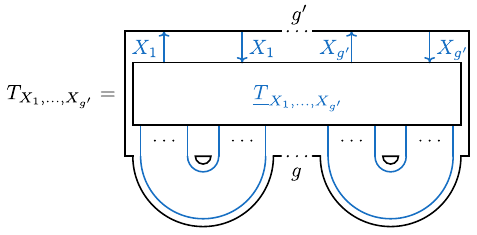}
\]
This determines objects $S_j := S_j(\myuline{T}_{X_1,\ldots  ,X_{g'}}) \in \calC$ such that
\[
 F_\calC(\myuline{T}_{X_1,\ldots,X_{g'}}) \in \calC \left( \bigotimes_{j=1}^g (S_j \otimes S_j^*), \bigotimes_{j=1}^{g'} (X_j \otimes X_j^*) \right).
\]
Notice that, up to skein equivalence, we can assume that the $j$th handle of $M_g$ intersects a single strand of the $\calC$-labeled ribbon graph $T_{X_1,\ldots,X_{g'}}$, since
\[
 \pic{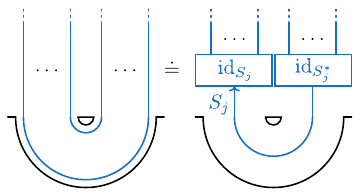}
\]
If we replace the $j$th handle of $M_g$ by a coupon labeled by the dinatural structure morphism $i_{S_j} : \eend \to S_j \otimes S_j^*$ of the adjoint end $\eend$, then, by applying the \RT{} functor $F_\calC$, we obtain a morphism of $\calC$ that is dinatural in $X_1,\ldots,X_{g'} \in \calC$. Therefore, thanks to the universal property satisfied by $\eend$, there exists a unique morphism $J_\eend(T) : \eend^{\otimes g} \to \eend^{\otimes g'}$ satisfying
\begin{equation}\label{E:KL_TQFT_top_tangle}
 \pic{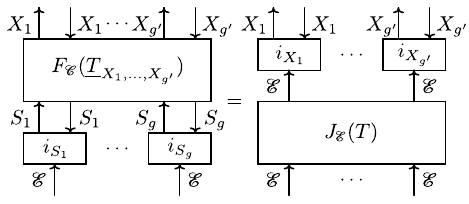}
\end{equation}
for $X_1,\ldots,X_{g'} \in \calC$. The uniqueness of this solution can be used to show that, if $\calC$ is a finite ribbon category, then
Equation~\eqref{E:KL_TQFT_top_tangle} defines a braided monoidal functor
\[
 J_\eend : \TTan \to \calC.
\]
In particular, as a direct consequence, we obtain a Hopf algebra structure on the adjoint end $\eend$ of a finite ribbon category $\calC$, whose structure morphisms are the images of those of $\bfSigma_1$ under $J_\eend$.

\begin{remark}\label{R:end_vs_coend_1}
 \KL{} TQFTs are originally defined using the coadjoint coend $\coend$ instead of the adjoint end $\eend$. However, if $\calC$ is a modular category, the two constructions are completely equivalent to each other. Indeed, instead of using top tangles in handlebodies, we can consider \textit{bottom tangles in handlebodies}, which are obtained from the former by simply applying a rotation $R$ of an angle $\pi$ around the horizontal axis $D^1 \times \{ (0,\frac{1}{2}) \} \subset D^2 \times I$.
 Just like before, every bottom tangle $T \in \BTan(g,g')$ induces a dinatural transformation that, by invoking the universal property of the coend $\coend$, determines a unique morphism $J_\coend(T) : \coend^{\otimes g} \to \coend^{\otimes g'}$. This defines a Hopf algebra structure on $\coend$. Notice however that the rotation $R$ exchanges the roles of product and coproduct, and of unit and counit. That is, by applying $R$ to the top tangles defining the algebra and the coalgebra structure on $\bfSigma_1$, we obtain the bottom tangles defining the coalgebra and the algebra structure, respectively. However, thanks to the universal properties satisfied by $\eend$ and $\coend$, there exists a unique morphism $\calD : \coend \to \eend$, called the \textit{Drinfeld map}, satisfying
 \begin{equation}\label{E:Drinfeld}
  \pic{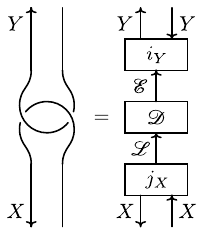}
 \end{equation}
 for all $X,Y \in \calC$. The Drinfeld map is a Hopf algebra morphism, and it is invertible if and only if $\calC$ is factorizable, see \cite[Proposition~4.11]{FGR17}.
\end{remark}

If $\calC$ is unimodular, then there exists, up to scalar, a unique
\textit{integral form on $\eend$}, which is a morphism $\lambda \in \calC(\eend,\one)$ satisfying
\begin{align}\label{E:int_form}
 (\id_\eend \otimes \lambda) \circ \Delta &= \eta \circ \lambda, &
 \lambda \circ S &= \lambda.
\end{align}
Similarly, and always up to scalar, there exists a unique \textit{integral element of $\eend$}, which is a morphism $\Lambda \in \calC(\one,\eend)$ satisfying
\begin{align}\label{E:int_element}
 \mu \circ (\Lambda \otimes \id_\eend) &= \Lambda \circ \varepsilon, &
 S \circ \Lambda &= \Lambda.
\end{align}
The two statements are proved in \cite[Proposition~3.1]{BKLT97} and \cite[Lemmas~5.2.8 \& 5.2.12]{KL01}.
If $\lambda$ and $\Lambda$ are both non-zero, then their composition $\lambda \circ \Lambda$ is a non-zero scalar multiple of $\id_{\one}$, thanks to \cite[Theorem~3.3]{BKLT97}. Furthermore, if $w = (\mu \otimes \mu) \circ (v_- \otimes \Delta \otimes v_-) \circ v_+$ denotes the Hopf copairing of $\eend$, then $\calC$ is factorizable if and only if $(\id_\eend \otimes \lambda) \circ w$ is a non-zero scalar multiple of $\Lambda$, thanks to \cite[Proposition~5.2.9]{KL01} and \cite[Proposition~7.5]{BD21}. Therefore, if $\calC$ is a modular category over an algebraically closed field $\Bbbk$, then, up to sign, there exists a unique \textit{normalized integral form on $\eend$}, which is an integral form $\lambda \in \calC(\eend,\one)$ that satisfies
\begin{equation}\label{E:int_normalization}
 (\lambda \otimes \lambda) \circ w = \id_{\one}.
\end{equation}
Notice that a normalized integral form satisfies Equations~\eqref{E:Hopf_integral_axiom} and \eqref{E:Hopf_cointegral_axiom}.

Let us pause to record a few identities that will be later used in the proof of Proposition~\ref{P:KL_TQFT}. First of all, Equation~\eqref{E:Hopf_braided_cocommutative_axiom} implies that the adjoint action $\alpha$ of $\eend$ on itself is braided cocommutative. As a consequence, $\lambda$ intertwines $\alpha$, which means that
\begin{align}\label{E:int_adj}
 \lambda \circ \alpha &= \varepsilon \otimes \lambda,
\end{align}
see
\cite[Figure~3.4.7]{BBDP23}.
Furthermore, $\lambda$ is a quantum character, which means that
\begin{align}\label{E:int_q_char}
 \lambda \circ \mu \circ (\id_\eend \otimes S^2) \circ c_{\eend,\eend} &= \lambda \circ \mu,
\end{align}
see \cite[Lemma~3.2]{BDF25}. Finally, $\lambda$ satisfies
\begin{equation}\label{E:int_slam-dunk}
 \lambda \circ (\lambda \otimes \mu) \circ (w \otimes \id_\eend) = \varepsilon,
\end{equation}
as follows directly from Equation~\eqref{E:int_normalization}.

We are now ready to move on to arbitrary Kirby tangles. If $T : g \to g'$ is an arbitrary Kirby tangle, a top tangle presentation $\tilde{T}$ of $T$ is a top tangle $\tilde{T} : g \to g'+h$ satisfying
\[
 \pic{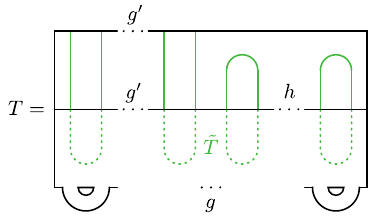}
\]
Notice that top tangle presentations always exist, but they are of course by no means unique. If $\tilde{T} : g \to g'+h$ is a top tangle presentation of a Kirby tangle $T : g \to g'$, then we consider the morphism $J_\eend(T) : \eend^{\otimes g} \to \eend^{\otimes g'}$ defined by
\begin{equation}\label{E:KL_TQFT}
 \pic{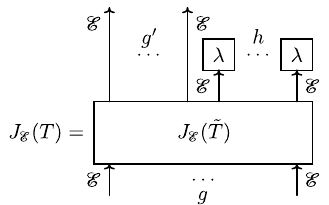}
\end{equation}
The following statement implies that this definition is independent of the choice of the top tangle presentation.

\begin{proposition}\label{P:KL_TQFT}
 If $\calC$ is a modular category and $\lambda \in \calC(\eend,\one)$ is a normalized integral form on $\eend$, then
 Equation~\eqref{E:KL_TQFT}
defines a braided monoidal functor
 \[
  J_\eend : \KTan \to \calC.
 \]
\end{proposition}

As explained earlier, Proposition~\ref{P:KL_TQFT} follows from the construction of Kerler and Lyubashenko \cite{KL01}, and a proof can be found in \cite{BD21}. However, a different proof, based entirely on the universal property of $\eend$, is given in Appendix~\ref{A:proof_KL_TQFT}.

\subsection{Kerler--Lyubashenko--Reshetikhin--Turaev Functor}\label{S:KLRT_functors}

Next, let us consider a finite ribbon category $\calC$ with adjoint end $\eend$, and let us recall the construction of a ribbon functor $F_\eend : \calB_\calC \to \calC$
that we will refer to as the \textit{\KLRT{} functor} (called the \LRT{} functor in \cite[Proposition~3.1]{DGGPR19}, where it is denoted $F_\Lambda : \calR_\Lambda \to \calC$) extending the \RT{} functor $F_\calC : \calR_\calC \to \calC$ of $\calC$-labeled ribbon graphs.

First, the \textit{smoothing} of a generalized $\calC$-labeled bichrome graph $T \subset M$ from $P \subset \Sigma$ to $P' \subset \Sigma'$ is the unoriented unlabeled red framed tangle $R(T) \subset M$ from $R(P) \subset \Sigma$ to $R(P') \subset \Sigma'$ obtained from $T$ by throwing away blue edges and coupons, and by replacing bichrome coupons as shown
\[
 \pic{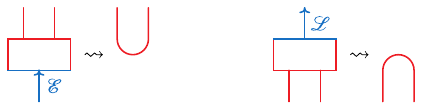}
\]
A regular $\calC$-labeled bichrome set $P \subset D^2$ is \textit{$\ell$-split} if $P = R(P) \bcs B(P)$, where $R(P) \subset D^2$ is a regular red set composed of $2\ell$ points, $B(P) \subset D^2$ is a regular $\calC$-labeled blue set, and $\bcs$ denotes the extension of the tensor product of $\KTan$ to regular bichrome graphs in $D^2 \times I$.
If $P,P' \subset D^2$ are an $\ell$-split and an $\ell'$-split regular $\calC$-labeled bichrome set, respectively, then we say a generalized $\calC$-labeled bichrome graph $T \subset D^2 \times I$ from $P$ to $P'$ is a \textit{generalized $\calC$-labeled Kirby $(\ell,\ell')$-graph}, or sometimes simply a \textit{generalized Kirby graph}, if $R(T) = A(R(T)) \cup C(R(T))$, where $A(R(T))$ is the union of a bottom $\ell$-tangle and a top $\ell'$-tangle, and $C(R(T))$ is a framed link. A \textit{$\calC$-labeled bottom-top $(\ell,\ell')$-graph}, or sometimes simply a \textit{bottom-top graph}, is a generalized Kirby graph $T$ such that $C(R(T)) = \varnothing$.

Let us start from bottom-top graphs. If $P,P' \subset D^2$ are an $\ell$-split and an $\ell'$-split regular $\calC$-labeled bichrome set, respectively, and $T \subset D^2 \times I$ is a $\calC$-labeled bottom-top $(\ell,\ell')$-graph from $P$ to $P'$, then, for all $X_1,\ldots,X_\ell,X'_1,\ldots,X'_{\ell'} \in \calC$, let $T_{X_1,\ldots,X_\ell,X'_1,\ldots,X'_{\ell'}}$ be the $\calC$-labeled ribbon graph obtained from $T$ by orienting the $k$th bottom component and the $k'$th top component of its smoothing from right to left, by labeling them by $X_k$ and $X'_{k'}$, respectively, and by labeling every bichrome coupon intersecting them by the component at either $X_k$ or $(X_k)^*$ and at either $X'_{k'}$ or $(X'_{k'})^*$, respectively, of the family of structure morphisms of either $\eend$ or $\coend$
(according to its configuration). Then, by applying the \RT{} functor $F_\calC$, we obtain a morphism of $\calC$ that is dinatural in $X_1,\ldots,X_\ell,X'_1,\ldots,X'_{\ell'} \in \calC$. Therefore, thanks to the universal property satisfied by $\eend$ and $\coend$, there exists a unique morphism $F_\calC(T) : \coend^{\otimes \ell} \otimes F_\calC(B(P)) \to \eend^{\otimes \ell'} \otimes F_\calC(B(P'))$ satisfying
\begin{equation}\label{E:KLRT_functor_pure_top_graph}
 \pic{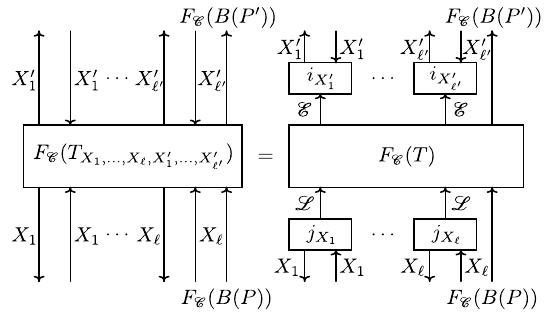}
\end{equation}
for all $X_1, \ldots, X_\ell,X'_1,\ldots,X'_{\ell'} \in \calC$.

Let us move on to arbitrary Kirby graphs. If $P,P' \subset D^2$ are an $\ell$-split and an $\ell'$-split regular $\calC$-labeled bichrome set, respectively, and $T \subset D^2 \times I$ is a generalized $\calC$-labeled Kirby $(\ell,\ell')$-graph from $P$ to $P'$, then a \textit{bottom-top graph presentation of $T$} is a $\calC$-labeled bottom-top $(\ell,k+\ell')$-graph $\tilde{T}$ from $P$ to $\tilde{P} \bcs P'$, where $\tilde{P} \subset D^2$ is a regular red set composed of $2k$ points, satisfying
\[
 \pic{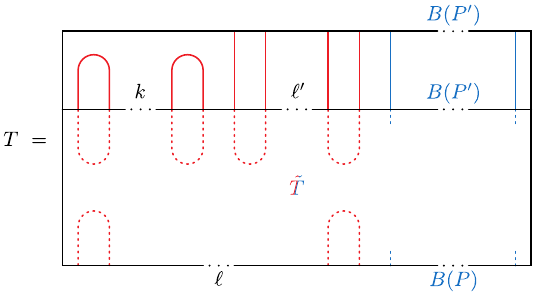}
\]
If a $\calC$-labeled bottom-top $(\ell,k+\ell')$-graph $\tilde{T}$ from $P$ to $\tilde{P} \bcs P'$ is a bottom-top graph presentation of a generalized $\calC$-labeled Kirby $(\ell,\ell')$-graph $T$ from $P$ to $P'$, then $F_\eend(T) : \coend^{\otimes \ell} \otimes F_\calC(B(P)) \to \eend^{\otimes \ell'} \otimes F_\calC(B(P'))$ is defined as
\begin{equation}\label{E:KLRT_functor}
 \pic{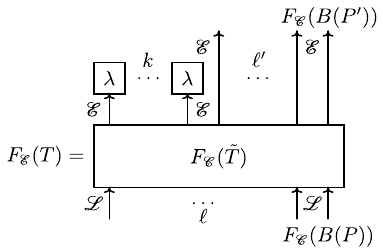}
\end{equation}
where $\lambda \in \calC(\eend,\one)$ is an integral form on $\eend$. When $\ell = \ell' = 0$, this map extends to a ribbon functor.

\begin{proposition}\label{P:KLRT_functor}
 If $\calC$ is a finite ribbon category and $\lambda \in \calC(\eend,\one)$ is an integral form on the adjoint end $\eend$, then
 Equation~\eqref{E:KLRT_functor}
 defines a ribbon functor
 \[
  F_\eend : \calB_\calC \to \calC.
 \]
\end{proposition}

For a proof of Proposition~\ref{P:KLRT_functor}, see \cite[Propositions~3.1 \& 3.5]{DGGPR19}. Notice that the statement remains true if the integral form $\lambda \in \calC(\eend,\one)$ is replaced by any quantum character on $\eend$, meaning any antipode-invariant form on $\eend$ satisfying Equation~\eqref{E:int_q_char}, compare with the proof of \cite[Proposition~4.3]{BDF25}.

\begin{remark}\label{R:end_vs_coend_3}
 The \LRT{} functor of \cite[Equation~(27)]{DGGPR19} is actually defined using the coadjoint coend $\coend$ instead of the adjoint end $\eend$, and should therefore be denoted $F_\coend : \calB_\calC \to \calC$ here. Notice however that, in the setting that is relevant to us (that is, when $\calC$ is factorizable), the two constructions are completely equivalent to each other, for the same reasons we pointed out in Remark~\ref{R:end_vs_coend_1}.
\end{remark}

The construction above allows us to extend the notion of skein equivalence. If $P,P' \subset D^2$ are an $\ell$-split and an $\ell'$-split regular $\calC$-labeled bichrome set, respectively, then we say two linear combinations $T$ and $T'$ of generalized $\calC$-labeled Kirby $(\ell,\ell')$-graphs from $P$ to $P'$ are \textit{skein equivalent in $D^2 \times I$} if
\[
 F_\eend(T) = F_\eend(T').
\]
If $M$ is a \dmnsnl{3} cobordism with corners from $\Sigma$ to $\Sigma'$, and if $P \subset \Sigma$ and $P' \subset \Sigma'$ are $\calC$-labeled blue sets, then we say two linear combinations $T$ and $T'$ of $\calC$-labeled bichrome graphs in $M$ from $P$ to $P'$ are \textit{skein equivalent in $M$} if they are related by a finite sequence of skein equivalences in embedded copies of $D^2 \times I$. In this case, we write
\[
 T \doteq T'.
\]

\subsection{Renormalized Kerler--Lyubashenko TQFT}\label{S:renormalized_KL_TQFTs}

Now, let us consider a modular category $\calC$, and let us recall the construction of a symmetric monoidal functor $V_\eend : \ACob_\calC \to \FVect_\Bbbk$
that we will refer to as the \textit{renormalized \KL{} TQFT} (called the renormalized Lyubashenko TQFT in \cite[Theorem~4.12]{DGGPR19}, where it is denoted $\rmV_\calC : \adCob_\calC \to \Vect_\Bbbk$).

Let us fix a trace $\rmt$ on the ideal $\Proj(\calC)$ of projective objects of $\calC$. As recalled in Section~\ref{S:ribbon_categories}, such a trace is unique up to scalar, and furthermore non-degenerate. It is convenient to lock the normalization of $\rmt$ to those of the integral form $\lambda$ on $\eend$ and of the integral element $\Lambda$ of $\eend$, by exploiting the fact that, if $\varepsilon_{\one} : P_{\one} \to \one$ is the projective cover of $\one$, then there exists a unique morphism $\eta_{\one} : \one \to P_{\one}$ satisfying
\begin{equation}\label{E:eta_one}
 i_{P_X} \circ \Lambda = \delta_{X,\one} \eta_{\one} \otimes \varepsilon_{\one}^*
\end{equation}
for every simple object $X \in \calC$, as follows from \cite[Corollary~6.4]{GR17} and Remark~\ref{R:end_vs_coend_1}. Notice that $\eta_{\one} : \one \to P_{\one}$ makes $P_{\one}$ into the injective envelope of $\one$. Then, we fix the normalization of $\rmt$ by asking that
\[
 \rmt_{P_{\one}}(\eta_{\one} \circ \epsilon_{\one}) = 1.
\]
This normalization is possible thanks to the non-degeneracy of $\rmt$, together with the fact that $\varepsilon_{\one}$ and $\eta_{\one}$ generate their respective morphism spaces. Notice that $\epsilon_{\one} \circ \eta_{\one} \neq 0$ if and only if $\calC$ is semisimple.

A $\calC$-labeled blue set $P$ in a \dmnsnl{2} cobordism $\Sigma$ is \textit{admissible} if it admits a label in $\Proj(\calC)$. Similarly, a $\calC$-labeled bichrome graph in a \dmnsnl{3} cobordism with corners $M$ is \textit{admissible} if it admits a label in $\Proj(\calC)$, and it is \textit{closed} if it is disjoint from $\partial M$. In particular, every admissible closed bichrome graph $T : \varnothing \to \varnothing$ in $\calB_\calC$ determines an endomorphism $\bfS^3_T = (S^3,T,0)$ of $\varnothing$ in $\ACob_\calC$. Furthermore, it admits a cutting presentation, which is a bichrome graph $T_P : P \to P$ in $\calB_\calC$, with $P$ admissible, whose trace is $T$. If $T$ is a closed admissible bichrome graph in $\calB_\calC$ and $T_P$ is a cutting presentation, then the scalar
\begin{equation}\label{E:renormalized_admissible_closed_graph_invariant}
 F'_\eend(T) := \rmt_{F_\calC(P)}(F_\eend(T_P))
\end{equation}
is independent of $T_P$, and is a topological invariant of $T$. See \cite[Theorem~3.3]{DGGPR19} for a proof, where $F'_\eend$ is denoted $F'_\intL$.

Let $\bfM = (M,T,n)$ be an admissible connected closed \mnfld{3}, and let $L$ be a surgery presentation of $M$ given by a red framed link in $S^3$ whose linking matrix has signature $n$. We assume the bichrome graph $T$ to be contained in the exterior of the surgery link $L$, so that we can think of them as being simultaneously embedded into $S^3$. Then, the scalar
\begin{equation}\label{E:renormalized_Lyubashenko}
 J'_\eend(\bfM) := F'_\eend(L \cup T)
\end{equation}
is a topological invariant of the admissible triple $\bfM = (M,T,n)$, called the \textit{renormalized Lyubashenko invariant}. See \cite[Theorem~3.8]{DGGPR19} for a proof, where $J'_\eend$ is denoted $\rmL'_\calC$.

\begin{remark}\label{R:normalization}
 Equation~\eqref{E:renormalized_Lyubashenko} defines directly the normalization of the invariant determined by \cite[Equation~(6)]{DGGPR20}. Indeed, since here we are using the integral form $\lambda \in \calC(\eend,\one)$ satisfying Equations~\eqref{E:Hopf_integral_axiom} and \eqref{E:Hopf_cointegral_axiom}, which is unique up to sign, the coefficient $\calD$ appearing in \cite[Equation~(4)]{DGGPR20} is equal to $1$, in our conventions.
\end{remark}

The universal construction of \cite{BHMV95} can be used to extend $J'_\eend$ to a TQFT defined on $\ACob_\calC$. Indeed, if $\bfM_i$ is an admissible connected closed \mnfld{3} for every $1 \leqs i \leqs k$, we start by setting
\[
 J'_\eend \left( \bigsqcup_{i=1}^k \bfM_i \right) := \prod_{i=1}^k J'_\eend(\bfM_i).
\]
Then, if $\bfSigma$ is an object of $\ACob_\calC$, we consider vector spaces $\calV(\bfSigma)$ and $\calV'(\bfSigma)$ with bases given by morphisms of $\ACob_\calC$ of the form $\bfM_{\bfSigma} : \varnothing \to \bfSigma$ and $\bfM'_{\bfSigma} : \bfSigma \to \varnothing$, respectively. These vector spaces are paired via
\begin{align*}
 \langle \_,\_ \rangle_{\bfSigma} : \calV'(\bfSigma) \times \calV(\bfSigma) & \to \Bbbk \\*
 (\bfM'_{\bfSigma},\bfM_{\bfSigma}) & \mapsto J'_\eend(\bfM'_{\bfSigma} \circ \bfM_{\bfSigma}).
\end{align*}
State spaces are then defined as quotients with respect to the left and the right radical of the pairing $\langle \_,\_ \rangle_{\bfSigma}$, namely
\begin{equation}\label{E:renormalized_KL_state_spaces}
 V_\eend(\bfSigma) := \calV(\bfSigma) / \rrad \langle \_,\_ \rangle_{\bfSigma}, \qquad
 V'_\eend(\bfSigma) := \calV'(\bfSigma) / \lrad \langle \_,\_ \rangle_{\bfSigma}.
\end{equation}
If $\bfM : \bfSigma \to \bfSigma'$ is a morphism of $\ACob_\calC$, its associated operators are simply defined as
\begin{align}
 V_\eend(\bfM) : V_\eend(\bfSigma) & \to V_\eend(\bfSigma')
 &
 V'_\eend(\bfM) : V'_\eend(\bfSigma') & \to V'_\eend(\bfSigma) \label{E:renormalized_KL_time-evolution_operators} \\*
 {}[\bfM_{\bfSigma}] & \mapsto [\bfM \circ \bfM_{\bfSigma}], &
 {}[\bfM'_{\bfSigma}] & \mapsto [\bfM'_{\bfSigma} \circ \bfM]. \nonumber
\end{align}

\begin{proposition}\label{P:renormalized_KL_TQFT}
 If $\calC$ is a modular category with adjoint end $\eend$, then Equations~\eqref{E:renormalized_KL_state_spaces} and \eqref{E:renormalized_KL_time-evolution_operators} define a pair of symmetric monoidal functors
 \[
 V_\eend : \ACob_\calC \to \FVect_\Bbbk, \quad
 V'_\eend : (\ACob_\calC)^{\op} \to \FVect_\Bbbk.
\]
\end{proposition}

For a proof of Proposition~\ref{P:renormalized_KL_TQFT}, see \cite[Theorem~4.12]{DGGPR19}.

\begin{remark}\label{R:end_vs_coend_5}
 Notice that, with respect to the conventions of \cite[Section~4.3]{DGGPR19}, the TQFTs $V_\eend$ and $V'_\eend$ have been interchanged here, as a consequence of Remark~\ref{R:end_vs_coend_4}.
\end{remark}

These functors are dually paired, in the sense that, for every $\bfSigma \in \ACob_\calC$, the bilinear form
\[
 \langle \_,\_ \rangle_{\bfSigma} : V'_\eend(\bfSigma) \times V_\eend(\bfSigma) \to \Bbbk
\]
is non-degenerate, and for all $\bfM : \bfSigma \to \bfSigma'$, $[\bfM_{\bfSigma}] \in V_\eend(\bfSigma)$, and $[\bfM'_{\bfSigma'}] \in V'_\eend(\bfSigma')$, it satisfies
\begin{equation}\label{E:adjunction_V_V'}
 \langle \bfM'_{\bfSigma'}, V_\eend(\bfM)(\bfM_{\bfSigma})
 \rangle_{\bfSigma'} =
 \langle V'_\eend(\bfM)(\bfM'_{\bfSigma'}), \bfM_{\bfSigma}
 \rangle_{\bfSigma}.
\end{equation}
As a consequence, the linear map
\begin{align}
 t_{\bfSigma,\bfSigma'} : V_\eend(\bfSigma') \otimes V'_\eend(\bfSigma) &\to \Hom_\Bbbk(V_\eend(\bfSigma),V_\eend(\bfSigma')) \label{E:linear_maps_tensor_product} \\*
 [\bfM_{\bfSigma'}] \otimes [\bfM'_{\bfSigma}] &\mapsto V_\eend(\bfM_{\bfSigma'} \ast \bfM'_{\bfSigma})
 = \langle  [\bfM'_{\bfSigma}],\_ \rangle_{\bfSigma} [\bfM_{\bfSigma'}] \nonumber
\end{align}
is an isomorphism.

\subsection{Morphism spaces and admissible skein modules}

We finish this section by recalling a few important properties of the renormalized \KL{} TQFT, which will be later used for the main construction of this paper.

First of all, the TQFT is invariant under (admissible) color changes: indeed, closed green components can be turned into red ones, and closed red components can be turned into $\calC$-labeled blue graphs in the presence of projective blue labels. More precisely, the image under $V_\eend$ of a morphism of $\ACob_\calC$ does not change if we perform a move of the following type:
\begin{gather*}
 \pic{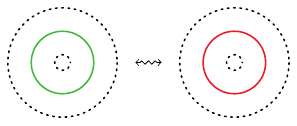} \tag{GR}\label{E:GR} \\*
 \pic{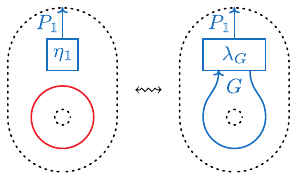} \tag{RB}\label{E:RB}
\end{gather*}
Here, $\eta_{\one} : \one \to P_{\one}$ is the monomorphism determined by Equation~\eqref{E:eta_one}, while $\lambda_G : G \otimes G^* \to P_{\one}$ is any morphism satisfying
\begin{equation}
 \lambda_G \circ i_G = \eta_{\one} \circ \lambda,
\end{equation}
whose existence is ensured by the fact that $P_{\one}$ is injective, together with the fact that $i_G : \eend \to G \otimes G^*$ is a monomorphism, because $G$ is a projective generator of $\calC$. These operations can be performed inside any solid torus $S^1 \times D^2$ arbitrarily embedded into a \dmnsnl{3} cobordism $M$ (represented here as the annuli contained between the dashed circles on both sides of the operations). We refer to \eqref{E:GR} as the \textit{green-to-red move}, and to \eqref{E:RB} as the \textit{red-to-blue move}. The proof of the invariance of $V_\eend$ under \eqref{E:GR} and \eqref{E:RB} follows from \cite[Propositions~4.5 \& 4.10]{DGGPR19}, together with Remarks~\ref{R:end_vs_coend_1} and \ref{R:normalization}.

Furthermore, the TQFT is also invariant under (admissible) connected sums of diagrams representing different connected components of a decorated cobordism. More precisely, the image under $V_\eend$ of a morphism of $\ACob_\calC$ does not change if we perform a move of the following type:
\begin{gather*}
 \pic{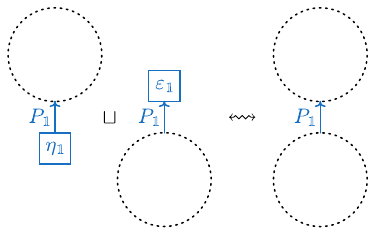} \tag{CS}\label{E:CS}
\end{gather*}
These operations replace an embedded copy of $S^0 \times D^3$ in a \dmnsnl{3} cobordism $M \sqcup M'$ with a copy of $D^1 \times S^2$ (represented here as the exteriors of the dashed circles on both sides of the operation). We refer to \eqref{E:CS} as the \textit{connected sum move}. The proof of the invariance of $V_\eend$ under \eqref{E:CS} follows from \cite[Proposition~4.10]{DGGPR19}, together with Remarks~\ref{R:end_vs_coend_1} and \ref{R:normalization}.

Next, let $\bfSigma = (\Sigma,P,\Lagr)$ and $\bfSigma' = (\Sigma',P',\Lagr')$ be objects of $\bfACob_\calC(\bfGamma,\bfGamma')$, and let $M$ be a connected \dmnsnl{3} cobordism with corners from $\Sigma$ to $\Sigma'$. We denote by $\calS(M;P,P')$ the free vector space generated by the set of all isotopy classes of $\calC$-labeled bichrome graphs $T \subset M$ from $P$ to $P'$, and we denote by $\calS'(M;P,P')$ the subspace of $\calS(M;P,P')$ generated by isotopy classes of admissible $\calC$-labeled bichrome graphs $T \subset M$ from $P$ to $P'$. A \textit{skein equivalence in $\calS(M;P,P')$} is a finite sequence of skein equivalences in embedded copies of $D^2 \times I$, and an \textit{admissible skein equivalence in $\calS'(M;P,P')$} is a finite sequence of skein equivalences in embedded copies of $D^2 \times I$ whose horizontal boundaries always intersect an edge labeled by an object of $\Proj(\calC)$. The \textit{skein module} $S_\eend(M;P,P')$ is the quotient of $\calS(M;P,P')$ with respect to the skein equivalence relation, and the \textit{admissible skein module} $S'_\eend(M;P,P')$ is the quotient of $\calS'(M;P,P')$ with respect to the admissible skein equivalence relation. We set $S^\partial_\eend(M;P,P') := S'_\eend(M;P,P')$ if $\bfSigma' = \varnothing$, and $S^\partial_\eend(M;P,P') := S_\eend(M;P,P')$ otherwise. Notice that, if either $P$ or $P'$ is admissible, then automatically $S_\eend(M;P,P') = S'_\eend(M;P,P')$.

If $\bfSigma_i = (\Sigma_i,P_i,\Lagr_i)$ and $\bfSigma'_i = (\Sigma'_i,P'_i,\Lagr'_i)$ are objects of $\bfACob_\calC(\bfGamma_i,\bfGamma'_i)$ for every integer $1 \leqs i \leqs k$, and if $M_i$ is a connected \dmnsnl{3} cobordism with corners from $\Sigma_i$ to $\Sigma'_i $ for every integer $1 \leqs i \leqs k$, then we set
\[
 S^\partial_\eend \left( \bigsqcup_{i=1}^k M_i; \bigsqcup_{i=1}^k P_i, \bigsqcup_{i=1}^k P'_i \right) := \bigotimes_{i=1}^k S^\partial_\eend (M_i;P_i,P'_i)
\]

Now, let $\Sigma$ and $\Sigma'$ be \dmnsnl{2} cobordisms from $\Gamma$ to $\Gamma'$. A \dmnsnl{3} cobordism with corners $M$ from $\Sigma$ to $\Sigma'$ is
\textit{boundary bijective}
if the embedding
\[
 \iota : \Sigma' \cup_{\overline{\Gamma} \sqcup \Gamma'} \overline{\Sigma} \hookrightarrow M
\]
induced by the boundary parametrization determines
a bijection
\[
 \iota_* : \pi_0(\Sigma' \cup_{\overline{\Gamma} \sqcup \Gamma'} \overline{\Sigma}) \to \pi_0(M).
\]
Then, the proof of the next statement follows from \cite[Proposition~4.11 \& Theorem~4.12]{DGGPR19}.

\begin{lemma}\label{L:connected_1}
 Let $\bfSigma = (\Sigma,P,\Lagr)$ be an object of $\ACob_\calC$. If $M$ is a boundary bijective \dmnsnl{3} cobordism from $\varnothing$ to $\Sigma$, and $M'$ is a boundary bijective \dmnsnl{3} cobordism from $\Sigma$ to $\varnothing$, then the linear maps
 \begin{align*}
  \pi_M : S_\eend(M;\varnothing,P) &\to V_\eend(\bfSigma) &
  \pi_{M'} : S'_\eend(M';P,\varnothing) &\to V'_\eend(\bfSigma) \\*
  [T] &\mapsto [M,T,0] &
  [T'] &\mapsto [M',T',0]
 \end{align*}
 are surjective.
\end{lemma}

Notice that changing the signature of a morphism of $\ACob_\calC$ multiplies its image under the TQFTs $V_\eend$ and $V'_\eend$ by an invertible scalar. This is a direct consequence of Equation~\ref{E:renormalized_Lyubashenko} and \cite[Proposition~2.6]{DGGPR19}.

Finally, for all objects $P \in \calB_\calC$ and integers $g \geqs 0$, let us fix the \mrphsm{1} $\bfD^2(P) : \varnothing \to \bfS^1$ of $\bfACob_\calC$ given by
\begin{equation}\label{E:disc_object}
 \bfD^2(P) := (D^2,P,\{0\}),
\end{equation}
and the \mrphsm{1} $\bfSigma_g : \varnothing \to \bfS^1$ of $\bfNCob = \bfACob_\calC \cap \bfCob$ defined by Equation~\eqref{E:surgery_functor_objects}.
Let us also denote by $(\bfD^2(P))^\dagger : \bfS^1 \to \varnothing$ the adjoint \mrphsm{1} to $\bfD^2(P) : \varnothing \to \bfS^1$ defined by Equation~\eqref{E:adjoint_1-morphism}. Then, we consider the linear map
\begin{align}
 \Psi_{P,g} : \calC(F_\calC(P),\eend^{\otimes g}) &\to V_\eend((\bfD^2(P))^\dagger \circ \bfSigma_g) \label{E:algebraic_model} \\*
 f &\mapsto ((\bfD^2(P))^\dagger \triangleright [M(\eta^{{} \bcs g}),T_f,0]) \ast \radun_{\bfD^2(P)} \nonumber
\end{align}
determined by Equation~\eqref{E:right_adjunction_unit}, where $M(\eta^{{} \bcs g})$ denotes the \dmnsnl{3} cobordism with corners from $\Sigma_0$ to $\Sigma_g$ associated to the $g$th tensor power of the unit $\eta$ of $\KTan$ by Equation~\eqref{E:surgery_functor_morphisms}, and where $T_f \subset M(\eta^{{} \bcs g})$ is the $\calC$-labeled bichrome graph from $P$ to $\varnothing$ defined by
\begin{align}\label{E:skein_model_picture}
 T_f := \pic{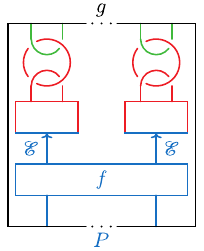}
\end{align}
We also consider the linear map
\begin{align}
 \Psi'_{P,g} : S'_\eend(M(\epsilon^{{} \bcs g});\varnothing,P) &\to V'_\eend((\bfD^2(P))^\dagger \circ \bfSigma_g) \label{E:skein_model} \\*
 T &\mapsto \ladcoun_{\bfD^2(P)} \ast ((\bfD^2(P))^\dagger \triangleright [M(\epsilon^{{} \bcs g}),T,0]) \nonumber
\end{align}
determined by Equation~\eqref{E:left_adjunction_counit}, where $M(\varepsilon^{{} \bcs g})$ denotes the \dmnsnl{3} cobordism with corners from $\Sigma_g$ to $\Sigma_0$ associated to the $g$th tensor power of the counit $\epsilon$ of $\KTan$ by Equation~\eqref{E:surgery_functor_morphisms}.

\begin{proposition}\label{P:algebraic_and_skein_model}
 If $\bfD^2(P), \bfSigma_g : \varnothing \to \bfS^1$ denote the \mrphsms{1} of $\bfACob_\calC$ defined by Equations~\eqref{E:disc_object} and \eqref{E:surgery_functor_objects}, respectively, then the linear map
 \[
  \Psi_{P,g} : \calC(F_\calC(P),\eend^{\otimes g}) \to V_\eend((\bfD^2(P))^\dagger \circ \bfSigma_g)
 \]
 defined by Equation~\eqref{E:algebraic_model} is an isomorphism that is natural in $P \in \calB_\calC$, and similarly the linear map
 \[
  \Psi'_{P,g} : S'_\eend(M(\epsilon^{{} \bcs g});\varnothing,P) \to V'_\eend((\bfD^2(P))^\dagger \circ \bfSigma_g)
 \]
 defined by Equation~\eqref{E:skein_model} is an isomorphism that is natural in $P \in \calB_\calC$.
\end{proposition}

The first statement is a direct consequence of \cite[Proposition~4.17]{DGGPR19}, while the second one is proved in \cite[Lemma~6.8]{H25} for $P = \varnothing$, although the proof can be adapted to the general case.

\begin{remark}\label{R:opposite_admissibility_convention}
 Proposition~\ref{P:algebraic_and_skein_model} implies that state spaces for $V_\eend$ are naturally isomorphic to morphism spaces of $\calC$, while state spaces for $V'_\eend$ are naturally isomorphic to admissible skein modules associated with the ideal $\Proj(\calC)$. This is precisely the reason why we adopt the opposite convention with respect to \cite{DGGPR19} for what concerns the admissibility condition. Indeed, the \KL{} TQFT is formulated in algebraic terms that are closer to the description of the functor $V_\eend$ than to the skein-theoretic language that describes $V'_\eend$. Both models have their own advantages, and the choice of conventions here is merely designed to facilitate the formulation of the relation between the \KL{} TQFT and its renormalized version.
\end{remark}

\section{Extended TQFTs}\label{S:ETQFTs}

In this section, we construct a \fnctr{2} $\bfA_\eend : \bfACob_\calC \to \bfCat_\Bbbk$, and we prove that its finite completion $\bfhatA_\eend : \bfACob_\calC \to \coCat_\Bbbk$ admits a symmetric monoidal structure in Theorem~\ref{T:symmetric_monoidality}. This result rests on an explicit description of the end of a functor from $\calC \times \calC^\op$ to $\bfA_\eend({\bfS^1} \disjun {(\bfS^1)^*})$ in terms of (the opposite of) the annulus $I \times S^1$. We also compute images of generating objects and \mrphsms{1} of $\bfACob_\calC$ to recover and extend analogous computations of \cite{D21}, thus obtaining results that match the definition of \cite{L25}. We finish by showing that the ETQFT $\bfhatA_\eend : \bfACob_\calC \to \coCat_\Bbbk$ indeed contains both the \KL{} TQFT $J_\eend : \CCob \to \calC$ and its renormalized version $V_\eend : \ACob_\calC \to \FVect_\Bbbk$, in Propositions~\ref{P:KL_TQFT_inside_ETQFT} and \ref{P:renormalized_KL_TQFT_inside_ETQFT}.

\subsection{Extended universal construction}\label{S:extended_universal_construction}

First, let us recall how to extend the TQFT $V_\eend : \ACob_\calC \to \FVect_\Bbbk$ to a pair of \fnctrs{2}
\begin{align*}
 \bfA_\eend &: \bfACob_\calC \to \bfCat_\Bbbk, &
 \bfA'_\eend &: (\bfACob_\calC)^\op \to \bfCat_\Bbbk
\end{align*}
following a procedure called the \textit{extended universal construction}, which is inspired by the universal construction of \cite{BHMV95}, and which was already used in \cite[Section~3.4]{D17} and in \cite[Section~5]{D21}.

Let us start by considering an object $\bfGamma$ of $\bfACob_\calC$. Let $\calA(\bfGamma)$ denote the linear category generated by the category $\bfACob_\calC(\varnothing,\bfGamma)$ of \mrphsms{1} $\bfSigma_{\bfGamma} : \varnothing \to \bfGamma$ and \mrphsms{2} $\bfM_{\bfGamma} : \bfSigma_{\bfGamma} \Rightarrow \bfSigma''_{\bfGamma}$ of $\bfACob_\calC$, and let $\calA'(\bfGamma)$ denote the linear category generated by the category $\bfACob_\calC(\bfGamma,\varnothing)$ of \mrphsms{1} $\bfSigma'_{\bfGamma} : \bfGamma \to \varnothing$ and \mrphsms{2} $\bfM'_{\bfGamma} : \bfSigma'_{\bfGamma} \Rightarrow \bfSigma'''_{\bfGamma}$ of $\bfACob_\calC$. Here, recall that the linear category generated by another category is simply defined as the linear category with the same set of objects, and with vector spaces of morphisms generated by the corresponding sets of morphisms. Next, we consider the bilinear functor
\[
 \langle \_ , \_ \rangle_{\bfGamma} : \calA'(\bfGamma) \times \calA(\bfGamma) \to \FVect_\Bbbk,
\]
sending every object $(\bfSigma'_{\bfGamma}, \bfSigma_{\bfGamma}) \in \calA'(\bfGamma) \times \calA(\bfGamma)$ to the state space
\[
 V_\eend(\bfSigma'_{\bfGamma} \circ \bfSigma_{\bfGamma}) \in \FVect_\Bbbk,
\]
and sending every morphism $(\bfM'_{\bfGamma},\bfM_{\bfGamma}) \in \calA'(\bfGamma)(\bfSigma'_{\bfGamma},\bfSigma'''_{\bfGamma}) \times \calA(\bfGamma)(\bfSigma_{\bfGamma},\bfSigma''_{\bfGamma})$ to the linear operator
\[
 V_\eend(\bfM'_{\bfGamma} \circ \bfM_{\bfGamma}) : V_\eend(\bfSigma'_{\bfGamma} \circ \bfSigma_{\bfGamma}) \to V_\eend(\bfSigma'''_{\bfGamma} \circ \bfSigma''_{\bfGamma}).
\]
The annihilator of $\calA'(\bfGamma)$ with respect to $\langle \_ , \_ \rangle_{\bfGamma}$ is the linear congruence $\calA'(\bfGamma)^{\perp}$ on $\calA(\bfGamma)$ defined by
\[
 \calA'(\bfGamma)^{\perp}(\bfSigma_{\bfGamma},\bfSigma''_{\bfGamma}) := \{ \bfM_{\bfGamma} \in \calA(\bfGamma)(\bfSigma_{\bfGamma},\bfSigma''_{\bfGamma}) \mid \langle \id_{\bfSigma'_{\bfGamma}},\bfM_{\bfGamma} \rangle = 0 \Forall \bfSigma'_{\bfGamma} \in \calA'(\bfGamma) \},
\]
and analogously the annihilator of $\calA(\bfGamma)$ with respect to $\langle \_ , \_ \rangle_{\bfGamma}$ is the linear congruence $\calA(\bfGamma)^{\perp}$ on $\calA'(\bfGamma)$ defined by
\[
 \calA(\bfGamma)^{\perp}(\bfSigma'_{\bfGamma},\bfSigma'''_{\bfGamma}) := \{ \bfM'_{\bfGamma} \in \calA'(\bfGamma)(\bfSigma'_{\bfGamma},\bfSigma'''_{\bfGamma}) \mid \langle \bfM'_{\bfGamma},\id_{\bfSigma_{\bfGamma}} \rangle = 0 \Forall \bfSigma_{\bfGamma} \in \calA(\bfGamma) \}.
\]
Here, the term congruence is used in the sense of \cite[Section~II.8]{M71}. Then, the image of $\bfGamma$ under $\bfA_\eend$ is defined to be the linear category
\[
 \bfA_\eend(\bfGamma) := \calA(\bfGamma)/\calA'(\bfGamma)^{\perp},
\]
while the image of $\bfGamma$ under $\bfA'_\eend$ is defined to be the linear category
\[
 \bfA'_\eend(\bfGamma) := \calA'(\bfGamma)/\calA(\bfGamma)^{\perp}.
\]
Here, recall that the quotient of a linear category with respect to a linear congruence is simply defined as the linear category with the same set of objects, and with vector spaces of morphisms obtained as quotients of the original ones with respect to the linear subspaces determined by the linear congruence.

Next, let us consider a \mrphsm{1} $\bfSigma : \bfGamma \to \bfGamma'$ of $\bfACob_\calC$. The image of $\bfSigma$ under $\bfA_\eend$ is defined to be the linear functor
\[
 \bfA_\eend(\bfSigma) : \bfA_\eend(\bfGamma) \to \bfA_\eend(\bfGamma')
\]
sending every object $\bfSigma_{\bfGamma} \in \bfA_\eend(\bfGamma)$ to the object
\[
 \bfSigma \circ \bfSigma_{\bfGamma} \in \bfA_\eend(\bfGamma'),
\]
and every morphism $\left[ \bfM_{\bfGamma} \right] \in \bfA_\eend(\bfGamma)(\bfSigma_{\bfGamma},\bfSigma''_{\bfGamma})$ to the morphism
\[
 \left[ \bfSigma \triangleright \bfM_{\bfGamma} \right] \in \bfA_\eend(\bfGamma')(\bfSigma \circ \bfSigma_{\bfGamma},\bfSigma \circ \bfSigma''_{\bfGamma}),
\]
while the image of $\bfSigma$ under $\bfA'_\eend$ is defined to be the linear functor
\[
 \bfA'_\eend(\bfSigma) : \bfA'_\eend(\bfGamma') \to \bfA'_\eend(\bfGamma)
\]
sending every object $\bfSigma'_{\bfGamma'} \in \bfA'_\eend(\bfGamma')$ to the object
\[
 \bfSigma'_{\bfGamma'} \circ \bfSigma \in \bfA'_\eend(\bfGamma),
\]
and every morphism $\left[ \bfM'_{\bfGamma'} \right] \in \bfA'_\eend(\bfGamma')(\bfSigma'_{\bfGamma'},\bfSigma'''_{\bfGamma'})$ to the morphism
\[
 \left[ \bfM'_{\bfGamma'} \triangleleft \bfSigma \right] \in \bfA'_\eend(\bfGamma)(\bfSigma'_{\bfGamma'} \circ \bfSigma, \bfSigma'''_{\bfGamma'} \circ \bfSigma).
\]

Finally, let us consider a \mrphsm{2} $\bfM : \bfSigma \Rightarrow \bfSigma'$ of $\bfACob_\calC$ between \mrphsms{1} $\bfSigma, \bfSigma' : \bfGamma \to \bfGamma'$. The image of $\bfM$ under $\bfA_\eend$ is defined to be the natural transformation
\[
 \bfA_\eend(\bfM) : \bfA_\eend(\bfSigma) \Rightarrow \bfA_\eend(\bfSigma')
\]
associating with every object $\bfSigma_{\bfGamma} \in \bfA_\eend(\bfGamma)$ the morphism
\[
 \left[ \bfM \triangleleft \bfSigma_{\bfGamma} \right] \in \bfA_\eend(\bfGamma')(\bfSigma \circ \bfSigma_{\bfGamma},\bfSigma' \circ \bfSigma_{\bfGamma}),
\]
while the image of $\bfM$ under $\bfA'_\eend$ is defined to be the natural transformation
\[
 \bfA'_\eend(\bfM) : \bfA'_\eend(\bfSigma) \Rightarrow \bfA'_\eend(\bfSigma')
\]
associating with every object $\bfSigma'_{\bfGamma'} \in \bfA'_\eend(\bfGamma')$ the morphism
\[
 \left[ \bfSigma'_{\bfGamma'} \triangleright \bfM \right] \in \bfA'_\eend(\bfGamma)(\bfSigma'_{\bfGamma'} \circ \bfSigma,\bfSigma'_{\bfGamma'} \circ \bfSigma').
\]

\subsection{Properties of the universal construction}

Next, let us record some fundamental properties shared by every \fnctr{2} obtained from the extended universal construction, such as the existence of a faithful separate linear functor $\bfmu_{\bfGamma,\bfGamma'} : (\bfA_\eend(\bfGamma),\bfA_\eend(\bfGamma')) \to \bfA_\eend(\bfGamma \disjun \bfGamma')$ for all $\bfGamma, \bfGamma' \in \bfACob_\calC$, which provides a lax monoidal structure for $\bfA_\eend$.

Indeed, let us consider the \trnsfrmtn{2}
\[
 \bfmu : \otimes \circ \left( \bfA_\eend \times \bfA_\eend \right) \Rightarrow \bfA_\eend \circ \disjun
\]
defined as follows: for all objects $\bfGamma$ and $\bfGamma'$ of $\bfACob_\calC$, we specify the separate linear functor
\[
 \bfmu_{\bfGamma,\bfGamma'} : (\bfA_\eend(\bfGamma),\bfA_\eend(\bfGamma')) \to \bfA_\eend(\bfGamma \disjun \bfGamma')
\]
sending every object $(\bfSigma_{\bfGamma},\bfSigma_{\bfGamma'})$ of $\bfA_\eend(\bfGamma) \otimes \bfA_\eend(\bfGamma')$ to the object $\bfSigma_{\bfGamma} \disjun \bfSigma_{\bfGamma'}$ of $\bfA_\eend(\bfGamma \disjun \bfGamma')$, and every morphism
\[
 \left[ \bfM_{\bfGamma} \right] \otimes \left[ \bfM_{\bfGamma'} \right]
\]
of $\bfA_\eend(\bfGamma)(\bfSigma_{\bfGamma},\bfSigma''_{\bfGamma}) \otimes \bfA_\eend(\bfGamma')(\bfSigma_{\bfGamma'},\bfSigma''_{\bfGamma'})$
to the morphism
\[
 \left[ \bfM_{\bfGamma} \disjun \bfM_{\bfGamma'} \right]
\]
of $\bfA_\eend(\bfGamma \disjun \bfGamma')(\bfSigma_{\bfGamma} \disjun \bfSigma_{\bfGamma'},\bfSigma''_{\bfGamma} \disjun \bfSigma''_{\bfGamma'})$. Next, for all \mrphsms{1} $\bfSigma : \bfGamma \to \bfGamma''$ and $\bfSigma' : \bfGamma' \to \bfGamma'''$ of $\bfACob_\calC$, we specify the natural transformation
\[
 \bfmu_{\bfSigma,\bfSigma'} : \bfA_\eend(\bfSigma \disjun \bfSigma') \circ \bfmu_{\bfGamma,\bfGamma'} \Rightarrow \bfmu_{\bfGamma'',\bfGamma'''} \circ \left( \bfA_\eend(\bfSigma) \otimes \bfA_\eend(\bfSigma') \right)
\]
associating with every object $(\bfSigma_{\bfGamma},\bfSigma_{\bfGamma'})$ of $\bfA_\eend(\bfGamma) \otimes \bfA_\eend(\bfGamma')$ the morphism
\[
 [\sqcup_{(\varnothing,\varnothing),(\bfGamma,\bfGamma'),(\bfGamma'',\bfGamma''')}]
\]
of $\bfA_\eend(\bfGamma'' \disjun \bfGamma''')((\bfSigma \disjun \bfSigma') \circ (\bfSigma_{\bfGamma} \disjun \bfSigma_{\bfGamma'}),{(\bfSigma \circ \bfSigma_{\bfGamma})} \disjun {(\bfSigma' \circ \bfSigma_{\bfGamma'})})$, where the coherence \mrphsm{2} $\sqcup_{(\varnothing,\varnothing),(\bfGamma,\bfGamma'),(\bfGamma'',\bfGamma''')}$ is provided by the monoidal structure of the \ctgr{2} $\bfACob_\calC$.

\begin{proposition}\label{P:lax_monoidality_ETQFT}
 For all objects $\bfGamma, \bfGamma' \in \bfACob_\calC$, the separate linear functor $\bfmu_{\bfGamma,\bfGamma'} : (\bfA_\eend(\bfGamma),\bfA_\eend(\bfGamma')) \to \bfA_\eend(\bfGamma \disjun \bfGamma')$ is faithful.
\end{proposition}

The proof of Proposition~\ref{P:lax_monoidality_ETQFT} is identical to the one of \cite[Proposition 5.1]{D21}, up to some minor change of notation (for instance, the enriched tensor product of linear categories is denoted $\sqtimes$ in \cite{D21}, while it is denoted $\otimes$ here, since we reserve the former notation for Kelly's box tensor product of small finitely complete linear categories).

\begin{remark}\label{R:mono-epi}
 Let $\bfGamma \in \bfACob_\calC$ be an object, let $\bfSigma_{\bfGamma}, \bfSigma''_{\bfGamma} \in \bfA_\eend(\bfGamma)$ be objects, and let $[\bfM_{\bfGamma}] \in \bfA_\eend(\bfGamma)(\bfSigma_{\bfGamma},\bfSigma''_{\bfGamma})$ be a morphism. As a direct consequence of the extended universal construction, if, for every object $\bfSigma'_{\bfGamma} \in \bfA'_\eend(\bfGamma)$, the linear map $V_\eend(\bfSigma'_{\bfGamma} \triangleright \bfM_{\bfGamma}) \in \Hom_\Bbbk(V_\eend(\bfSigma'_{\bfGamma} \circ \bfSigma_{\bfGamma}),V_\eend(\bfSigma'_{\bfGamma} \circ \bfSigma''_{\bfGamma}))$ is injective or surjective, then $[\bfM_{\bfGamma}] \in \bfA_\eend(\bfGamma)(\bfSigma_{\bfGamma},\bfSigma''_{\bfGamma})$ is a monomorphism or an epimorphism, respectively.
\end{remark}

\begin{lemma}\label{L:split}
 For all objects $\bfSigma_\varnothing,\bfSigma''_\varnothing \in \bfA_\eend(\varnothing)$, the linear map
 \begin{align*}
  v_{\bfSigma_\varnothing,\bfSigma''_\varnothing} : \bfA_\eend(\varnothing)(\bfSigma_\varnothing,\bfSigma''_\varnothing) &\to V_\eend(\bfSigma''_\varnothing) \otimes V'_\eend(\bfSigma_\varnothing) \\*
  [\bfM_\varnothing] &\mapsto t_{\bfSigma_\varnothing,\bfSigma''_\varnothing}^{-1}(V_\eend(\bfM_\varnothing))
 \end{align*}
 induced by Equation~\eqref{E:linear_maps_tensor_product} is an isomorphism, whose inverse is
 \begin{align*}
  a_{\bfSigma_\varnothing,\bfSigma''_\varnothing} : V_\eend(\bfSigma''_\varnothing) \otimes V'_\eend(\bfSigma_\varnothing) &\to \bfA_\eend(\varnothing)(\bfSigma_\varnothing,\bfSigma''_\varnothing) \\*
  [\bfM_{\bfSigma''_\varnothing}] \otimes [\bfM'_{\bfSigma_\varnothing}] &\mapsto [\bfM_{\bfSigma''_\varnothing} \ast \bfM'_{\bfSigma_\varnothing}].
 \end{align*}
\end{lemma}

\begin{proof}
 On one hand, if
 \[
  0 = \sum_{i=1}^m [\bfM_{\varnothing,i}] \in \bfA_\eend(\varnothing)(\bfSigma_\varnothing,\bfSigma''_\varnothing),
 \]
 then
 \[
  \sum_{i=1}^m V_\eend(\bfSigma'_\varnothing \triangleright \bfM_{\varnothing,i}) = 0
 \]
 for every object $\bfSigma'_\varnothing \in \bfA'_\eend(\varnothing)$, so this holds in particular for $\id_\varnothing \in \bfA'_\eend(\varnothing)$. This means that $v_{\bfSigma_\varnothing,\bfSigma''_\varnothing}$ is well-defined.

 On the other hand, if
 \[
  0 = \sum_{i=1}^m V_\eend(\bfM_{\varnothing,i}) \in \Hom_\Bbbk(V_\eend(\bfSigma_\varnothing),V_\eend(\bfSigma''_\varnothing)),
 \]
 then
 \[
  \sum_{i=1}^m J'_\eend(\bfM'_{\bfSigma''_\varnothing} \ast \bfM_{\varnothing,i} \ast \bfM_{\bfSigma_\varnothing}) = 0
 \]
 for all $[\bfM_{\bfSigma_\varnothing}] \in V_\eend(\bfSigma_\varnothing)$ and $[\bfM'_{\bfSigma''_\varnothing}] \in V'_\eend(\bfSigma''_\varnothing)$. Since for all $\bfSigma'_\varnothing \in \bfA'_\eend(\varnothing)$ and $[\bfM_{\bfSigma'_\varnothing \circ \bfSigma_\varnothing}] \in V_\eend(\bfSigma'_\varnothing \circ \bfSigma_\varnothing) = V_\eend(\bfSigma_\varnothing \disjun \bfSigma'_\varnothing) \cong V_\eend(\bfSigma_\varnothing) \otimes V_\eend(\bfSigma'_\varnothing)$ there exist $[\bfM_{\bfSigma_\varnothing,1}],\ldots,[\bfM_{\bfSigma_\varnothing,n}] \in V_\eend(\bfSigma_\varnothing)$ and $[\bfM_{\bfSigma'_\varnothing,1}],\ldots,[\bfM_{\bfSigma'_\varnothing,n}] \in V_\eend(\bfSigma'_\varnothing)$ such that
 \[
  [\bfM_{\bfSigma'_\varnothing \circ \bfSigma_\varnothing}]
  = \sum_{j=1}^n [\bfM_{\bfSigma_\varnothing,j} \disjun \bfM_{\bfSigma'_\varnothing,j}]
  = \sum_{j=1}^n [\bfM_{\bfSigma'_\varnothing,j} \circ \bfM_{\bfSigma_\varnothing,j}],
 \]
 we have
 \begin{align*}
  &\sum_{i=1}^m J'_\eend (\bfM'_{\bfSigma'_\varnothing \circ \bfSigma''_\varnothing} \ast (\bfSigma'_\varnothing \triangleright \bfM_{\varnothing,i}) \ast \bfM_{\bfSigma'_\varnothing \circ \bfSigma_\varnothing}) \\
  &= \sum_{i=1}^m \sum_{j=1}^n J'_\eend (\bfM'_{\bfSigma'_\varnothing \circ \bfSigma''_\varnothing} \ast (\bfSigma'_\varnothing \triangleright \bfM_{\varnothing,i}) \ast (\bfM_{\bfSigma'_\varnothing,j} \circ \bfM_{\bfSigma_\varnothing,j})) \\
  &= \sum_{i=1}^m \sum_{j=1}^n J'_\eend ((\bfM'_{\bfSigma'_\varnothing \circ \bfSigma''_\varnothing} \ast (\bfM_{\bfSigma'_\varnothing,j} \triangleleft \bfSigma''_\varnothing)) \ast \bfM_{\varnothing,i} \ast \bfM_{\bfSigma_\varnothing,j}) \\
  &= 0
 \end{align*}
 for all $[\bfM'_{\bfSigma'_\varnothing \circ \bfSigma''_\varnothing}] \in V'_\eend(\bfSigma'_\varnothing \circ \bfSigma''_\varnothing)$. This means that $a_{\bfSigma_\varnothing,\bfSigma''_\varnothing}$ is well-defined too, and that $v_{\bfSigma_\varnothing,\bfSigma''_\varnothing}$ is injective.

 Since the composition $v_{\bfSigma_\varnothing,\bfSigma''_\varnothing} \circ a_{\bfSigma_\varnothing,\bfSigma''_\varnothing}$ is the identity of $V_\eend(\bfSigma''_\varnothing) \otimes V'_\eend(\bfSigma_\varnothing)$, we can conclude.
\end{proof}

In particular, morphism spaces of $\bfA_\eend(\varnothing)$ are finite-dimensional, and
\begin{align*}
 \bfA_\eend(\varnothing)(\id_\varnothing,\bfSigma_\varnothing) &\cong V_\eend(\bfSigma_\varnothing), &
 \bfA_\eend(\varnothing)(\bfSigma_\varnothing,\id_\varnothing) &\cong V'_\eend(\bfSigma_\varnothing)
\end{align*}
for every object $\bfSigma_\varnothing \in \bfA_\eend(\varnothing)$.

\subsection{Admissibility and exactness}

Now, let us establish a few immediate consequences of admissibility and coadmissibility of \mrphsms{1} of $\bfACob_\calC$ for what concerns left and right exactness of linear functors in the image of $\bfA_\eend$.

Indeed, let $\bfSigma : \bfGamma \to \bfGamma'$ be a \mrphsm{1} of $\bfCob_\calC$, and let $\bfSigma^\dagger : \bfGamma' \to \bfGamma$ denote its \textit{adjoint \mrphsm{1}} in $\bfCob_\calC$, which is defined by Equation~\eqref{E:adjoint_1-morphism}. If $\bfSigma : \bfGamma \to \bfGamma'$ is admissible, then $\bfSigma^\dagger : \bfGamma' \to \bfGamma$ is right adjoint to $\bfSigma : \bfGamma \to \bfGamma'$ in $\bfACob_\calC$, because the right adjunction counit $\radcoun_{\bfSigma} : \bfSigma \circ \bfSigma^\dagger \Rightarrow \id_{\bfGamma'}$ defined by Equation~\eqref{E:right_adjunction_counit} is admissible. Similarly, if $\bfSigma : \bfGamma \to \bfGamma'$ is coadmissible, then $\bfSigma^\dagger : \bfGamma' \to \bfGamma$ is left adjoint to $\bfSigma : \bfGamma \to \bfGamma'$ in $\bfACob_\calC$, because the left adjunction counit $\ladcoun_{\bfSigma} : \bfSigma^\dagger \circ \bfSigma \Rightarrow \id_{\bfGamma}$ defined by Equation~\eqref{E:left_adjunction_counit} is admissible. Notice that both the right adjunction unit $\radun_{\bfSigma} : \id_{\bfGamma} \Rightarrow \bfSigma^\dagger \circ \bfSigma$ and the right adjunction unit $\ladun_{\bfSigma} : \id_{\bfGamma'} \Rightarrow \bfSigma \circ \bfSigma^\dagger$ are always admissible. This immediately implies the following result.

\begin{lemma}\label{L:adjoint}
 Let $\bfGamma \in \bfACob_\calC$ be an object.
 \begin{enumerate}
  \item If $\bfSigma_{\bfGamma} \in \bfA_\eend(\bfGamma)$ is admissible, then $\bfA_\eend((\bfSigma_{\bfGamma})^\dagger) : \bfA_\eend(\bfGamma) \to \bfA_\eend(\varnothing)$ is right adjoint to $\bfA_\eend(\bfSigma_{\bfGamma}) : \bfA_\eend(\varnothing) \to \bfA_\eend(\bfGamma)$, so in particular the linear isomorphism
   \begin{align*}
    (h_{\bfSigma_{\bfGamma}})_{\bfSigma_\varnothing,\bfSigma''_{\bfGamma}} : \bfA_\eend(\bfGamma)(\bfSigma_{\bfGamma} \circ \bfSigma_\varnothing,\bfSigma''_{\bfGamma}) &\to \bfA_\eend(\varnothing)(\bfSigma_\varnothing,(\bfSigma_{\bfGamma})^\dagger \circ \bfSigma''_{\bfGamma}) \\*
    [\bfM_{\bfGamma}] &\mapsto [((\bfSigma_{\bfGamma})^\dagger \triangleright \bfM_{\bfGamma}) \ast (\radun_{\bfSigma_{\bfGamma}} \triangleleft \bfSigma_\varnothing)],
   \end{align*}
   determined by Equation~\eqref{E:right_adjunction_unit} is natural in $\bfSigma''_{\bfGamma} \in \bfA_\eend(\bfGamma)$ and $\bfSigma_\varnothing \in \bfA_\eend(\varnothing)$, with inverse
   \begin{align*}
    (e_{\bfSigma_{\bfGamma}})_{\bfSigma_\varnothing,\bfSigma''_{\bfGamma}} : \bfA_\eend(\varnothing)(\bfSigma_\varnothing,(\bfSigma_{\bfGamma})^\dagger \circ \bfSigma''_{\bfGamma}) &\to \bfA_\eend(\bfGamma)(\bfSigma_{\bfGamma} \circ \bfSigma_\varnothing,\bfSigma''_{\bfGamma}) \\*
    [\bfM_\varnothing] &\mapsto [(\radcoun_{\bfSigma_{\bfGamma}} \triangleleft \bfSigma''_{\bfGamma}) \ast (\bfSigma_{\bfGamma} \triangleright \bfM_\varnothing)]
   \end{align*}
   determined by Equation~\eqref{E:right_adjunction_counit}.
  \item If $\bfSigma''_{\bfGamma} \in \bfA_\eend(\bfGamma)$ is coadmissible, then it is automatically also admissible, and $\bfA_\eend((\bfSigma''_{\bfGamma})^\dagger) : \bfA_\eend(\bfGamma) \to \bfA_\eend(\varnothing)$ is both left and right adjoint to $\bfA_\eend(\bfSigma''_{\bfGamma}) : \bfA_\eend(\varnothing) \to \bfA_\eend(\bfGamma)$, so in particular the linear isomorphism
   \begin{align*}
    (e_{\bfSigma''_{\bfGamma}})_{\bfSigma_{\bfGamma},\bfSigma_\varnothing} : \bfA_\eend(\bfGamma)(\bfSigma_{\bfGamma},\bfSigma''_{\bfGamma} \circ \bfSigma_\varnothing) &\to \bfA_\eend(\varnothing)((\bfSigma''_{\bfGamma})^\dagger \circ \bfSigma_{\bfGamma},\bfSigma_\varnothing) \\*
    [\bfM_{\bfGamma}] &\mapsto [(\ladcoun_{\bfSigma''_{\bfGamma}} \triangleleft \bfSigma_\varnothing) \ast ((\bfSigma''_{\bfGamma})^\dagger \triangleright \bfM_{\bfGamma})],
   \end{align*}
   determined by Equation~\eqref{E:left_adjunction_counit} is natural in $\bfSigma_{\bfGamma} \in \bfA_\eend(\bfGamma)$ and $\bfSigma_\varnothing \in \bfA_\eend(\varnothing)$, with inverse
   \begin{align*}
    (h_{\bfSigma''_{\bfGamma}})_{\bfSigma_{\bfGamma},\bfSigma_\varnothing} : \bfA_\eend(\varnothing)((\bfSigma''_{\bfGamma})^\dagger \circ \bfSigma_{\bfGamma},\bfSigma_\varnothing) &\to \bfA_\eend(\bfGamma)(\bfSigma_{\bfGamma},\bfSigma''_{\bfGamma} \circ \bfSigma_\varnothing) \\*
    [\bfM_\varnothing] &\mapsto [(\bfSigma''_{\bfGamma} \triangleright \bfM_\varnothing) \ast (\ladun_{\bfSigma''_{\bfGamma}} \triangleleft \bfSigma_{\bfGamma})]
   \end{align*}
   determined by Equation~\eqref{E:left_adjunction_unit}.
 \end{enumerate}
\end{lemma}

Furthermore, Remark~\ref{R:adjoints_and_(co)limits}.$(ii)$ immediately implies the following result.

\begin{lemma}\label{L:adjoint_exactness}
 Let $\bfSigma : \bfGamma \to \bfGamma'$ be a \mrphsm{1} of $\bfACob_\calC$. If $\bfSigma : \bfGamma \to \bfGamma'$ is admissible, then the linear functor $\bfA_\eend(\bfSigma) : \bfA_\eend(\bfGamma) \to \bfA_\eend(\bfGamma')$ is right exact, and if $\bfSigma : \bfGamma \to \bfGamma'$ is coadmissible, then the linear functor $\bfA_\eend(\bfSigma) : \bfA_\eend(\bfGamma) \to \bfA_\eend(\bfGamma')$ is left exact.
\end{lemma}

We also establish the following result, which will be used in the proof of Proposition~\ref{P:product_coherence}.

\begin{lemma}\label{L:coadjoint_exactness}
 Let $\bfGamma \in \bfACob_\calC$ be an object, and let $\bfSigma_{\bfGamma} \in \bfA_\eend(\bfGamma)$ be an object. If $\bfSigma_{\bfGamma}$ is coadmissible, then the linear functor
 \[
  \bfA_\eend(\bfGamma)(\_,\bfSigma_{\bfGamma}) : \bfA_\eend(\bfGamma)^\op \to \Vect_\Bbbk
 \]
 is right exact.
\end{lemma}

\begin{proof}
 Thanks to Equation~\eqref{E:adjunction_V_V'} and Lemmas~\ref{L:split} and \ref{L:adjoint}, we have a chain of natural isomorphisms
 \begin{align*}
  \bfA_\eend(\bfGamma)(\_,\bfSigma_{\bfGamma})
  &\cong \bfA_\eend(\varnothing)((\bfSigma_{\bfGamma})^\dagger \circ \_,\id_\varnothing)
  \cong V'_\eend((\bfSigma_{\bfGamma})^\dagger \circ \_)
  \cong (V_\eend((\bfSigma_{\bfGamma})^\dagger \circ \_))^* \\*
  &\cong (\bfA_\eend(\varnothing)(\id_\varnothing,(\bfSigma_{\bfGamma})^\dagger \circ \_))^*
 \end{align*}
 between linear functors from $\bfA_\eend(\bfGamma)^\op$ to $\Vect_\Bbbk$.

 Now, since $(\bfSigma_{\bfGamma})^\dagger$ is coadmissible, Lemma~\ref{L:adjoint_exactness} implies that $\bfA_\eend((\bfSigma_{\bfGamma})^\dagger)$ is left exact. Then, if a finite diagram $D : J \to \bfA_\eend(\bfGamma)$ admits a limit, we have
 \begin{align*}
  \colim_{j \in J^\op} \bfA_\eend(\varnothing)((\bfSigma_{\bfGamma})^\dagger \circ D^\op(j),\id_\varnothing)
  &\cong \colim_{j \in J^\op} \left( \bfA_\eend(\varnothing)(\id_\varnothing,(\bfSigma_{\bfGamma})^\dagger \circ D^\op(j)) \right)^* \\
  &\cong \left( \lim_{j \in J} \bfA_\eend(\varnothing)(\id_\varnothing,(\bfSigma_{\bfGamma})^\dagger \circ D(j)) \right)^* \\
  &\cong \left( \bfA_\eend(\varnothing) \left( \id_\varnothing, \lim_{j \in J} ((\bfSigma_{\bfGamma})^\dagger \circ D(j)) \right) \right)^* \\
  &\cong \left( \bfA_\eend(\varnothing) \left( \id_\varnothing, (\bfSigma_{\bfGamma})^\dagger \circ \left( \lim_{j \in J} D(j) \right) \right) \right)^* \\
  &\cong \bfA_\eend(\varnothing) \left( (\bfSigma_{\bfGamma})^\dagger \circ \left( \lim_{j \in J} D(j) \right), \id_\varnothing \right). \qedhere
 \end{align*}
\end{proof}

\subsection{Generation by boundary bijective cobordisms}

Next, we are going to show that, for every object $\bfGamma \in \bfACob_\calC$ and all objects $\bfSigma_{\bfGamma}, \bfSigma''_{\bfGamma} \in \bfA_\eend(\bfGamma)$, the morphism space $\bfA_\eend(\bfGamma)(\bfSigma_{\bfGamma}, \bfSigma''_{\bfGamma})$ can be generated by admissible skeins inside a single cobordism with corners, as a consequence of Lemma~\ref{L:connected_1}. This implies that the separate linear functor $\bfmu_{\bfGamma,\bfGamma'} : (\bfA_\eend(\bfGamma),\bfA_\eend(\bfGamma')) \to \bfA_\eend(\bfGamma \disjun \bfGamma')$ of Proposition~\ref{P:lax_monoidality_ETQFT} is full for all $\bfGamma, \bfGamma' \in \bfACob_\calC$.

\begin{lemma}\label{L:connected_2}
 Let $\bfSigma_{\bfGamma} = (\Sigma_\Gamma,P_\Gamma,\Lagr_\Gamma)$ and $\bfSigma''_{\bfGamma} = (\Sigma''_\Gamma,P''_\Gamma,\Lagr''_\Gamma)$ be objects of $\bfA_\eend(\bfGamma)$. If $M$ is a boundary bijective \dmnsnl{3} cobordism with corners from $\Sigma_\Gamma$ to $\Sigma''_\Gamma$, then the linear map
 \begin{align*}
  \pi_M : S^\partial_\eend(M;P_\Gamma,P''_\Gamma) &\to \bfA_\eend(\bfGamma)(\bfSigma_{\bfGamma}, \bfSigma''_{\bfGamma}) \\*
  [T] &\mapsto [M,T,0]
 \end{align*}
 is surjective.
\end{lemma}

\begin{proof}
 First of all, up to isomorphism in $\bfACob_\calC$, we can assume that $\bfSigma_{\bfGamma}$ is of the form $\bfSigma'_{\bfGamma} \circ \bfSigma_\varnothing$ for some admissible $\bfSigma'_{\bfGamma} = (\Sigma'_\Gamma,P'_\Gamma,\Lagr'_\Gamma) \in \bfA_\eend(\bfGamma)$ and some $\bfSigma_\varnothing = (\Sigma_\varnothing,P_\varnothing,\Lagr_\varnothing) \in \bfA_\eend(\varnothing)$. Then consider the \dmnsnl{3} cobordism
 \[
   h(M) = ((\Sigma'_\Gamma)^\dagger \triangleright M) \ast (\radun_{\Sigma'_\Gamma} \triangleleft \Sigma_\varnothing).
 \]
 from $\Sigma_\varnothing$ to $(\Sigma'_\Gamma)^\dagger \circ \Sigma''_\Gamma$. Notice that the linear map
 \begin{align*}
  S^\partial_\eend(M;P'_\Gamma \cup P_\varnothing,P''_\Gamma) &\to S^\partial_\eend(h(M);P_\varnothing,(P'_\Gamma)^\dagger \cup P''_\Gamma) \\*
  [T] &\mapsto [(\id_{(P'_\Gamma)^\dagger} \cup T) \ast (\radun_{P'_\Gamma} \cup \id_{P_\varnothing})]
 \end{align*}
 is an isomorphism, because $\bfSigma'_{\bfGamma}$ is admissible.
 Since $M$ is boundary bijective, then
 $h(M)$ is boundary bijective too. Therefore, Lemma~\ref{L:connected_1} implies that the linear map
 \[
  \pi_{h(M)} : S^\partial_\eend(h(M);P_\varnothing,(P'_\Gamma)^\dagger \cup P''_\Gamma) \to V_\eend((\bfSigma'_{\bfGamma})^\dagger \circ \bfSigma''_{\bfGamma}) \otimes V'_\eend(\bfSigma_\varnothing)
 \]
 is surjective, and Lemmas~\ref{L:split} and \ref{L:adjoint} imply that
 \[
  V_\eend((\bfSigma'_{\bfGamma})^\dagger \circ \bfSigma''_{\bfGamma}) \otimes V'_\eend(\bfSigma_\varnothing) \cong \bfA_\eend(\bfGamma)(\bfSigma'_{\bfGamma} \circ \bfSigma_\varnothing,\bfSigma''_{\bfGamma}). \qedhere
 \]

\end{proof}

\begin{proposition}\label{P:fullness}
 For all objects $\bfGamma, \bfGamma' \in \bfACob_\calC$, the separate linear functor $\bfmu_{\bfGamma,\bfGamma'} : (\bfA_\eend(\bfGamma),\bfA_\eend(\bfGamma')) \to \bfA_\eend(\bfGamma \disjun \bfGamma')$ is full.
\end{proposition}

\begin{proof}
 Let $\bfSigma_{\bfGamma} = (\Sigma_\Gamma,P_\Gamma,\Lagr_\Gamma)$ and $\bfSigma''_{\bfGamma} = (\Sigma''_\Gamma,P''_\Gamma,\Lagr''_\Gamma)$ be objects of $\bfA_\eend(\bfGamma)$, and let $\bfSigma_{\bfGamma'} = (\Sigma_{\Gamma'},P_{\Gamma'},\Lagr_{\Gamma'})$ and $\bfSigma''_{\bfGamma'} = (\Sigma''_{\Gamma'},P''_{\Gamma'},\Lagr''_{\Gamma'})$ be objects of $\bfA_\eend(\bfGamma')$. If $M$ is a boundary bijective \dmnsnl{3} cobordism with corners from $\Sigma_\Gamma$ to $\Sigma''_\Gamma$, and $M'$ is a boundary bijective \dmnsnl{3} cobordism with corners from $\Sigma_{\Gamma'}$ to $\Sigma''_{\Gamma'}$, then $M \sqcup M'$ is a boundary bijective \dmnsnl{3} cobordism with corners from $\Sigma_\Gamma \sqcup \Sigma_{\Gamma'}$ to $\Sigma''_\Gamma \sqcup \Sigma''_{\Gamma'}$. Then, Lemma~\ref{L:connected_2} implies that $\pi_{M \sqcup M'}$ is surjective. Since
 \[
  S^\partial_\eend(M \sqcup M';P_\Gamma \sqcup P_{\Gamma'},P''_\Gamma \sqcup P''_{\Gamma'}) \cong S^\partial_\eend(M;P_\Gamma,P''_\Gamma) \otimes S^\partial_\eend(M';P_{\Gamma'},P''_{\Gamma'}),
 \]
 this means that the linear map from $\bfA_\eend(\bfGamma)(\bfSigma_{\bfGamma}, \bfSigma''_{\bfGamma}) \otimes \bfA_\eend(\bfGamma')(\bfSigma_{\bfGamma'}, \bfSigma''_{\bfGamma'})$ to $\bfA_\eend(\bfGamma \disjun \bfGamma')(\bfSigma_{\bfGamma} \disjun \bfSigma_{\bfGamma'}, \bfSigma''_{\bfGamma} \disjun \bfSigma''_{\bfGamma'})$ defined by the separate linear functor $\bfmu_{\bfGamma,\bfGamma'}$ is surjective. \qedhere
\end{proof}

\subsection{Identification of the circle category}

Now, let us identify explicitly the circle category $\bfA_\eend(\bfS^1)$ with $\calC$. In order to do this, let us consider the functor
\[
 \bfD^2 : \calB_\calC \to \bfACob_\calC(\varnothing,\bfS^1)
\]
sending every object $P$ of $\calB_\calC$ to the \mrphsm{1} $\bfD^2(P) : \varnothing \to \bfS^1$ of $\bfACob_\calC$ given by Equation~\eqref{E:disc_object},
and sending every morphism $T$ of $\calB_\calC(P,P')$ to the \mrphsm{2} $\bfD^2(T) : \bfD^2(P) \Rightarrow \bfD^2(P')$ of $\bfACob_\calC$ given by (the equivalence class of)
\begin{equation}\label{E:disc_functor_graphs}
 \bfD^2(T) := (D^2 \times I,T,0).
\end{equation}
We use a similar notation for the linear functor
\[
 \bfD^2 : \calC \to \bfA_\eend(\bfS^1)
\]
sending every object $X$ of $\calC$ to the object $\bfD^2_X := \bfD^2(O_X)$ of $\bfA_\eend(\bfS^1)$, where $O_X$ denotes the regular $\calC$-labeled blue set determined by the center of $D^2$ with positive orientation and label $X$, and sending every morphism $f$ of $\calC(X,Y)$ to the morphism $[\bfD^2_f] := [\bfD^2(O_f)]$ of $\bfA_\eend(\bfS^1)(\bfD^2_X,\bfD^2_Y)$, where $O_f$ denotes the $\calC$-labeled ribbon graph represented here below.
\begin{equation}\label{E:disc_functor_C}
 O_f := \pic{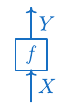}
\end{equation}
Notice that the family of ribbon graphs
\[
 \varphi = \{ \varphi_P \in \calB_\calC(P,O_{F_\calC(P)}) \mid P \in \calB_\calC \}
\]
given by
\begin{align*}
 \varphi_P &:= \pic{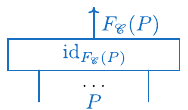}
\end{align*}
defines a natural isomorphism
\begin{center}
 \begin{tikzpicture}[descr/.style={fill=white}]
  \node (P0) at ({sqrt(2)},0) {$\bfA_\eend(\bfS^1)$};
  \node (P1) at (0,{sqrt(2)}) {$\bfACob_\calC(\varnothing,\bfS^1)$};
  \node (P2) at ({-sqrt(2)},0) {$\calB_\calC$};
  \node (P3) at (0,{-sqrt(2)}) {$\calC$};
   \node (P4) at (0,0) {$\phantom{\varphi} \Downarrow \varphi$};
  \draw
  (P1) edge[->] node[right,yshift=5pt] {$\bfA_\eend$} (P0)
  (P2) edge[->] node[left,yshift=5pt] {$\bfD^2$} (P1)
  (P2) edge[->] node[left,yshift=-5pt] {$F_\eend$} (P3)
  (P3) edge[->] node[right,yshift=-5pt] {$\bfD^2$} (P0);
 \end{tikzpicture}
\end{center}

The next result is formulated in terms of the canonical action of $\FVect_\Bbbk$ on $\calC$, defined by Equation~\eqref{E:canonical_Vect-action}.

\begin{lemma}\label{L:connected_objects}
 For every object $\bfSigma_\varnothing$ of $\bfA_\eend(\varnothing)$, the objects $\bfD^2_{\one} \circ \bfSigma_\varnothing$ and $\bfD^2_{V_\eend(\bfSigma_\varnothing) \otimes \one}$ of $\bfA_\eend(\bfS^1)$ are isomorphic.
\end{lemma}

\begin{proof}
 If
 \[
  \{ [\bfM_{\varnothing,i}] \in V_\eend(\bfSigma_\varnothing) \mid i \in I \}, \qquad
  \{ [\bfM'_{\varnothing,i}] \in V'_\eend(\bfSigma_\varnothing) \mid i \in I \}
 \]
 are dual bases of $V_\eend(\bfSigma_\varnothing)$ and $V'_\eend(\bfSigma_\varnothing)$ with respect to the non-degenerate pairing $\langle \_,\_ \rangle_{\bfSigma_\varnothing}$ of Equation~\eqref{E:adjunction_V_V'}, which is induced by composition, then we have
 \begin{align*}
  \sum_{i \in I} [\bfM_{\varnothing,i} \ast \bfM'_{\varnothing,i}] &= \id_{\bfSigma_\varnothing}, &
  [\bfM'_{\varnothing,i} \ast \bfM_{\varnothing,j}] &= \delta_{i,j} \id_{\id_\varnothing}
 \end{align*}
 in $\bfA_\eend(\varnothing)$. Similarly, if $\pi_i : V_\eend(\bfSigma_\varnothing) \otimes \one \to \one$ and $\iota_i : \one \to V_\eend(\bfSigma_\varnothing) \otimes \one$ are the corresponding projection and inclusion morphisms, then we have
 \begin{align*}
  \sum_{i \in I} \iota_i \circ \pi_i &= \id_{V_\eend(\bfSigma_\varnothing) \otimes \one}, &
  \pi_i \circ \iota_j &= \delta_{i,j} \id_{\one}
 \end{align*}
 in $\calC$. Then, thanks to a skein equivalence,
 \begin{align*}
  \sum_{i \in I} [\bfD^2_{\pi_i} \circ \bfM_{\varnothing,i}] &
 \end{align*}
 is an isomorphism, whose inverse is
 \begin{align*}
  \sum_{i \in I} [\bfD^2_{\iota_i} \circ \bfM'_{\varnothing,i}]. & \qedhere
 \end{align*}
\end{proof}

Let $\bfGamma = \Gamma$ and $\bfGamma' = \Gamma'$ be objects of $\bfACob_\calC$. A \mrphsm{1} $\bfSigma = (\Sigma,P,\Lagr)$ from $\bfGamma$ to $\bfGamma'$ in $\bfACob_\calC$ is \textit{target surjective} if the embedding
\[
 \iota : \Gamma' \hookrightarrow \Sigma
\]
induced by the outgoing boundary parametrization determines a surjection
\[
 \iota_* : \pi_0(\Gamma') \twoheadrightarrow \pi_0(\Sigma).
\]
Notice that target surjective \mrphsms{1} are always admissible.

Let $\bfGamma$ be an object of $\bfACob_\calC$. We denote by $\TS_\eend(\bfGamma)$ the full linear subcategory of $\bfA_\eend(\bfGamma)$ whose objects are target surjective.

Let $\bfSigma : \bfGamma \to \bfGamma'$ be a \mrphsm{1} of $\bfACob_\calC$. If $\bfSigma : \bfGamma \to \bfGamma'$ is target surjective, then $\bfA_\eend(\bfSigma) : \bfA_\eend(\bfGamma) \to \bfA_\eend(\bfGamma')$ restricts to a linear functor
\[
 \bfA_\eend(\bfSigma) : \TS_\eend(\bfGamma) \to \TS_\eend(\bfGamma').
\]

\begin{lemma}\label{L:ts_equivalent}
 If $\bfGamma \in \bfACob_\calC$ is non-empty, then the linear category $\TS_\eend(\bfGamma)$ is equivalent to $\bfA_\eend(\bfGamma)$.
\end{lemma}

\begin{proof}
 Let $\bfSigma_{\bfGamma} \in \bfA_\eend(\bfGamma)$ be an object. Up to isomorphism in $\bfACob_\calC$, we can assume that $\bfSigma_{\bfGamma}$ is of the form $\bfSigma \circ \bfD^2 \circ \bfSigma_\varnothing$ for some $\bfSigma_\varnothing \in \bfA_\eend(\varnothing)$ and some target surjective \mrphsm{1} $\bfSigma : \bfS^1 \to \bfGamma$ of $\bfACob_\calC$. Then, since $\bfD^2$ is isomorphic to $\bfD^2_{\one}$ in $\bfA_\eend(\bfS^1)$, Lemma~\ref{L:connected_objects} yields an isomorphism between $\bfSigma \circ \bfD^2 \circ \bfSigma_\varnothing \in \bfA_\eend(\bfGamma)$ and $\bfSigma \circ \bfD^2_{V_\eend(\bfSigma_\varnothing) \otimes \one} \in \TS_\eend(\bfGamma)$. \qedhere
\end{proof}

Let $\bfSigma,\bfSigma' : \bfGamma \to \bfGamma'$ be \mrphsms{1} of $\bfCob_\calC$, let $\bfM : \bfSigma \Rightarrow \bfSigma'$ be a \mrphsm{2} of $\bfCob_\calC$, and let $\bfM^\rmT : (\bfSigma')^\rmT \Rightarrow \bfSigma^\rmT$ denote its \textit{transpose \mrphsm{2}} in $\bfCob_\calC$, which is defined by Equation~\eqref{E:transpose_2-morphism}. If $\bfGamma = \varnothing$ and both $\bfSigma$ and $\bfSigma'$ are target surjective, then $\bfM^\rmT$ is admissible if and only if $\bfM$ is. Since $\bfM^\rmT$ is diffeomorphic to $\bfM$, this immediately implies the following result.

\begin{lemma}\label{L:transpose_non-empty}
 For every object $\bfGamma \in \bfACob_\calC$, the linear functor
 \[
  \_^\rmT : \TS_\eend(\bfGamma)^\op \to \TS_\eend(\bfGamma^*)
 \]
 defined by Equations~\eqref{E:transpose_1-morphism} and \eqref{E:transpose_2-morphism} is an equivalence, whose inverse is the opposite of the functor
 \[
  \_^\rmT : \TS_\eend(\bfGamma^*)^\op \to \TS_\eend(\bfGamma).
 \]
\end{lemma}

We are now ready to prove the following result.

\begin{proposition}\label{P:circle_category}
 The linear functor $\bfD^2 : \calC \to \bfA_\eend(\bfS^1)$ is an equivalence.
\end{proposition}

\begin{proof}
 The functor $\bfD^2$ is fully faithful, because
 \begin{align*}
  \bfD^2 : \calC(X,Y) &\to \bfA_\eend(\bfS^1)(\bfD^2_X,\bfD^2_Y) \\*
  f &\mapsto [\bfD^2_f]
 \end{align*}
 is a linear isomorphism, thanks to Proposition~\ref{P:algebraic_and_skein_model} and Lemmas~\ref{L:split} and \ref{L:adjoint}.

 Furthermore, $\bfD^2$ is essentially surjective. Indeed, thanks to Lemmas~\ref{L:connected_objects} and \ref{L:ts_equivalent}, every object of $\bfA_\eend(\bfS^1)$ is isomorphic to a target surjective object. Furthermore, every object of $\TS_\eend(\bfS^1)$ is isomorphic to $\bfSigma_g \bcs \bfD^2(P)$, where $\bfSigma_g$ is the object of $\CCob$ defined by Equation~\eqref{E:surgery_functor_objects}, and $\bfD^2(P)$ is the \mrphsm{1} of $\bfACob_\calC(\varnothing,\bfS^1)$ defined by Equation~\eqref{E:disc_object}. Then, using the same notation of Equations~\eqref{E:algebraic_model} and \eqref{E:skein_model}, let $[\bfM(g,P)] \in \bfA_\eend(\bfS^1)(\bfD^2_{\eend^{\otimes g} \otimes F_\calC(P)},\bfSigma_g \bcs \bfD^2(P))$ denote the morphism given by
 \[
  \bfM(g,P) := (M(\eta^{{} \bcs g}),S(g,P),0),
 \]
 where $S(g,P) \subset M(\eta^{{} \bcs g})$ denotes the $\calC$-labeled bichrome graph
 \begin{equation}\label{E:disc_essentially_surjective_1}
  S(g,P) := \pic{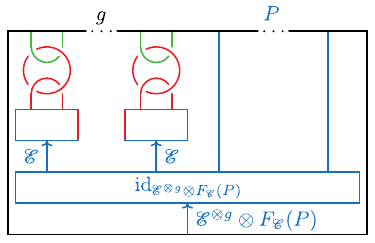}
 \end{equation}
 Then, the morphism $[\bfM'(g,P)] \in \bfA_\eend(\bfS^1)(\bfSigma_g \bcs \bfD^2(P),\bfD^2_{\eend^{\otimes g} \otimes F_\calC(P)})$ given by
 \[
  \bfM'(g,P) := (M(\epsilon^{{} \bcs g}),S'(g,P),0),
 \]
 where $S'(g,P) \subset M(\epsilon^{{} \bcs g})$ denotes the $\calC$-labeled bichrome graph
 \begin{equation}\label{E:disc_essentially_surjective_2}
  S'(g,P) := \pic{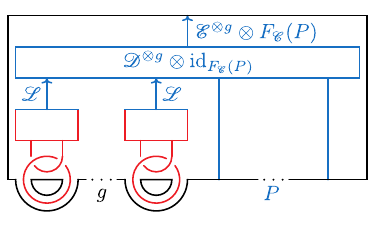}
 \end{equation}
 for the Drinfeld morphism $\calD : \coend \to \eend$ defined by Equation~\eqref{E:Drinfeld}, satisfies
 \begin{align*}
  [\bfM'(g,P) \ast \bfM(g,P)] &= \id_{\bfD^2_{\eend^{\otimes g} \otimes F_\calC(P)}}, &
  [\bfM(g,P) \ast \bfM'(g,P)] &= \id_{\bfSigma_g \bcs \bfD^2(P)}.
 \end{align*}
 Indeed, up to green-to-red moves, slam-dunks, and slides, this follows from Equation~\eqref{E:Drinfeld} and \cite[Lemma~2.10]{DGGPR19}, which yield the skein equivalences
 \[
  \pic{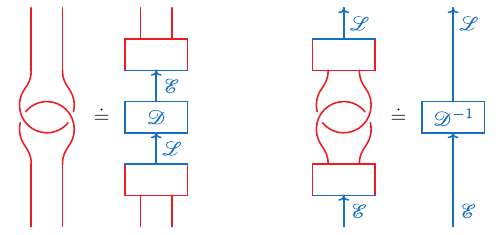}
 \]
\end{proof}

\subsection{Colimit density of coadmissible target surjective objects}

Next, we are going to show that coadmissible objects are finite colimit dense in $\bfA_\eend(\bfGamma)$ for every object $\bfGamma$ of $\bfACob_\calC$. The empty object has to be treated separately, so we start from there.

\begin{lemma}\label{L:empty_additive_generator}
 The object $\id_\varnothing$ of $\bfA_\eend(\varnothing)$ is a finite additive generator of $\bfA_\eend(\varnothing)$.
\end{lemma}

\begin{proof}
 First of all, in $\bfA_\eend(\varnothing)$, we have an isomorphism
 \begin{equation}
  \id_\varnothing \cong (\bfD^2_{P_{\one}})^\dagger \circ \bfD^2_{\one}. \label{E:admissible_non-empty_representative_of_empty}
 \end{equation}
 Indeed, thanks to Proposition~\ref{P:algebraic_and_skein_model} and Lemma~\ref{L:split}, we have
 \begin{align*}
  \bfA_\eend(\varnothing)(\id_\varnothing,(\bfD^2_{P_{\one}})^\dagger \circ \bfD^2_{\one})
  &\cong V_\eend((\bfD^2_{P_{\one}})^\dagger \circ \bfD^2_{\one})
  \cong \calC(P_{\one},\one),
 \end{align*}
 which is generated by the projective cover epimorphism $\epsilon_{\one} : P_{\one} \to \one$, and similarly
 \begin{align*}
  \bfA_\eend(\varnothing)((\bfD^2_{P_{\one}})^\dagger \circ \bfD^2_{\one},\id_\varnothing)
  &\cong V'_\eend((\bfD^2_{P_{\one}})^\dagger \circ \bfD^2_{\one})
  \cong S'_\eend(D^2 \times I;O_{\one},O_{P_{\one}})
  \cong \calC(\one,P_{\one}),
 \end{align*}
 which is generated by the injective hull monomorphism $\eta_{\one} : \one \to P_{\one}$. The specified generators are inverse to each other, thanks to invariance under admissible skein equivalence, and to Equation~\eqref{E:eta_one} and move~\eqref{E:CS}.

 Furthermore, thanks to Proposition~\ref{P:circle_category}, for all objects $X,Y \in \calC$, the direct sum of $(\bfD^2_{P_{\one}})^\dagger \circ \bfD^2_X, (\bfD^2_{P_{\one}})^\dagger \circ \bfD^2_Y \in \bfA_\eend(\varnothing)$ is given by $(\bfD^2_{P_{\one}})^\dagger \circ \bfD^2_{X \oplus Y} \in \bfA_\eend(\varnothing)$, because every linear functor preserves direct sums. Then, thanks to Equation~\eqref{E:canonical_Vect-action_intertwiner} and Lemma~\ref{L:connected_objects}, for every object $\bfSigma_\varnothing \in \bfA_\eend(\varnothing)$ we have isomorphisms
 \begin{align}
  \bfSigma_\varnothing
  &\cong (\bfD^2_{P_{\one}})^\dagger \circ \bfD^2_{\one} \circ \bfSigma_\varnothing
  \cong V_\eend(\bfSigma_\varnothing) \otimes \left( (\bfD^2_{P_{\one}})^\dagger \circ \bfD^2_{\one} \right). \label{E:admissible_connected_representative}
 \end{align}
\end{proof}

Let $\bfGamma$ be an object of $\bfACob_\calC$. We denote by $\CTS_\eend(\bfGamma)$ the full linear subcategory of $\TS_\eend(\bfGamma)$ whose objects are coadmissible.

Let $\bfSigma : \bfGamma \to \bfGamma'$ be a \mrphsm{1} of $\bfACob_\calC$. If $\bfSigma : \bfGamma \to \bfGamma'$ is target surjective and coadmissible, then $\bfA_\eend(\bfSigma) : \TS_\eend(\bfGamma) \to \TS_\eend(\bfGamma')$ restricts to a linear functor
\[
 \bfA_\eend(\bfSigma) : \CTS_\eend(\bfGamma) \to \CTS_\eend(\bfGamma').
\]

\begin{lemma}\label{L:cts_colimit_dense}
 For every object $\bfGamma \in \bfACob_\calC$, the linear category $\CTS_\eend(\bfGamma)$ is finite colimit dense in $\bfA_\eend(\bfGamma)$.
\end{lemma}

\begin{proof}
 Let us separate two cases.

 If $\bfGamma = \varnothing$, then the only target surjective object of $\bfA_\eend(\varnothing)$ is $\id_\varnothing$, which is also coadmissible, so the claim follows from Lemma~\ref{L:empty_additive_generator}.

 If $\bfGamma \neq \varnothing$, let $\bfSigma_{\bfGamma} = (\Sigma_\Gamma,P_\Gamma,\Lagr_\Gamma) \in \bfA_\eend(\bfGamma)$ be an object. First of all, up to isomorphism, we can assume that $\bfSigma_{\bfGamma}$ is target surjective, thanks to Lemma~\ref{L:ts_equivalent}. Then, let $k(\bfSigma_{\bfGamma}) \in \N$ be the number of connected components of $\Sigma_\Gamma$ that violate coadmissibility. If $k(\bfSigma_{\bfGamma}) = 0$, then $\bfSigma_{\bfGamma}$ is coadmissible. If $k(\bfSigma_{\bfGamma}) > 0$, then $\bfSigma_{\bfGamma}$ is isomorphic to $\bfSigma \circ \bfD^2$ in $\bfACob_\calC$, where $\bfSigma : \bfS^1 \to \bfGamma$ is admissible. Thanks to Proposition~\ref{P:circle_category}, $\bfA_\eend(\bfS^1)$ is equivalent to $\calC$, and since $\Proj(\calC)$ is finite colimit dense in $\calC$, then $\bfD^2$ is the colimit of a finite diagram $D : J \to \CTS_\eend(\bfS^1)$. Then, since $\bfA_\eend(\bfSigma) : \bfA_\eend(\bfS^1) \to \bfA_\eend(\bfGamma)$ is right exact, $\bfSigma_{\bfGamma}$ is the colimit of the finite diagram $\bfA_\eend(\bfSigma) \circ D : J \to \bfA_\eend(\bfGamma)$. Since $k(\bfSigma \circ D(j)) < k(\bfSigma_{\bfGamma})$ for every $j \in J$, an induction on $k(\bfSigma_{\bfGamma})$ implies that $\bfSigma_{\bfGamma}$ belongs to the finite colimit closure of $\CTS_\eend(\bfGamma)$. \qedhere
\end{proof}

\subsection{Coevaluation as the (co)end}

Now, let us explain how to interpret the annulus $I \times S^1$ and its opposite as the end of a functor to $\bfA_\eend({\bfS^1} \disjun {(\bfS^1)^*})$ and as the coend of a functor to $\bfA_\eend({(\bfS^1)^*} \disjun {\bfS^1})$.

Indeed, let $\rcoev_{\bfS^1} \in \bfA_\eend({(\bfS^1)^*} \disjun {\bfS^1})$ and $\lcoev_{\bfS^1} \in \bfA_\eend({\bfS^1} \disjun {(\bfS^1)^*})$ denote the objects determined by the right and left coevaluation \mrphsms{1} of $\bfS^1$ defined by Equations~\eqref{E:right_coevaluation} and \eqref{E:left_coevaluation}.

Then, let $[\bfW_X] \in \bfA_\eend({\bfS^1} \disjun {(\bfS^1)^*})(\lcoev_{\bfS^1},{\bfD^2_X} \disjun {(\bfD^2_X)^\rmT})$ denote the morphism given by
\[
 \bfW_X := (W,I_X,0),
\]
where the transpose of a \mrphsm{1} of $\bfCob_\calC$ is defined by Equation~\eqref{E:transpose_1-morphism}, where $W$ denotes the \dmnsnl{3} cobordism with corners from $\lcoev_{S^1}$ to $D^2 \sqcup (D^2)^\rmT$ whose support is
given by $I \times D^2 \subset \R^3$ minus two open balls of radius $\frac{1}{4}$ and centers $(0,0,0)$ and $(1,0,0)$,
and where $I_X \subset W$ denotes the $\calC$-labeled ribbon graph from $\varnothing$ to $O_X \sqcup (O_X)^\dagger$
represented by the following picture.
\[
 \pic{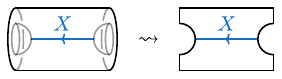}
\]
Notice that, since both $\lcoev_{\bfS^1}$ and ${\bfD^2_X} \disjun {(\bfD^2_X)^\rmT}$ are admissible, we can consider $[\bfW_X^\rmT] \in \bfA_\eend({(\bfS^1)^*} \disjun {\bfS^1})({(\bfD^2_X)^\rmT} \disjun {\bfD^2_X},\rcoev_{\bfS^1})$, where the transpose of a \mrphsm{2} of $\bfCob_\calC$ is defined by Equation~\eqref{E:transpose_2-morphism}.

\begin{proposition}\label{P:end}
 The dinatural family of morphisms
 \[
  [\bfW_X] : \lcoev_{\bfS^1} \to {\bfD^2_X} \disjun {(\bfD^2_X)^\rmT}
 \]
 of $\bfA_\eend({\bfS^1} \disjun {(\bfS^1)^*})$ yields
 \[
  \int_{X \in \calC} {\bfD^2_X} \disjun {(\bfD^2_X)^\rmT} = \lcoev_{\bfS^1},
 \]
 and the dinatural family of morphisms
 \[
  [\bfW_X^\rmT] : {(\bfD^2_X)^\rmT} \disjun {\bfD^2_X} \to \rcoev_{\bfS^1}
 \]
 of $\bfA_\eend({(\bfS^1)^*} \disjun {\bfS^1})$ yields
 \[
  \int^{X \in \calC} {(\bfD^2_X)^\rmT} \disjun {\bfD^2_X} = \rcoev_{\bfS^1}.
 \]
\end{proposition}

\begin{proof}
 Let us show that, for every object
 \[
  \bfSigma_{{\bfS^1} \disjun {(\bfS^1)^*}} = (\Sigma_{S^1 \sqcup (S^1)^*},P_{S^1 \sqcup (S^1)^*},\Lagr_{S^1 \sqcup (S^1)^*}) \in \bfA_\eend({\bfS^1} \disjun {(\bfS^1)^*})
 \]
 and every dinatural family of morphisms
 \[
  [\bfM_X] \in \bfA_\eend({\bfS^1} \disjun {(\bfS^1)^*}) \left( \bfSigma_{{\bfS^1} \disjun {(\bfS^1)^*}},{\bfD^2_X} \disjun {(\bfD^2_X)^\rmT} \right),
 \]
 there exists a unique morphism
 \[
  [\bfM] \in \bfA_\eend({\bfS^1} \disjun {(\bfS^1)^*}) \left( \bfSigma_{{\bfS^1} \disjun {(\bfS^1)^*}}, \lcoev_{\bfS^1} \right)
 \]
 satisfying
 \[
  [\bfM_X] = [\bfW_X] \circ [\bfM].
 \]
 First of all, up to isomorphism, we can assume that
 $\bfSigma_{{\bfS^1} \disjun {(\bfS^1)^*}}$ is target surjective, thanks to Lemma~\ref{L:ts_equivalent}. In particular, we have two cases:
 \begin{enumerate}
  \item There exists an object
   \[
    \bfSigma_{\bfS^1} = (\Sigma_{S^1},P_{S^1},\Lagr_{S^1}) \in \TS_\eend(\bfS^1)
   \]
   such that $\bfSigma_{{\bfS^1} \disjun {(\bfS^1)^*}}$ is isomorphic to $(\bfP^2 \cdot (\id_{S^1} \disjun \_^*))^\dagger \circ \bfSigma_{\bfS^1}$ in $\bfACob_\calC$, where $\_^* : (S^1)^* \to S^1$ denotes the diffeomorphism induced by complex conjugation, see Equation~\eqref{E:right_action_1-morphism} for the right action of a diffeomorphism on a morphism of $\bfACob_\calC$;
  \item There exist objects
   \begin{gather*}
    \bfSigma_{\bfS^1} = (\Sigma_{S^1},P_{S^1},\Lagr_{S^1}) \in \TS_\eend(\bfS^1), \\*
    \bfSigma_{(\bfS^1)^*} = (\Sigma_{(S^1)^*},P_{(S^1)^*},\Lagr_{(S^1)^*}) \in \TS_\eend((\bfS^1)^*)
   \end{gather*}
   such that $\bfSigma_{{\bfS^1} \disjun {(\bfS^1)^*}}$ is isomorphic to $\bfSigma_{\bfS^1} \sqcup \bfSigma_{(\bfS^1)^*}$ in $\bfACob_\calC$.
 \end{enumerate}
 Then, Proposition~\ref{P:circle_category} implies that, up to isomorphism in $\bfA_\eend(\bfS^1)$, we can assume that $\Sigma_{S^1} = D^2$ and $P_{S^1} = O_Y$ for some $Y \in \calC$, and similarly, up to isomorphism in $\bfA_\eend((\bfS^1)^*)$, we can assume that $\Sigma_{(S^1)^*} = (D^2)^\rmT$ and $P_{(S^1)^*} = (O_Z)^\dagger$ for some $Z \in \calC$. Furthermore, thanks to Lemma~\ref{L:fc_equivalence_0}, we can assume that $Y,Z \in \calC$ are projective, because $\CTS_\eend({\bfS^1} \disjun {(\bfS^1)^*})$ is finite colimit dense in $\bfA_\eend({\bfS^1} \disjun {(\bfS^1)^*})$.

 In the first case, Lemmas~\ref{L:adjoint} and \ref{L:split} imply that
 \[
  \bfA_\eend({\bfS^1} \disjun {(\bfS^1)^*}) \left( (\bfP^2 \cdot (\id_{S^1} \disjun \_^*))^\dagger \circ \bfD^2_Y,{\bfD^2_X} \disjun {(\bfD^2_X)^\rmT} \right)
 \]
 is naturally isomorphic to
 \[
  V_\eend\left( (\bfD^2_Y)^\dagger \circ (\bfP^2 \cdot (\id_{S^1} \disjun \_^*)) \circ ({\bfD^2_X} \disjun {(\bfD^2_X)^\rmT}) \right),
 \]
 and Proposition~\ref{P:algebraic_and_skein_model} implies that it is naturally isomorphic to
 \[
  \calC(Y,X \otimes X^*)
 \]
 through the map that sends every morphism $f_X$ in $\calC(Y,X \otimes X^*)$ to the morphism $[W \ast M,T_X,0]$ in $\bfA_\eend({\bfS^1} \disjun {(\bfS^1)^*}) \left( (\bfP^2 \cdot (\id_{S^1} \disjun \_^*))^\dagger \circ \bfD^2_Y,{\bfD^2_X} \disjun {(\bfD^2_X)^\rmT} \right)$, where $M$ denotes the \dmnsnl{3} cobordism with corners defined by the right action of an isomorphism from $(P^2 \cdot (\id_{S^1} \disjun \_^*))^\dagger \circ D^2$ to $\lcoev_{S^1}$ on $\id_{\lcoev_{S^1}}$, and $T_X \subset W \ast M$ denotes the $\calC$-labeled bichrome graph from $O_Y$ to $O_X \sqcup (O_X)^\dagger$ appearing on the left-hand side of Equation~\eqref{E:dinatural_skein_equivalence_1}.
 \begin{equation}\label{E:dinatural_skein_equivalence_1}
  \pic{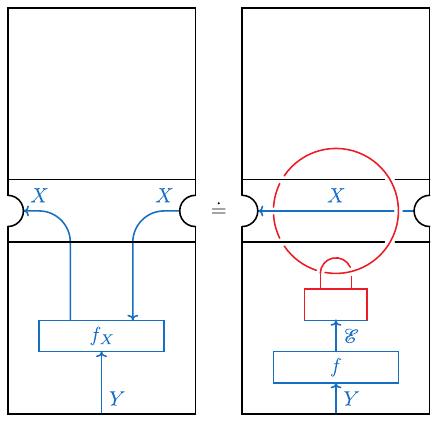}
 \end{equation}
 Then, dinaturality of $[\bfM_X]$ implies dinaturality of $f_X$, which in turn implies the existence of a unique morphism $f \in \calC(Y,\eend)$
 satisfying the admissible skein equivalence of Equation~\eqref{E:dinatural_skein_equivalence_1}.
 This means that there exists a $\calC$-labeled bichrome graph $T \subset M$ from $O_Y$ to $\varnothing$ satisfying
 \[
  [\bfM_X] = [\bfW_X] \circ [M,T,0].
 \]

 In the second case, Proposition~\ref{P:renormalized_KL_TQFT} and Lemmas~\ref{L:adjoint} and \ref{L:split} imply that
 \[
  \bfA_\eend({\bfS^1} \disjun {(\bfS^1)^*}) \left( {\bfD^2_Y} \disjun {(\bfD^2_Z)^\rmT},{\bfD^2_X} \disjun {(\bfD^2_X)^\rmT} \right)
 \]
 is naturally isomorphic to
 \[
  V_\eend \left( (\bfD^2_Y)^\dagger \circ \bfD^2_X \right) \otimes V_\eend \left( ((\bfD^2_Z)^\dagger \circ \bfD^2_X)^\rmT \right),
 \]
 and Proposition~\ref{P:algebraic_and_skein_model} implies that it is naturally isomorphic to
 \[
  \calC(Y,X) \otimes \calC(X,Z)
 \]
 through the map that sends every $f_X = g_X \otimes h_X \in \calC(Y,X) \otimes \calC(X,Z)$ to
 \[
  [{\bfD^2_{g_X}} \disjun {(\bfD^2_{h_X})^\rmT}] \in \bfA_\eend({\bfS^1} \disjun {(\bfS^1)^*}) \left( {\bfD^2_Y} \disjun {(\bfD^2_Z)^\rmT},{\bfD^2_X} \disjun {(\bfD^2_X)^\rmT} \right).
 \]
 Then, let us consider morphisms $r_Y \in \calC(Y \otimes P_{\one},Y)$ and $s_Z \in \calC(Z,Z \otimes P_{\one})$ satisfying
 \begin{align*}
  r_Y \circ (\id_Y \otimes \eta_{\one}) &= \id_Y, &
  (\id_Z \otimes \epsilon_{\one}) \circ s_Z &= \id_Z
 \end{align*}
 for the monomorphism $\eta_{\one} \in \calC(\one,P_{\one}) $ and the epimorphism $\epsilon_{\one} \in \calC(P_{\one},\one)$ appearing in move~\eqref{E:CS}, which exist because $Y,Z \in \calC$ are projective. Then, up to connected sum and admissible skein equivalence,
 \[
  [{\bfD^2_{g_X}} \disjun {(\bfD^2_{h_X})^\rmT}] = [W \ast M,T_X,0]
 \]
 in $\bfA_\eend({\bfS^1} \disjun {(\bfS^1)^*}) \left( {\bfD^2_Y} \disjun {(\bfD^2_Z)^\rmT},{\bfD^2_X} \disjun {(\bfD^2_X)^\rmT} \right)$, where $M$ denotes the \dmnsnl{3} cobordism with corners from $D^2 \sqcup (D^2)^\rmT$ to $\lcoev_{S^1}$ obtained from $W$ by reversing its orientation and exchanging incoming and outgoing horizontal boundary identifications, and $T_X \subset W \ast M$ denotes the $\calC$-labeled bichrome graph from $O_Y \sqcup (O_Z)^\dagger$ to $O_X \sqcup (O_X)^\dagger$ appearing on the left-hand side of Equation~\eqref{E:dinatural_skein_equivalence_3}.
 \begin{equation}\label{E:dinatural_skein_equivalence_3}
  \pic{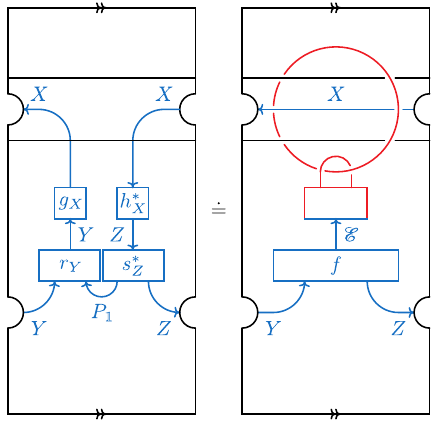}
 \end{equation}
 Then, dinaturality of $[\bfM_X]$ implies dinaturality of $f_X = g_X \otimes h_X$, which implies dinaturality of $f'_X = (g_X \otimes h_X^*) \circ (r_Y \otimes s_Z^*) \circ (\id_Y \otimes \lcoev_{P_{\one}} \otimes \id_{Z^*})$, which in turn implies the existence of a unique morphism $f' \in \calC(Y \otimes Z^*,\eend)$
 satisfying the admissible skein equivalence of Equation~\eqref{E:dinatural_skein_equivalence_3}.
 This means that there exists a $\calC$-labeled bichrome graph $T \subset M$ from $O_Y \sqcup (O_Z)^\dagger$ to $\varnothing$ satisfying
 \[
  [\bfM_X] = [\bfW_X] \circ [M,T,0].
 \]

 Notice that, in both cases, this decomposition is unique, because, if $G$ is a projective generator of $\calC$, then $[\bfW_G]$ is a monomorphism. In order to prove it, thanks to Remark~\ref{R:mono-epi}, we can show
 that the linear map $V_\eend(\bfSigma'_{{\bfS^1} \disjun {(\bfS^1)^*}} \triangleright \bfW_G)$ is injective for every object
 \[
  \bfSigma'_{{\bfS^1} \disjun {(\bfS^1)^*}} = (\Sigma'_{S^1 \sqcup (S^1)^*},P'_{S^1 \sqcup (S^1)^*},\Lagr'_{S^1 \sqcup (S^1)^*})
 \]
 of $\bfA'_\eend({\bfS^1} \disjun {(\bfS^1)^*})$. Now, Equation~\eqref{E:adjunction_V_V'} implies that $V_\eend(\bfSigma'_{{\bfS^1} \disjun {(\bfS^1)^*}} \triangleright \bfW_G)$ is injective if and only if $V'_\eend(\bfSigma'_{{\bfS^1} \disjun {(\bfS^1)^*}} \triangleright \bfW_G)$ is surjective. If $M''$ is a boundary bijective \dmnsnl{3} cobordism from $\Sigma'_{S^1 \sqcup (S^1)^*} \circ (D^2 \sqcup (D^2)^\rmT)$ to $\varnothing$, then the \dmnsnl{3} cobordism
 \[
  M' := M'' \ast (\Sigma'_{S^1 \sqcup (S^1)^*} \triangleright W)
 \]
 from $\Sigma'_{S^1 \sqcup (S^1)^*} \circ \lcoev_{S^1}$ to $\varnothing$ is boundary bijective too. Then, thanks to Lemma~\ref{L:connected_1}, the linear map
 \begin{align*}
  \pi_{M'} : S'_\eend(M';P'_{S^1 \sqcup (S^1)^*},\varnothing) &\to V'_\eend(\bfSigma'_{{\bfS^1} \disjun {(\bfS^1)^*}} \circ \lcoev_{\bfS^1}) \\*
  [T'] &\mapsto [M',T',0]
 \end{align*}
 is surjective.
 Since $G$ is a projective generator, then, up to admissible skein equivalence, for every admissible $\calC$-labeled bichrome graph $T' \subset M'$ from $P'_{S^1 \sqcup (S^1)^*}$ to $\varnothing$ there exist admissible $\calC$-labeled bichrome graphs $T''_1,\ldots,T''_m \subset M''$ from $P'_{S^1 \sqcup (S^1)^*} \cup (O_G \sqcup (O_G)^\dagger)$ to $\varnothing$ such that
 \[
  [M',T',0] = \sum_{i=1}^m [M'',T''_i,0] \circ [\bfSigma'_{{\bfS^1} \disjun {(\bfS^1)^*}} \triangleright \bfW_G].
 \]
 Indeed, up to isotopy and skein equivalence, $T'$ intersects $W$ along $I_X$, and we can suppose that $X$ is projective, because $T'$ is admissible. Since the identity of every projective object $X$ can be written as a sum of endomorphisms that factor through the projective generator $G$, the claim follows from an admissible skein equivalence.

 The second statement is obtained from the first one by applying Lemma~\ref{L:transpose_non-empty}. \qedhere
\end{proof}

\begin{remark}\label{R:coend_cokernel}
 Thanks to Lemma~\ref{L:adjoint_exactness}, the separate linear functor
 \begin{align*}
  (\bfD^2)^\rmT \sqcup \bfD^2 : (\calC^\op,\calC) &\to \bfA_\eend({(\bfS^1)^*} \disjun {\bfS^1}) \\*
  X \otimes Y &\mapsto (\bfD^2_X)^\rmT \sqcup \bfD^2_Y
 \end{align*}
 is separately right exact, because both $(\bfD^2_X)^\rmT \sqcup \id_{\bfS^1}$ and $\id_{(\bfS^1)^*} \sqcup \bfD^2_Y$ are admissible. Then, it follows from Proposition~\ref{P:end} and \cite[Corollary~5.1.8]{KL01} that $\rcoev_{\bfS^1}$ is the cokernel of a morphism in $\bfA_\eend({(\bfS^1)^*} \disjun {\bfS^1})$, namely
 \[
  ([{(\bfD^2_f)^\rmT} \disjun {\id_{\bfD^2_G}}] - [{\id_{(\bfD^2_G)^\rmT}} \disjun {\bfD^2_f}])_{f \in \calB} : \bigoplus_{f \in \calB} {(\bfD^2_G)^\rmT} \disjun {\bfD^2_G} \to {(\bfD^2_G)^\rmT} \disjun {\bfD^2_G},
 \]
 where $G$ is a projective generator of $\calC$, and $\calB$ is a basis of $\End_\calC(G)$, as in Section~\ref{S:ribbon_categories}. Notice that this morphism belongs to the image of the functor $\bfmu_{(\bfS^1)^*,\bfS^1}$ of Proposition~\ref{P:lax_monoidality_ETQFT}.
\end{remark}

\subsection{Limit density of coadmissible target bijective objects}

Next, we are going to show that coadmissible objects are also finite limit dense in $\bfA_\eend(\bfGamma)$ for every object $\bfGamma$ of $\bfACob_\calC$. This will allow us to prove that the separate linear functor $\bfmu_{\bfGamma,\bfGamma'} : (\bfA_\eend(\bfGamma),\bfA_\eend(\bfGamma')) \to \bfA_\eend(\bfGamma \disjun \bfGamma')$ of Proposition~\ref{P:lax_monoidality_ETQFT} is separately left exact for all $\bfGamma, \bfGamma' \in \bfACob_\calC$, and that the image of the \fnctr{2} $\bfA_\eend : \bfACob_\calC \to \bfCat_\Bbbk$ is contained in the \ctgr{2} $\SCat_\Bbbk$ of separate linear categories.

Let $\bfGamma = \Gamma$ and $\bfGamma' = \Gamma'$ be objects of $\bfACob_\calC$. A \mrphsm{1} $\bfSigma = (\Sigma,P,\Lagr)$ from $\bfGamma$ to $\bfGamma'$ in $\bfACob_\calC$ is \textit{target injective} if the embedding
\[
 \iota : \Gamma' \hookrightarrow \Sigma
\]
induced by the outgoing boundary parametrization determines an injection
\[
 \iota_* : \pi_0(\Gamma') \hookrightarrow \pi_0(\Sigma).
\]

Let $\bfGamma$ be an object of $\bfACob_\calC$. We denote by $\CTB_\eend(\bfGamma)$ the full linear subcategory of $\CTS_\eend(\bfGamma)$ whose objects are target injective.

Let $\bfSigma : \bfGamma \to \bfGamma'$ be a \mrphsm{1} of $\bfACob_\calC$. If $\bfSigma : \bfGamma \to \bfGamma'$ is coadmissible and target bijective, then $\bfA_\eend(\bfSigma) : \bfA_\eend(\bfGamma) \to \bfA_\eend(\bfGamma')$ restricts to a linear functor $\bfA_\eend(\bfSigma) : \CTB_\eend(\bfGamma) \to \CTB_\eend(\bfGamma')$.

\begin{lemma}\label{L:ctb_colimit_dense}
 For every object $\bfGamma \in \bfACob_\calC$, the linear category $\CTB_\eend(\bfGamma)$ is finite colimit dense in the linear category $\bfA_\eend(\bfGamma)$.
\end{lemma}

\begin{proof}
 Let us separate two cases.

 If $\bfGamma = \varnothing$, then the only target bijective object of $\bfA_\eend(\varnothing)$ is $\id_\varnothing$, which is also coadmissible, so the claim follows from Lemma~\ref{L:empty_additive_generator}.

 If $\bfGamma \neq \varnothing$, let $\bfSigma_{\bfGamma} = (\Sigma_\Gamma,P_\Gamma,\Lagr_\Gamma) \in \bfA_\eend(\bfGamma)$ be an object. First of all, up to iterated finite colimits, we can assume that $\bfSigma_{\bfGamma}$ is coadmissible and target surjective, thanks to Lemma~\ref{L:cts_colimit_dense}. Then, let $n(\bfSigma_{\bfGamma}) \in \N$ be the number of connected components of $\Gamma$ that violate target injectivity. If $n(\bfSigma_{\bfGamma}) = 0$, then $\bfSigma_{\bfGamma}$ is target injective. If $n(\bfSigma_{\bfGamma}) > 0$, then $\bfSigma_{\bfGamma}$ is isomorphic to $\bfSigma \circ \rcoev_{\bfS^1}$ in $\bfACob_\calC$, where $\bfSigma : {(\bfS^1)^*} \disjun {\bfS^1} \to \bfGamma$ is admissible and $n(\bfSigma) < n(\bfSigma_{\bfGamma})$. Thanks to Remark~\ref{R:coend_cokernel}, $\rcoev_{\bfS^1}$ is the colimit of a finite diagram $D : J \to \CTB_\eend({(\bfS^1)^*} \disjun {\bfS^1})$. Then, since $\bfA_\eend(\bfSigma) : \bfA_\eend({(\bfS^1)^*} \disjun {\bfS^1}) \to \bfA_\eend(\bfGamma)$ is right exact, $\bfSigma_{\bfGamma}$ is the colimit of the finite diagram $\bfA_\eend(\bfSigma) \circ D : J \to \bfA_\eend(\bfGamma)$. Since $n(\bfSigma \circ D(j)) < n(\bfSigma_{\bfGamma})$ for every $j \in J$, an induction on $n(\bfSigma_{\bfGamma})$ implies that $\bfSigma_{\bfGamma}$ belongs to the finite colimit closure of $\CTB_\eend(\bfGamma)$. \qedhere
\end{proof}

\begin{proposition}\label{P:left_exactness}
 For all objects $\bfGamma, \bfGamma' \in \bfACob_\calC$, the separate linear functor $\bfmu_{\bfGamma,\bfGamma'} : (\bfA_\eend(\bfGamma),\bfA_\eend(\bfGamma')) \to \bfA_\eend(\bfGamma \disjun \bfGamma')$ is separately left exact.
\end{proposition}

\begin{proof}
 Thanks to Propositions~\ref{P:lax_monoidality_ETQFT} and \ref{P:fullness}, the restriction
 \[
  \bfmu_{\bfGamma,\bfGamma'} : (\CTB_\eend(\bfGamma),\CTB_\eend(\bfGamma')) \to \CTB_\eend(\bfGamma \disjun \bfGamma')
 \]
 is an equivalence, since it is clearly essentially surjective. Then, Lemma~\ref{L:ctb_colimit_dense} implies that
 \[
  \bfmu_{\bfGamma,\bfGamma'} : (\bfA_\eend(\bfGamma),\bfA_\eend(\bfGamma')) \to \bfA_\eend(\bfGamma \disjun \bfGamma')
 \]
 is separate finite colimit dense. Then, the claim follows from Lemma~\ref{L:fc_equivalence_1}. \qedhere
\end{proof}

\begin{lemma}\label{L:ctb_limit_dense}
 For every object $\bfGamma \in \bfACob_\calC$, the linear category $\CTB_\eend(\bfGamma)$ is finite limit dense in the linear category $\bfA_\eend(\bfGamma)$.
\end{lemma}

\begin{proof}
 Let us separate two cases.

 If $\bfGamma = \varnothing$, then the only target bijective object of $\bfA_\eend(\varnothing)$ is $\id_\varnothing$, which is also coadmissible, so the claim follows from Lemma~\ref{L:empty_additive_generator}.

 If $\bfGamma \neq \varnothing$, let $\bfSigma_{\bfGamma} = (\Sigma_\Gamma,P_\Gamma,\Lagr_\Gamma) \in \bfA_\eend(\bfGamma)$ be an object. First of all, up to isomorphism, we can assume that $\bfSigma_{\bfGamma}$ is target surjective, thanks to Lemma~\ref{L:ts_equivalent}. Then, let $k(\bfSigma_{\bfGamma}) \in \N$ be the number of connected components of $\Sigma_\Gamma$ that violate coadmissibility, and let $n(\bfSigma_{\bfGamma}) \in \N$ be the number of connected components of $\Gamma$ that violate target injectivity. Then $\bfSigma_{\bfGamma}$ is isomorphic to
 \[
  \bfSigma \circ ({(\bfD^2)^{{} \disjun {k(\bfSigma_{\bfGamma})}}} \disjun {(\lcoev_{\bfS^1})^{{} \disjun {n(\bfSigma_{\bfGamma})}}})
 \]
 in $\bfACob_\calC$, where $\bfSigma : {(\bfS^1)^{{} \disjun {k(\bfSigma_{\bfGamma})}}} \disjun {({\bfS^1} \disjun {(\bfS^1)^*})^{{} \disjun {n(\bfSigma_{\bfGamma})}}} \to \bfGamma$ is coadmissible and target bijective. Thanks to Proposition~\ref{P:circle_category}, $\bfA_\eend(\bfS^1)$ is equivalent to $\calC$, and since $\Proj(\calC)$ is finite limit dense in $\calC$, then $\bfD^2$ is the limit of a finite diagram $D : J \to \CTB_\eend(\bfS^1)$. At the same time, thanks to Remark~\ref{R:end_kernel}, $\lcoev_{\bfS^1}$ is the limit of a finite diagram $D' : J' \to \CTB_\eend({\bfS^1} \disjun {(\bfS^1)^*})$. Thanks to Proposition~\ref{P:left_exactness},
 \[
  {(\bfD^2)^{{} \disjun {k(\bfSigma_{\bfGamma})}}} \disjun {(\lcoev_{\bfS^1})^{{} \disjun {n(\bfSigma_{\bfGamma})}}}
 \]
 belongs to the finite limit closure of $\CTB_\eend({(\bfS^1)^{{} \disjun {k(\bfSigma_{\bfGamma})}}} \disjun {({\bfS^1} \disjun {(\bfS^1)^*})^{{} \disjun {n(\bfSigma_{\bfGamma})}}})$. Then, since $\bfA_\eend(\bfSigma) : \bfA_\eend({(\bfS^1)^{{} \disjun {k(\bfSigma_{\bfGamma})}}} \disjun {({\bfS^1} \disjun {(\bfS^1)^*})^{{} \disjun {n(\bfSigma_{\bfGamma})}}}) \to \bfA_\eend(\bfGamma)$ is left exact, $\bfSigma_{\bfGamma}$ belongs to the finite limit closure of $\CTB_\eend(\bfGamma)$. \qedhere
\end{proof}

\begin{proposition}\label{P:left_exactness_Sigma}
 The image of the \fnctr{2} $\bfA_\eend : \bfACob_\calC \to \bfCat_\Bbbk$ is contained in $\SCat_\Bbbk$.
\end{proposition}

\begin{proof}
 We need to prove that, for all objects $\bfGamma,\bfGamma' \in \bfACob_\calC$ and for every \mrphsm{1} $\bfSigma \in \bfACob_\calC(\bfGamma,\bfGamma')$, the linear functor $\bfA_\eend(\bfSigma) : \bfA_\eend(\bfGamma) \to \bfA_\eend(\bfGamma')$ is left exact.

 If $\bfSigma = (\Sigma,P,\Lagr)$, let $k(\bfSigma) \in \N$ be the number of connected components of $\Sigma$ that violate coadmissibility. Then $\bfSigma$ is isomorphic to
 \[
  \bfSigma' \circ ({(\bfD^2)^{{} \disjun {k(\bfSigma)}}} \disjun {\id_{\bfGamma}}),
 \]
 where $\bfSigma' : {((\bfS^1)^{{} \disjun {k(\bfSigma)}}} \disjun \bfGamma) \to \bfGamma'$ is coadmissible. Now, the separate linear functor
 \[
  (\bfD^2)^{{} \disjun {k(\bfSigma)}} \otimes \_ : \bfA_\eend(\bfGamma) \to (\bfA_\eend((\bfS^1)^{{} \disjun {k(\bfSigma)}}),\bfA_\eend(\bfGamma))
 \]
 is exact. Then, Lemma~\ref{L:adjoint_exactness} and Proposition~\ref{P:left_exactness} imply that the linear functor
 \[
  \bfA_\eend(\bfSigma') \circ \bfmu_{(\bfS^1)^{{} \disjun {k(\bfSigma)}},\bfGamma} \circ ((\bfD^2)^{{} \disjun {k(\bfSigma)}} \otimes \_)
 \]
 is left exact. \qedhere

\end{proof}

\begin{remark}\label{R:end_kernel}
 Thanks to Propositions~\ref{P:left_exactness} and \ref{P:left_exactness_Sigma}, the separate linear functor
 \begin{align*}
  \bfD^2 \sqcup (\bfD^2)^\rmT : (\calC,\calC^\op) &\to \bfA_\eend({\bfS^1} \disjun {(\bfS^1)^*}) \\*
  X \otimes Y &\mapsto \bfD^2_X \sqcup (\bfD^2_Y)^\rmT
 \end{align*}
 is separately left exact, because both $\bfA_\eend(\bfD^2_X \sqcup \id_{(\bfS^1)^*})$ and $\bfA_\eend(\id_{\bfS^1} \sqcup (\bfD^2_Y)^\rmT)$ are left exact. Then, it follows from  Proposition~\ref{P:end} and \cite[Corollary~5.1.8]{KL01} that $\lcoev_{\bfS^1}$ is the kernel of a morphism in $\bfA_\eend({\bfS^1} \disjun {(\bfS^1)^*})$, namely
 \[
  ([{\bfD^2_f} \disjun {\id_{(\bfD^2_G)^\rmT}}] - [{\id_{\bfD^2_G}} \disjun {(\bfD^2_f)^\rmT}])_{f \in \calB} : {\bfD^2_G} \disjun {(\bfD^2_G)^\rmT} \to \bigoplus_{f \in \calB} {\bfD^2_G} \disjun {(\bfD^2_G)^\rmT},
 \]
 where $G$ is a projective generator of $\calC$, and $\calB$ is a basis of $\End_\calC(G)$, as in Section~\ref{S:ribbon_categories}. Notice that this morphism belongs to the image of the functor $\bfmu_{\bfS^1,(\bfS^1)^*}$ of Proposition~\ref{P:lax_monoidality_ETQFT}.
\end{remark}

\subsection{Symmetric monoidality}\label{S:monoidality}

In this section, we prove our main result, by equipping the finite completion of the \fnctr{2} $\bfA_\eend : \bfACob_\calC \to \SCat_\Bbbk$ with a symmetric monoidal structure, thus promoting it to an ETQFT.

First, let $\bfeta : \Bbbk \to \bfA_\eend(\varnothing)$ denote the linear functor sending the unique object of $\Bbbk$ to the object $\id_\varnothing$ of $\bfA_\eend(\varnothing)$.

\begin{proposition}\label{P:unit_coherence}
 The linear functor $\bfeta : \Bbbk \to \bfA_\eend(\varnothing)$ is a finite completion equivalence.
\end{proposition}

\begin{proof}
 The linear functor $\bfeta$ is fully faithful, because
 \[
  \bfA_\eend(\varnothing)(\id_\varnothing,\id_\varnothing) \cong V_\eend(\id_\varnothing) \otimes V'_\eend(\id_\varnothing) \cong \Bbbk,
 \]
 thanks to Lemma~\ref{L:split}. It is also both finite limit dense and finite colimit dense, thanks to Lemma~\ref{L:empty_additive_generator}. Then, since $\id_\varnothing$ is biadmissible, Lemma~\ref{L:coadjoint_exactness} and Proposition~\ref{P:fc_equivalence} imply that $\bfeta$ is a finite completion equivalence.
\end{proof}

Notice that, under the canonical identification between the finite completion $\hat{\Bbbk}$ of the linear category $\Bbbk$ and $\FVect_\Bbbk$, the finite completion $\bfhateta : \FVect_\Bbbk \to \bfhatA_\eend(\varnothing)$ of the linear functor $\bfeta : \Bbbk \to \bfA_\eend(\varnothing)$ of Proposition~\ref{P:unit_coherence} sends every finite-dimensional vector space $V \in \FVect_\Bbbk$ to the object $V \otimes \id_\varnothing \in \bfhatA_\eend(\varnothing)$. This defines a linear equivalence, and the representable linear functor $\bfhatA_\eend(\varnothing)(\id_\varnothing,\_) : \bfhatA_\eend(\varnothing) \to \Vect_\Bbbk$ is a quasi-inverse. Indeed, thanks to Lemma~\ref{L:split}, the isomorphism
\[
 \bfA_\eend(\varnothing)(\id_\varnothing,\bfSigma_\varnothing) \cong V_\eend(\bfSigma_\varnothing)
\]
is natural in $\bfSigma_\varnothing \in \bfA_\eend(\varnothing)$, and the claim is a direct consequence of Proposition~\ref{P:renormalized_KL_TQFT} and Equations~\eqref{E:admissible_non-empty_representative_of_empty} and \eqref{E:admissible_connected_representative}.

\begin{proposition}\label{P:product_coherence}
 For all objects $\bfGamma, \bfGamma' \in \bfACob_\calC$, the separate linear functor $\bfmu_{\bfGamma,\bfGamma'} : (\bfA_\eend(\bfGamma),\bfA_\eend(\bfGamma')) \to \bfA_\eend(\bfGamma \disjun \bfGamma')$ is a finite completion equivalence.
\end{proposition}

\begin{proof}
 The linear functor $\bfmu_{\bfGamma,\bfGamma'}$ is fully faithful and separately left exact, thanks to Propositions~\ref{P:lax_monoidality_ETQFT}, \ref{P:fullness}, and \ref{P:left_exactness}. Furthermore, thanks to Lemma~\ref{L:ctb_limit_dense}, the full linear subcategory $\CTB_\eend(\bfGamma) \otimes \CTB_\eend(\bfGamma')$ of $\bfA_\eend(\bfGamma) \otimes \bfA_\eend(\bfGamma')$ is separate finite limit dense, and so is $\bfmu_{\bfGamma,\bfGamma'}$. Furthermore, thanks to Lemma~\ref{L:coadjoint_exactness}, the separate linear functor $\bfA_\eend(\bfGamma \disjun \bfGamma')(\_,\bfmu_{\bfGamma,\bfGamma'}(\bfSigma_{\bfGamma} \otimes \bfSigma'_{\bfGamma}))$ is separately right exact for every $\bfSigma_{\bfGamma} \otimes \bfSigma'_{\bfGamma} \in \CTB_\eend(\bfGamma) \otimes \CTB_\eend(\bfGamma')$. Then, thanks to Proposition~\ref{P:fc_equivalence}, $\bfmu_{\bfGamma,\bfGamma'}$ is a finite completion equivalence. \qedhere
\end{proof}

We are now ready to prove our main result.

\begin{theorem}\label{T:symmetric_monoidality}
 The \fnctr{2} $\bfhatA_\eend : \bfACob_\calC \to \coCat_\Bbbk$ is symmetric monoidal.
\end{theorem}

\begin{proof}
 The claim follows directly from the results we have established up to here, but in order to prove it we need to check carefully that all the conditions required by \cite[Definition~D.16]{D17} are met. First of all, as a direct consequence of Proposition~\ref{P:left_exactness_Sigma}, we can consider the finite completion $\bfhatA_\eend : \bfACob_\calC \to \FCat_\Bbbk$ of the \fnctr{2} $\bfA_\eend : \bfACob_\calC \to \SCat_\Bbbk$, which is given by
 \[
  \bfhatA_\eend := \bfC \circ \bfA_\eend
 \]
 for the finite completion \fnctr{2} $\cmpl : \SCat_\Bbbk \to \FCat_\Bbbk$ of Proposition~\ref{P:completion}. 
 
 Next, thanks to Proposition~\ref{P:unit_coherence}, the completion
 \[
  \bfhateta := \cmpl(\bfeta)
 \]
 of the linear functor $\bfeta$ is an equivalence. Similarly, thanks to Proposition~\ref{P:product_coherence}, the \trnsfrmtn{2}
 \[
  \bfhatmu := ( \cmpl \triangleright \bfmu ) \ast ( \bfchi \triangleleft (\bfA_\eend \times \bfA_\eend))
 \]
 is a composition of equivalences, where $\bfchi : {\sqtimes} \circ {(\cmpl \times \cmpl)} \Rightarrow \cmpl \circ \otimes$ is part of the coherence data of the symmetric monoidal \fnctr{2} $\cmpl$ given by Proposition~\ref{P:completion}.
 
 Furthermore, we claim that
 \begin{gather*}
  \bfmu_{\varnothing,\bfGamma} \circ (\bfeta \otimes \id_{\bfA_\eend(\bfGamma)}) = \id_{\bfA_\eend(\bfGamma)}, \quad \bfmu_{\bfGamma,\varnothing} \circ (\id_{\bfA_\eend(\bfGamma)} \otimes \bfeta) = \id_{\bfA_\eend(\bfGamma)}, \\
  \bfmu_{\bfGamma,\bfGamma' \disjun \bfGamma''} \circ \left( \id_{\bfA_\eend(\bfGamma)} \otimes \bfmu_{\bfGamma',\bfGamma''} \right) = \bfmu_{\bfGamma \disjun \bfGamma',\bfGamma''} \circ \left( \bfmu_{\bfGamma,\bfGamma'} \otimes \id_{\bfA_\eend(\bfGamma'')} \right)
 \end{gather*}
 for all $\bfGamma, \bfGamma', \bfGamma'' \in \bfACob_\calC$. Indeed, this is an immediate consequence of the quasi-strictness of $\bfACob_\calC$. This means that the unit and associativity coherence data for $\bfhatA_\eend$ can be taken as identity \mdfctns{2}. 
 
 Finally, the symmetry coherence data for $\bfhatA_\eend$ can be taken as the \mdfctn{2} obtained by appropriately assembling the image of the symmetric braiding of $\bfACob_\calC$, so its axioms follow directly from those of $\bfACob_\calC$. \qedhere
\end{proof}

\subsection{Identification of the image}

In this section, we characterize explicitly the image of all the generating objects and \mrphsms{1} of $\bfACob_\calC$ under the ETQFT $\bfhatA_\eend : \bfACob_\calC \to \coCat_\Bbbk$.

Let us start from discs. First,
Equation~\eqref{E:canonical_Vect-action_intertwiner}
yields, for every object $X \in \calC$, an isomorphism
\[
 (\eta_X)_V \in \bfhatA_\eend(\bfS^1) \left( \bfD^2_{V \otimes X},V \otimes \bfD^2_X \right)
\]
that is natural in $V \in \FVect_\Bbbk$. In other words, we have the following result.

\begin{proposition}\label{P:unit}
 For every object $X \in \calC$, the family
 \[
  \eta_X = \left\{ (\eta_X)_V \in \bfhatA_\eend(\bfS^1) \left( \bfD^2_{V \otimes X},V \otimes \bfD^2_X \right) \Bigm\vert V \in \FVect_\Bbbk \right\}
 \]
 of isomorphisms of $\bfhatA_\eend(\bfS^1)$ defines a natural isomorphism
 \begin{center}
  \begin{tikzpicture}[descr/.style={fill=white}]
   \node (P0) at ({sqrt(2)},0) {$\bfhatA_\eend(\bfS^1)$};
   \node (P1) at (0,{sqrt(2)}) {$\calC$};
   \node (P2) at ({-sqrt(2)},0) {$\FVect_\Bbbk$};
   \node (P3) at (0,{-sqrt(2)}) {$\bfhatA_\eend(\varnothing)$};
    \node (P4) at (0,0) {$\phantom{\eta_X} \Downarrow \eta_X$};
   \draw
   (P1) edge[->] node[right,yshift=5pt] {$\bfD^2$} (P0)
   (P2) edge[->] node[left,yshift=5pt] {$\_ \otimes X$} (P1)
   (P2) edge[->] node[left,yshift=-5pt] {$\bfhateta$} (P3)
   (P3) edge[->] node[right,yshift=-5pt] {$\bfhatA_\eend(\bfD^2_X)$} (P0);
  \end{tikzpicture}
 \end{center}
\end{proposition}

Next, Proposition~\ref{P:algebraic_and_skein_model} and Equations~\eqref{E:admissible_non-empty_representative_of_empty} and \eqref{E:admissible_connected_representative} yield, for every object $X \in \calC$, an isomorphism
\[
 (\epsilon_X)_Y \in \bfhatA_\eend(\varnothing) \left( \calC(X,Y) \otimes \id_\varnothing,(\bfD^2_X)^\dagger \circ \bfD^2_Y \right)
\]
that is natural in $Y \in \calC$. In other words, we have the following result.

\begin{proposition}\label{P:counit}
 For every object $X \in \calC$, the family
 \[
  \epsilon_X = \left\{ (\epsilon_X)_Y \in \bfhatA_\eend(\varnothing) \left( \calC(X,Y) \otimes \id_\varnothing,(\bfD^2_X)^\dagger \circ \bfD^2_Y \right) \Bigm\vert Y \in \calC \right\}
 \]
 of isomorphisms of $\bfhatA_\eend(\varnothing)$ defines a natural isomorphism
 \begin{center}
  \begin{tikzpicture}[descr/.style={fill=white}]
   \node (P0) at ({sqrt(2)},0) {$\bfhatA_\eend(\varnothing)$};
   \node (P1) at (0,{sqrt(2)}) {$\FVect_\Bbbk$};
   \node (P2) at ({-sqrt(2)},0) {$\calC$};
   \node (P3) at (0,{-sqrt(2)}) {$\bfhatA_\eend(\bfS^1)$};
    \node (P4) at (0,0) {$\epsilon_X \Downarrow \phantom{\epsilon_X}$};
   \draw
   (P1) edge[->] node[right,yshift=5pt] {$\bfhateta$} (P0)
   (P2) edge[->] node[left,yshift=5pt] {$\calC(X,\_)$} (P1)
   (P2) edge[->] node[left,yshift=-5pt] {$\bfD^2$} (P3)
   (P3) edge[->] node[right,yshift=-5pt] {$\bfhatA_\eend((\bfD^2_X)^\dagger)$} (P0);
  \end{tikzpicture}
 \end{center}
\end{proposition}

Let us move on to pants.
First, Proposition~\ref{P:circle_category} yields an isomorphism
\[
 \pi_{X \otimes Y} \in \bfhatA_\eend(\bfS^1) \left( \bfD^2_{X \otimes Y},\bfP^2 \circ (\bfD^2_X \disjun \bfD^2_Y) \right)
\]
that is natural in $X \otimes Y \in (\calC,\calC)$. This extends uniquely to all $X \sqtimes Y \in \calC \sqtimes \calC$ because both $\bfD^2 \circ \otimes$ and $\bfhatA_\eend(\bfP^2) \circ \bfmu_{\bfS^1,\bfS^1} \circ (\bfD^2,\bfD^2)$ are separately left exact. In other words, if we use the shorthand notation $\bfD^2 \disjun \bfD^2 : \calC \sqtimes \calC \to \bfhatA_\eend(\bfS^1 \disjun \bfS^1)$ for the linear functor $\bfD^2 \disjun \bfD^2 := \bfhatmu_{\bfS^1,\bfS^1} \circ (\bfD^2 \sqtimes \bfD^2)$, we have the following result.

\begin{proposition}\label{P:product}
 The family
 \[
  \pi = \left\{ \pi_{X \sqtimes Y} \in \bfhatA_\eend(\bfS^1) \left( \bfD^2_{X \otimes Y},\bfP^2 \circ (\bfD^2_X \disjun \bfD^2_Y) \right) \Bigm\vert X \sqtimes Y \in \calC \sqtimes \calC \right\}
 \]
 of isomorphisms of $\bfhatA_\eend(\bfS^1)$ defines a natural isomorphism
 \begin{center}
  \begin{tikzpicture}[descr/.style={fill=white}]
   \node (P0) at ({sqrt(2)},0) {$\bfhatA_\eend(\bfS^1)$};
   \node (P1) at (0,{sqrt(2)}) {$\calC$};
   \node (P2) at ({-sqrt(2)},0) {$\calC \sqtimes \calC$};
   \node (P3) at (0,{-sqrt(2)}) {$\bfhatA_\eend(\bfS^1 \disjun \bfS^1)$};
    \node (P4) at (0,0) {$\phantom{\pi} \Downarrow \pi$};
   \draw
   (P1) edge[->] node[right,yshift=5pt] {$\bfD^2$} (P0)
   (P2) edge[->] node[left,yshift=5pt] {$\otimes$} (P1)
   (P2) edge[->] node[left,yshift=-5pt] {$\bfD^2 \disjun \bfD^2$} (P3)
   (P3) edge[->] node[right,yshift=-5pt] {$\bfhatA_\eend(\bfP^2)$} (P0);
  \end{tikzpicture}
 \end{center}
\end{proposition}

Finally, for every $Y \in \calC$, let us use the shorthand notation
\[
 \eend_{[1]} \sqtimes {(\eend_{[2]}^* \otimes Y)} = \int_{X \in \calC} X \sqtimes {(X^* \otimes Y)}
\]
for the end of the separate linear functor $\_ \sqtimes {(\_^* \otimes Y)} : (\calC,\calC^\op) \to \calC \sqtimes \calC$, and
\[
 \bfD^2_{\eend_{[1]}} \disjun \bfD^2_{\eend_{[2]}^* \otimes Y} = \int_{X \in \calC} \bfD^2_X \disjun \bfD^2_{X^* \otimes Y}
\]
for the end of the separate linear functor $\bfD^2_\_ \disjun \bfD^2_{\_^* \otimes Y} : (\calC,\calC^\op) \to \bfhatA_\eend(\bfS^1 \disjun \bfS^1)$ defined by $\bfD^2_\_ \disjun \bfD^2_{\_^* \otimes Y} := (\bfD^2 \disjun \bfD^2) \circ (\_ \sqtimes {(\_^* \otimes Y)})$. Notice that both separate linear functors are separately left exact.

Then, let
\[
 \eend_{[1]} \sqtimes {(\eend_{[2]}^* \otimes \_)} : \calC \to \calC \sqtimes \calC
\]
be the left exact linear functor sending every $Y \in \calC$ to $\eend_{[1]} \sqtimes {(\eend_{[2]}^* \otimes Y)} \in \calC \sqtimes \calC$.

Now, notice that, if $\_^* : (S^1)^* \to S^1$ denotes the diffeomorphism induced by complex conjugation, then we have an isomorphism
\[
 (\bfP^2)^\dagger \cong ({\id_{\bfS^1}} \disjun {(\bfP^2 \cdot (\_^* \disjun \id_{S^1}))}) \circ (\lcoev_{\bfS^1} \disjun \id_{\bfS^1})
\]
in $\bfACob_\calC$, see Equation~\eqref{E:right_action_1-morphism} for the right action of a diffeomorphism on a morphism of $\bfACob_\calC$. Then, Proposition~\ref{P:end} and Remark~\ref{R:end_kernel}, together with left exactness of $\bfhatmu_{\bfS^1,\bfS^1}$ and $\bfhatA_\eend({\id_{\bfS^1}} \disjun {(\bfP^2 \cdot (\_^* \disjun \id_{S^1}))})$, yield an isomorphism
\[
 \Delta_Y \in \bfhatA_\eend(\bfS^1 \disjun \bfS^1) \left( \bfD^2_{\eend_{[1]}} \disjun \bfD^2_{\eend_{[2]}^* \otimes Y},(\bfP^2)^\dagger \circ \bfD^2_Y \right)
\]
that is natural in $Y \in \calC$. In other words, we have the following result.

\begin{proposition}\label{P:coproduct}
 The family
 \[
  \Delta = \left\{ \Delta_Y \in \bfhatA_\eend(\bfS^1 \disjun \bfS^1) \left( \bfD^2_{\eend_{[1]}} \disjun \bfD^2_{\eend_{[2]}^* \otimes Y},(\bfP^2)^\dagger \circ \bfD^2_Y \right) \Bigm\vert Y \in \calC \right\}
 \]
 of isomorphisms of $\bfhatA_\eend(\bfS^1 \disjun \bfS^1)$ defines a natural isomorphism
 \begin{center}
  \begin{tikzpicture}[descr/.style={fill=white}]
   \node (P0) at ({sqrt(2)},0) {$\bfhatA_\eend(\bfS^1 \disjun \bfS^1)$};
   \node (P1) at (0,{sqrt(2)}) {$\calC \sqtimes \calC$};
   \node (P2) at ({-sqrt(2)},0) {$\calC$};
   \node (P3) at (0,{-sqrt(2)}) {$\bfhatA_\eend(\bfS^1)$};
    \node (P4) at (0,0) {$\Delta \Downarrow \phantom{\Delta}$};
   \draw
   (P1) edge[->] node[right,yshift=5pt] {$\bfD^2 \disjun \bfD^2$} (P0)
   (P2) edge[->] node[left,yshift=5pt] {$\eend_{[1]} \sqtimes {(\eend_{[2]}^* \otimes \_)}$} (P1)
   (P2) edge[->] node[left,yshift=-5pt] {$\bfD^2$} (P3)
   (P3) edge[->] node[right,yshift=-5pt] {$\bfhatA_\eend((\bfP^2)^\dagger)$} (P0);
  \end{tikzpicture}
 \end{center}
\end{proposition}

\subsection{Relation between TQFTs}

We finish by proving that, for a modular category $\calC$ with adjoint end $\eend$, the \KL{} TQFT $J_\eend$ is naturally isomorphic to the restriction of the ETQFT $\bfhatA_\eend$ to the category $\CCob \subset \bfACob_\calC(\varnothing,\bfS^1)$, and the renormalized \KL{} TQFT $V_\eend$ is naturally isomorphic to the restriction of $\bfhatA_\eend$ to the category $\ACob_\calC = \bfACob_\calC(\varnothing,\varnothing)$.

First of all, if $A$ and $A'$ are linear categories and $X \in A$ is an object, let us denote by
\[
 E_X : \bfLex_\Bbbk(A,A') \to A'
\]
the evaluation linear functor, that sends every left exact linear functor $F : A \to A'$ to the object $F(X) \in A'$.

We start from the \KL{} TQFT $J_\eend$. For every $g \in \KTan$, let us consider the invertible morphism
\[
 \alpha_g \in \bfA_\eend(\bfS^1) \left( \bfD^2_{\eend^{\otimes g}}, \bfSigma_g \right)
\]
given by
\[
 \bfM(g,\varnothing) = [M(\eta^{{} \bcs g}),S(g,\varnothing),0],
\]
where the invertible morphism $\bfM(g,P) \in \bfA_\eend(\bfS^1) ( \bfD^2_{\eend^{\otimes g} \otimes F_\calC(P)},\bfSigma_g \bcs \bfD^2(P) )$ and the $\calC$-labeled bichrome graph $S(g,P) \subset M(\eta^{{} \bcs g})$ are defined by Equation~\eqref{E:disc_essentially_surjective_1} for every $P \in \calB_\calC$. Let us set $S(g) := S(g,\varnothing)$.

\begin{proposition}\label{P:KL_TQFT_inside_ETQFT}
 The family
 \[
  \alpha = \left\{ \alpha_g \in \bfA_\eend(\bfS^1) \left( \bfD^2_{\eend^{\otimes g}}, \bfSigma_g \right) \Bigm\vert g \in \KTan \right\}
 \]
 of isomorphisms of $\bfA_\eend(\bfS^1)$ defines a natural isomorphism
 \begin{equation}\label{E:commutative_diagram_1}
  \raisebox{-0.5\height+2.5pt}{
  \begin{tikzpicture}[descr/.style={fill=white}] \everymath{\displaystyle}
   \node (P0) at (-1.5,0) {$\KTan$};
   \node (P1) at (0,1.5) {$\calC$};
   \node (P2) at (0,-1.5) {$\bfLex_\Bbbk(\bfhatA_\eend(\varnothing),\bfhatA_\eend(\bfS^1))$};
   \node (P3) at (1.5,0) {$\bfhatA_\eend(\bfS^1)$};
   \node (P4) at (0,0) {$\phantom{\alpha} \Downarrow \alpha$};
   \draw
   (P0) edge[->] node[left,yshift=5pt] {\scriptsize $J_\eend$}(P1)
   (P0) edge[->] node[left,yshift=-5pt] {\scriptsize $\bfhatA_\eend \circ \chi$} (P2)
   (P1) edge[->] node[right,yshift=5pt] {\scriptsize $\bfD^2$} (P3)
   (P2) edge[->] node[right,yshift=-5pt] {\scriptsize $E_{\id_\varnothing}$} (P3);
  \end{tikzpicture}
  }
 \end{equation}
\end{proposition}

\begin{proof}
 In order to prove naturality, we need to check that, for every $T \in \KTan(g,g')$,
 the vector
 \begin{align*}
  [(M(\eta^{{} \bcs g'}),S(g'),0) \ast \bfD^2_{J_\eend(T)}]
  &= [M(\eta^{{} \bcs g'}),S(g') \ast O_{J_\eend(T)},0]
 \end{align*}
 determined by Equation~\eqref{E:disc_functor_C} coincides with the vector
 \begin{align*}
  [\chi(T) \ast (M(\eta^{{} \bcs g}),S(g),0)]
  &= [M(T \ast \eta^{{} \bcs g}),S(g),\sigma(C(T \ast \eta^{{} \bcs g}))]
 \end{align*}
 determined by Equation~\eqref{E:surgery_functor_morphisms}.

 We start by noticing that, for every Kirby tangle $T \in \KTan(g,g')$, we have $\sigma(C(T \ast \eta^{{} \bcs g})) = \sigma(C(T))$, and a finite sequence of green-to-red moves yields
 \begin{align*}
  [M(T \ast \eta^{{} \bcs g}),S(g),\sigma(C(T))]
  &= [M(A(T) \ast \eta^{{} \bcs g}),C(T) \cup S(g),0].
 \end{align*}
 Furthermore, we can always consider a regular diagram such that
 \begin{itemize}
  \item the subtangle $A(T)$ is contained in the cylinder $D^2 \times I \subset M_g$,
  \item every handle of $M_g$ is traversed by a single component of $C(T)$,
  \item no component of $C(T)$ runs along two distinct handles of $M_g$.
 \end{itemize}
 This can be achieved by a sequence of signature-preserving Kirby moves, since
 \begin{align*}
  \pic{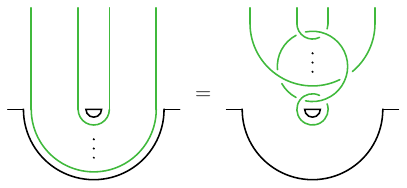}
 \end{align*}
 Under this assumption, up to red signature-preserving Kirby moves, we have
 \begin{align*}
  \pic{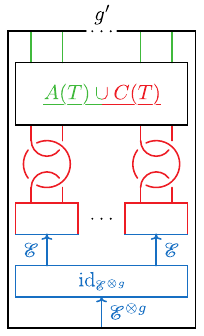}
  &\doteq \pic{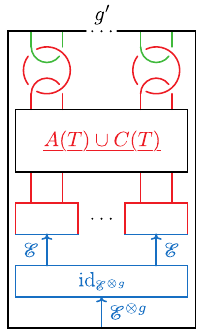}
 \end{align*}
 Therefore, if $T \in \KTan(g,g')$ is a Kirby tangle, and if $\tilde{T} \in \TTan(g,g'+h)$ is a top tangle presentation of $T$, we have a chain of admissible skein equivalences
 \begin{align*}
  \pic{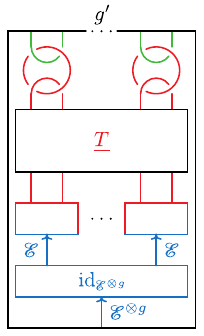}
  &\doteq \pic{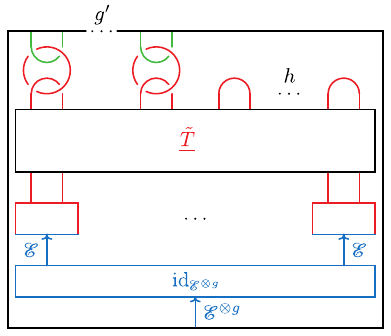} \\
  &\doteq \pic{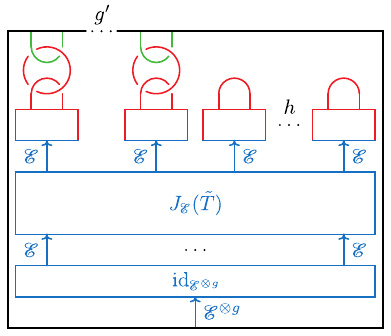} \\
  &\doteq \pic{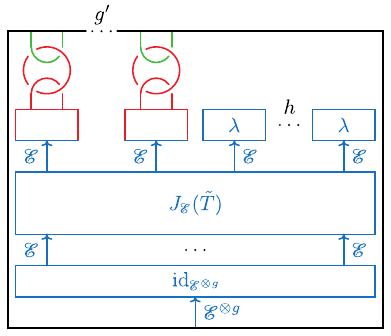} \\
  &\doteq \pic{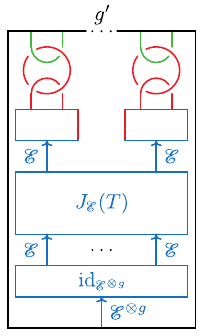}
  \doteq \pic{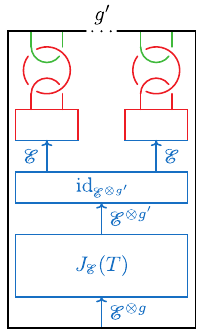} \qedhere
 \end{align*}
\end{proof}

We can move on now to the renormalized \KL{} TQFT $V_\eend$. Equations~\eqref{E:admissible_non-empty_representative_of_empty} and \eqref{E:admissible_connected_representative} yield an isomorphism
\[
 \beta_{\bfSigma} \in \bfhatA_\eend(\varnothing) \left( V_\eend(\bfSigma) \otimes \id_\varnothing,\bfSigma \right)
\]
that is natural in $\bfSigma \in \ACob_\calC$. In other words, we have the following result.

\begin{proposition}\label{P:renormalized_KL_TQFT_inside_ETQFT}
 The family
 \[
  \beta = \left\{ \beta_{\bfSigma} \in \bfhatA_\eend(\varnothing) \left( V_\eend(\bfSigma) \otimes \id_\varnothing,\bfSigma \right) \Bigm\vert \bfSigma \in \ACob_\calC \right\}
 \]
 of isomorphisms of $\bfhatA_\eend(\varnothing)$ defines a natural isomorphism
 \begin{equation}\label{E:commutative_diagram_2}
  \raisebox{-0.5\height+2.5pt}{
  \begin{tikzpicture}[descr/.style={fill=white}] \everymath{\displaystyle}
   \node (P0) at (-1.5,0) {$\ACob_\calC$};
   \node (P1) at (0,1.5) {$\FVect_\Bbbk$};
   \node (P2) at (0,-1.5) {$\bfLex_\Bbbk(\bfhatA_\eend(\varnothing),\bfhatA_\eend(\varnothing))$};
   \node (P3) at (1.5,0) {$\bfhatA_\eend(\varnothing)$};
   \node (P4) at (0,0) {$\phantom{\beta} \Downarrow \beta$};
   \draw
   (P0) edge[->] node[left,yshift=5pt] {\scriptsize $V_\eend$}(P1)
   (P0) edge[->] node[left,yshift=-5pt] {\scriptsize $\bfhatA_\eend$} (P2)
   (P1) edge[->] node[right,yshift=5pt] {\scriptsize $\bfhateta$} (P3)
   (P2) edge[->] node[right,yshift=-5pt] {\scriptsize $E_{\id_\varnothing}$} (P3);
  \end{tikzpicture}
  }
 \end{equation}
\end{proposition}

A direct relation between the \KL{} TQFT $J_\eend$ and its renormalized version $V_\eend$ can now be formulated as a family of natural isomorphisms indexed by $P \in \calB_\calC$. Indeed, as a direct consequence of Propositions~\ref{P:algebraic_and_skein_model}, \ref{P:KL_TQFT_inside_ETQFT}, and \ref{P:renormalized_KL_TQFT_inside_ETQFT}, we have the following result.

\begin{corollary}\label{C:natural_isomorphism}
 For every object $P \in \calB_\calC$, the family of linear maps
 \[
  \Psi_P = \{ \Psi_{P,g} : \calC(F_\calC(P),\eend^{\otimes g}) \to V_\eend((\bfD^2(P))^\dagger \circ \bfSigma_g) \mid g \in \KTan \}
 \]
 of Equation~\eqref{E:algebraic_model} defines a natural isomorphism
 \begin{equation}\label{E:commutative_diagram_3}
  \raisebox{-0.5\height+2.5pt}{
  \begin{tikzpicture}[descr/.style={fill=white}] \everymath{\displaystyle}
   \node (P0) at (-1.5,0) {$\KTan$};
   \node (P1) at (0,1.5) {$\calC$};
   \node (P2) at (0,-1.5) {$\ACob_\calC$};
   \node (P3) at (1.5,0) {$\FVect_\Bbbk$};
   \node (P4) at (0,0) {$\phantom{\Psi_P} \Downarrow \Psi_P$};
   \draw
   (P0) edge[->] node[left,yshift=5pt] {\scriptsize $J_\eend$}(P1)
   (P0) edge[->] node[left,yshift=-5pt] {\scriptsize $(\bfD^2(P))^\dagger \circ \chi(\_)$} (P2)
   (P1) edge[->] node[right,yshift=5pt] {\scriptsize $\calC(F_\calC(P),\_)$} (P3)
   (P2) edge[->] node[right,yshift=-5pt] {\scriptsize $V_\eend$} (P3);
  \end{tikzpicture}
  }
 \end{equation}
\end{corollary}

Notice that Proposition~\ref{P:algebraic_and_skein_model} and Corollary~\ref{C:natural_isomorphism} imply that $\Psi_{P,g}$ is natural in both $P \in \calB_\calC$ and $g \in \KTan$.

\appendix

\section{Finite completion equivalence criterion}\label{A:fc_equivalence}

In this appendix, we provide a proof of Proposition~\ref{P:fc_equivalence}. In order to do this, we first need to fix our notation.

Let us start by recalling that, for a linear category $A$, the Yoneda embedding implies that, for all $F \in \bfHom_\Bbbk(A,\Vect_\Bbbk)$ and $X \in A$, the linear map
\begin{align*}
 (\phi_A)_{F,X} : \Nat(A(X,\_),F) &\to F(X)
\end{align*}
sending every natural transformation $\alpha : A(X,\_) \Rightarrow F$ to the vector
\begin{equation}\label{E:Yoneda}
 (\phi_A)_{F,X}(\alpha) := \alpha_X(\id_X) \in F(X)
\end{equation}
is an isomorphism, whose inverse
\begin{align*}
 (\psi_A)_{F,X} : F(X) &\to \Nat(A(X,\_),F)
\end{align*}
sends every vector $x \in F(X)$ to the natural transformation
whose component
at $Y \in A$ is given by the linear map
\begin{align}
 ((\psi_A)_{F,X}(x))_Y : A(X,Y) &\to F(Y) \label{E:Yoneda_inverse} \\*
 f &\mapsto F(f)(x). \nonumber
\end{align}
Both isomorphisms are natural in $F \in \bfHom_\Bbbk(A,\Vect_\Bbbk)$ and $X \in A$, see for instance \cite[Section~2.4]{K82}.

Next, let us consider a separate linear category $\myuline{A}$. Thanks to \cite[Theorem~6.11]{K82}, the full subcategory $\SLex_\Bbbk(\myuline{A};\Vect_\Bbbk)$ of $\bfHom_\Bbbk(\myuline{A},\Vect_\Bbbk)$ is reflective, which means that the inclusion admits a left adjoint
\[
 L_{\myuline{A}} : \bfHom_\Bbbk(\myuline{A},\Vect_\Bbbk) \to \SLex_\Bbbk(\myuline{A};\Vect_\Bbbk),
\]
called the \textit{reflector}. Then, as explained in \cite[Chapter~IV, Section~3]{M71}, for every $F \in \bfHom_\Bbbk(\myuline{A},\Vect_\Bbbk)$, the adjunction unit $\eta_F : F \Rightarrow L_{\myuline{A}}(F)$ satisfies the following universal property: for every $\myuline{F'} \in \SLex_\Bbbk(\myuline{A};\Vect_\Bbbk)$, the linear map
\begin{align}
 \_ \ast \eta_F : \Nat(L_{\myuline{A}}(F),\myuline{F'}) &\to \Nat(F,\myuline{F'}) \label{E:universal_property_adjunction_unit} \\*
 \alpha &\mapsto \alpha \ast \eta_F \nonumber
\end{align}
is an isomorphism. Furthermore, for every $\myuline{F} \in \SLex_\Bbbk(\myuline{A};\Vect_\Bbbk)$, the adjunction unit $\eta_{\myuline{F}}$ is a natural isomorphism.

Since the proof of Proposition~\ref{P:fc_equivalence} will be based on several intermediate results, let us outline the strategy first. Let $\myuline{F} : \myuline{A} \to \myuline{A'}$ be a separately left exact separate linear functor between separate linear categories. The idea is to construct a quasi-inverse $Q_{\myuline{F}} : \myuline{\hat{A}'} \to \myuline{\hat{A}}$ of the linear functor $\myuline{\hat{F}} : \myuline{\hat{A}} \to \myuline{\hat{A}'}$ by restricting the opposite of the linear functor
\begin{align*}
 S_{\myuline{F}} : \SLex_\Bbbk(\myuline{A'};\Vect_\Bbbk) &\to \SLex_\Bbbk(\myuline{A};\Vect_\Bbbk) \\*
 \myuline{F'} &\mapsto \myuline{F'} \circ \myuline{F}
\end{align*}
to $\smash{\myuline{\hat{A}'}}$. In order to do this, we first need to show that the opposite of $\smash{S_{\myuline{F}}}$ sends $\smash{\myuline{\hat{A}'}}$ to $\smash{\myuline{\hat{A}}}$. As shown in Lemma~\ref{L:fc_equivalence_5}, if $\myuline{F}$ is fully faithful and separate finite limit dense, then this would follow from right exactness of $S_{\myuline{F}}$. However, this is not immediately clear. Indeed, while the linear functor
\begin{align*}
 H_{\myuline{F}} : \bfHom_\Bbbk(\myuline{A'},\Vect_\Bbbk) &\to \bfHom_\Bbbk(\myuline{A},\Vect_\Bbbk) \\*
 F' &\mapsto F' \circ \myuline{F}
\end{align*}
is exact for every separate linear functor $\myuline{F}$, because both limits and colimits in source and target of $H_{\myuline{F}}$ are computed object-wise, separate left exactness of $\myuline{F}$ only implies left exactness of $S_{\myuline{F}}$. This is because, while limits in source and target of $S_{\myuline{F}}$ are still computed object-wise (inclusion being a left exact right adjoint), object-wise colimits are not separately left exact in general, so they have to be reflected. In particular, in order to show right exactness of $S_{\myuline{F}}$, we need to find a natural isomorphism between $L_{\myuline{A}} \circ H_{\myuline{F}} : \bfHom_\Bbbk(\myuline{A'},\Vect_\Bbbk) \to \SLex_\Bbbk(\myuline{A};\Vect_\Bbbk)$ and $S_{\myuline{F}} \circ L_{\myuline{A'}} : \bfHom_\Bbbk(\myuline{A'},\Vect_\Bbbk) \to \SLex_\Bbbk(\myuline{A};\Vect_\Bbbk)$.

Consider $F' \in \bfHom_\Bbbk(\myuline{A'},\Vect_\Bbbk)$, with adjunction unit $\eta_{F'} : F' \Rightarrow L_{\myuline{A'}}(F')$. Applying to the latter the restriction functor $H_{\myuline{F}}$ yields the natural transformation $\eta_{F'} \triangleleft \myuline{F} : F' \circ \myuline{F} \Rightarrow L_{\myuline{A'}}(F') \circ \myuline{F}$. Then, since $L_{\myuline{A'}}(F') \circ \myuline{F}$ is separately left exact, there exists a unique natural transformation $(\gamma_{\myuline{F}})_{F'} : L_{\myuline{A}}(F' \circ \myuline{F}) \Rightarrow L_{\myuline{A'}}(F') \circ \myuline{F}$ such that
\begin{equation}\label{E:commutator_def}
 \eta_{F'} \triangleleft \myuline{F} = (\gamma_{\myuline{F}})_{F'} \ast \eta_{F' \circ \myuline{F}}.
\end{equation}
This yields a natural transformation
\begin{equation*}
 \gamma_{\myuline{F}} : L_{\myuline{A}} \circ H_{\myuline{F}} \Rightarrow S_{\myuline{F}} \circ L_{\myuline{A'}},
\end{equation*}
because the universal property satisfied by the adjunction unit implies
\begin{equation}\label{E:commutator_nat}
 (L_{\myuline{A'}}(\alpha') \triangleleft \myuline{F}) \ast (\gamma_{\myuline{F}})_{F'} = (\gamma_{\myuline{F}})_{F''} \ast L_{\myuline{A}}(\alpha' \triangleleft \myuline{F})
\end{equation}
for every $\alpha' \in \Nat(F',F'')$. As shown in Lemma~\ref{L:fc_equivalence_4}, if $\myuline{F}$ is fully faithful, and if there exists a separate finite limit dense separate linear subcategory $\myuline{B}$ of $\myuline{A}$ such that $\myuline{A'}(\_,\myuline{F}(\myuline{Y}))$ is separately right exact for every $\myuline{Y} \in \myuline{B}$, then $\gamma_{\myuline{F}}$ is invertible.

We start by establishing Lemmas~\ref{L:fc_equivalence_0}--\ref{L:fc_equivalence_3}, which will prepare the ground for the proof of Proposition~\ref{P:fc_equivalence} through Lemmas~\ref{L:fc_equivalence_4} and \ref{L:fc_equivalence_5}.

\begin{lemma}\label{L:fc_equivalence_0}
 Let $\myuline{B}$ be a separate colimit dense full linear subcategory of a separate linear category $\myuline{A}$, and let $\myuline{L}$ be a cone over a separate diagram $\myuline{D} : J \to \myuline{A}$. If, for every $\myuline{X} \in \myuline{B}$, the cone $\myuline{A}(\myuline{X},\myuline{L})$ is the limit of the diagram $\myuline{A}(\myuline{X},\_) \circ \myuline{D} : J \to \Vect_\Bbbk$, then $\myuline{L}$ is the limit of $\myuline{D} : J \to \myuline{A}$.
\end{lemma}

\begin{proof}
 Let $\myuline{C}$ be the full linear subcategory of $\myuline{A}$ whose objects $\myuline{X}$ satisfy
 \begin{align*}
  \myuline{A} \left( \myuline{X},\myuline{L} \right)
  &\cong \lim_{j \in J} \myuline{A}(\myuline{X},\myuline{D}(j)).
 \end{align*}
 Then, the linear category $\myuline{C}$ is closed under separate colimits, because, if a separate diagram $\myuline{D'} : J' \to \myuline{C}$ admits a colimit, then
 \begin{align*}
  \myuline{A} \left( \colim_{j' \in J'} \myuline{D'}(j'),\myuline{L} \right)
  &\cong \lim_{j' \in (J')^\op} \myuline{A} \left( (\myuline{D'})^\op(j'),\myuline{L} \right) \\*
  &\cong \lim_{j' \in (J')^\op} \lim_{j \in J} \myuline{A} \left( (\myuline{D'})^\op(j'),\myuline{D}(j) \right) \\*
  &\cong \lim_{j \in J} \lim_{j' \in (J')^\op} \myuline{A} \left( (\myuline{D'})^\op(j'),\myuline{D}(j) \right) \\*
  &\cong \lim_{j \in J} \myuline{A} \left( \colim_{j' \in J'} \myuline{D'}(j'),\myuline{D}(j) \right).
 \end{align*}
 Furthermore, $\myuline{C}$ contains $\myuline{B}$, by assumption. Therefore, $\myuline{C} = \myuline{A}$. Thanks to Remark~\ref{R:adjoints_and_(co)limits}.$(i)$, representable linear functors jointly reflect limits. \qedhere
\end{proof}

\begin{lemma}\label{L:fc_equivalence_1}
 Let $\myuline{F} : \myuline{A} \to \myuline{A'}$ be a separate linear functor between separate linear categories $\myuline{A}$ and $\myuline{A'}$. If $\myuline{F}$ is fully faithful and separate colimit dense, then $\myuline{F}$ preserves all limits.
\end{lemma}

\begin{proof}
 If a separate diagram $\myuline{D} : J \to \myuline{A}$ admits a limit, then the full linear subcategory of $\myuline{A'}$ given by the image of the separate linear functor $\myuline{F}$, together with the cone
 \[
  \myuline{F}\left( \lim_{j \in J} \myuline{D}(j) \right)
 \]
 over the separate diagram $\myuline{F} \circ \myuline{D} : J \to \myuline{A'}$, satisfies the hypotheses of Lemma~\ref{L:fc_equivalence_0}. Indeed, if $\myuline{X} \in \myuline{A}$, then
 \begin{align*}
  \myuline{A'} \left( \myuline{F}(\myuline{X}),\myuline{F}\left( \lim_{j \in J} \myuline{D}(j) \right) \right)
  &\cong \myuline{A} \left( \myuline{X},\lim_{j \in J} \myuline{D}(j) \right) \\*
  &\cong \lim_{j \in J} \myuline{A} ( \myuline{X}, \myuline{D}(j) ) \\*
  &\cong \lim_{j \in J} \myuline{A'} \left( \myuline{F}(\myuline{X}),\myuline{F}(\myuline{D}(j)) \right).
 \end{align*}
 Therefore,
 \[
  \myuline{F}\left( \lim_{j \in J} \myuline{D}(j) \right) \cong \lim_{j \in J} \myuline{F} (\myuline{D}(j)). \qedhere
 \]
\end{proof}

\begin{lemma}\label{L:fc_equivalence_2}
 Let $\myuline{Y}$ be an object of a small separate linear category $\myuline{A}$. If the linear functor
 \begin{align*}
  \myuline{A}(\_,\myuline{Y}) : \myuline{A}^\op &\to \Vect_\Bbbk \\*
  \myuline{X} &\mapsto \myuline{A}(\myuline{X},\myuline{Y})
 \end{align*}
 is separately right exact, then for every $F \in \bfHom_\Bbbk(\myuline{A},\Vect_\Bbbk)$ the component of the adjunction unit $(\eta_F)_{\myuline{Y}} : F(\myuline{Y}) \to (L_{\myuline{A}}(F))(\myuline{Y})$ is a linear isomorphism.
\end{lemma}

\begin{proof}
 Thanks to \cite[Theorem~4.50]{K82}, the evaluation functor
 \[
  E_{\myuline{Y}} : \bfHom_\Bbbk(\myuline{A},\Vect_\Bbbk) \to \Vect_\Bbbk
 \]
 sending every linear functor $F : \myuline{A} \to \Vect_\Bbbk$ to the vector space $F(\myuline{Y})$ has a right adjoint
 \[
  R_{\myuline{Y}} : \Vect_\Bbbk \to \bfHom_\Bbbk(\myuline{A},\Vect_\Bbbk)
 \]
 sending every vector space $V$ to the linear functor $R_{\myuline{Y}}(V) : \myuline{A} \to \Vect_\Bbbk$ defined by
 \[
  R_{\myuline{Y}}(V)(\myuline{X}) := \Hom_\Bbbk(\myuline{A}(\myuline{X},\myuline{Y}),V),
 \]
 see \cite[Equation~(4.9)]{K82}. In particular, the linear map
 \begin{align*}
  \Psi_{F,V} : \Nat(F,R_{\myuline{Y}}(V)) &\to \Hom_\Bbbk(F(\myuline{Y}),V)
 \end{align*}
 that sends every natural transformation $\alpha : F \Rightarrow R_{\myuline{Y}}(V)$ to the linear map
 \begin{align}
  \Psi_{F,V}(\alpha) : F(\myuline{Y}) &\to V \label{E:evaluation_adjunction_1} \\*
  y &\mapsto \alpha_{\myuline{Y}}(y)(\id_{\myuline{Y}}) \nonumber
 \end{align}
 is invertible and natural in $F \in \bfHom_\Bbbk(\myuline{A},\Vect_\Bbbk)$ and $V \in \Vect_\Bbbk$, and its inverse
 \begin{align*}
  \Theta_{F,V} : \Hom_\Bbbk(F({\myuline{Y}}),V) &\to \Nat(F,R_{\myuline{Y}}(V))
 \end{align*}
 sends every linear map $f : F(\myuline{Y}) \to V$ to the natural transformation whose component at $\myuline{X} \in \myuline{A}$ sends every vector $x \in F({\myuline{X}})$ to the linear map
 \begin{align}
  (\Theta_{F,V}(f))_{\myuline{X}}(x) : \myuline{A}(\myuline{X},\myuline{Y}) &\to V \label{E:evaluation_adjunction_2} \\*
  \myuline{h} &\mapsto f(F(\myuline{h})(x)). \nonumber
 \end{align}
 Now, $R_{\myuline{Y}}(V)$ is separately left exact because, if a separate finite diagram $\myuline{D} : J \to \myuline{A}$ admits a limit, then we have the following chain of linear isomorphisms:
 \begin{align*}
  R_{\myuline{Y}}(V) \left( \lim_{j \in J} \myuline{D}(j) \right)
  &= \Hom_\Bbbk \left( \myuline{A} \left( \lim_{j \in J} \myuline{D}(j),\myuline{Y} \right),V \right) \\*
  &\cong \Hom_\Bbbk \left( \colim_{j \in J^\op} \myuline{A}(\myuline{D}^\op(j),\myuline{Y}),V \right) \\*
  &\cong \lim_{j \in J} \Hom_\Bbbk(\myuline{A}(\myuline{D}(j),\myuline{Y}),V) \\*
  &= \lim_{j \in J} R_{\myuline{Y}}(V)(\myuline{D}(j)).
 \end{align*}
 Since $R_{\myuline{Y}}(V) \in \SLex_\Bbbk(\myuline{A};\Vect_\Bbbk)$, Equation~\eqref{E:universal_property_adjunction_unit} defines an isomorphism
 \begin{align*}
  \_ \ast \eta_F : \Nat(L_{\myuline{A}}(F),R_{\myuline{Y}}(V)) &\to \Nat(F,R_{\myuline{Y}}(V)).
 \end{align*}
 Thanks to Equations~\eqref{E:evaluation_adjunction_1} and \eqref{E:evaluation_adjunction_2}, for every linear map $f : (L_{\myuline{A}}(F))(\myuline{Y}) \to V$, the linear map $\Psi_{F,V}(\Theta_{L_{\myuline{A}}(F),V}(f) \ast \eta_F) : F(\myuline{Y}) \to V$ evaluated at $y \in F(\myuline{Y})$ yields
 \begin{align*}
  \Psi_{F,V}(\Theta_{L_{\myuline{A}}(F),V}(f) \ast \eta_F)(y)
  &= (\Theta_{L_{\myuline{A}}(F),V}(f) \ast \eta_F)_{\myuline{Y}}(y)(\id_{\myuline{Y}}) \\*
  &= f(L_{\myuline{A}}(F)(\id_{\myuline{Y}})((\eta_F)_{\myuline{Y}}(y))) \\*
  &= f((\eta_F)_{\myuline{Y}}(y)).
 \end{align*}
 Therefore,
 \begin{align*}
  \_ \circ (\eta_F)_{\myuline{Y}} : \Hom_\Bbbk((L_{\myuline{A}}(F))(\myuline{Y}),V) &\to \Hom_\Bbbk(F(\myuline{Y}),V)
 \end{align*}
 is an isomorphism for every vector space $V$, which means that
 \[
  \Hom_\Bbbk((\eta_F)_{\myuline{Y}},\_) : \Hom_\Bbbk((L_{\myuline{A}}(F))(\myuline{Y}),\_) \Rightarrow \Hom_\Bbbk(F(\myuline{Y}),\_)
 \]
 is a natural isomorphism. Then, the Yoneda embedding implies that
 \[
  (\eta_F)_{\myuline{Y}} : F(\myuline{Y}) \to (L_{\myuline{A}}(F))(\myuline{Y})
 \]
 is a linear isomorphism. \qedhere
\end{proof}

\begin{lemma}\label{L:fc_equivalence_3}
 Let $\myuline{F} : \myuline{A} \to \myuline{A'}$ be a separately left exact separate linear functor between separate linear categories $\myuline{A}$ and $\myuline{A'}$, and let $\myuline{Y}$ be an object of $\myuline{A}$. If $\myuline{F}$ is fully faithful, and if $\myuline{A'}(\_,\myuline{F}(\myuline{Y})) : (\myuline{A'})^\op \to \Vect_\Bbbk$ is separately right exact, then $\myuline{A}(\_,\myuline{Y}) : \myuline{A}^\op \to \Vect_\Bbbk$ is separately right exact.
\end{lemma}

\begin{proof}
 If a separate finite diagram $\myuline{D} : J \to \myuline{A}$ admits a limit, then we have the following chain of linear isomorphisms:
 \begin{align*}
  \myuline{A} \left( \lim_{j \in J} \myuline{D}(j), \myuline{Y} \right)
  &\cong \myuline{A'} \left( \myuline{F} \left( \lim_{j \in J} \myuline{D}(j) \right), \myuline{F}(\myuline{Y}) \right) \\*
  &\cong \myuline{A'} \left( \lim_{j \in J} \myuline{F}(\myuline{D}(j)), \myuline{F}(\myuline{Y}) \right) \\*
  &\cong \colim_{j \in J^\op} \myuline{A'} (\myuline{F}(\myuline{D}(j)), \myuline{F}(\myuline{Y})) \\*
  &\cong \colim_{j \in J^\op} \myuline{A} (\myuline{D}(j),\myuline{Y}). \qedhere
 \end{align*}
\end{proof}

We are now ready to prove Proposition~\ref{P:fc_equivalence}. We will do this in two steps.

\begin{lemma}\label{L:fc_equivalence_4}
 Let $\myuline{F} : \myuline{A} \to \myuline{A'}$ be a separately left exact separate linear functor between small separate linear categories $\myuline{A}$ and $\myuline{A'}$. If $\myuline{F}$ is fully faithful, and if there exists a separate finite limit dense full separate subcategory $\myuline{B}$ of $\myuline{A}$ such that, for every $\myuline{Y} \in \myuline{B}$, the linear functor $\myuline{A'}(\_,\myuline{F}(\myuline{Y})) : (\myuline{A'})^\op \to \Vect_\Bbbk$ is separately right exact, then the restriction functor $S_{\myuline{F}}$ is right exact.
\end{lemma}

\begin{proof}
 Since, for every $\myuline{Y} \in \myuline{B}$, the linear functor $\myuline{A'}(\_,\myuline{F}(\myuline{Y})) : (\myuline{A'})^\op \to \Vect_\Bbbk$ is separately right exact, then the linear functor $\myuline{A}(\_,\myuline{Y}) : \myuline{A}^\op \to \Vect_\Bbbk$ is separately right exact too, thanks to Lemma~\ref{L:fc_equivalence_3}. Therefore, thanks to Lemma~\ref{L:fc_equivalence_2},
 for every $F' \in \bfHom_\Bbbk(\myuline{A'},\Vect_\Bbbk)$ both components of the adjunction units
 \begin{align*}
  (\eta_{\smash{\myuline{F'} \circ \myuline{F}}})_{\myuline{Y}} : \myuline{F'}(\myuline{F}(\myuline{Y})) &\to L_{\myuline{A}}(\myuline{F'} \circ \myuline{F})(\myuline{Y}), \\*
  (\eta_{\smash{\myuline{F'}}})_{\myuline{F}(\myuline{Y})} : \myuline{F'}(\myuline{F}(\myuline{Y})) &\to L_{\myuline{A'}}(\myuline{F'})(\myuline{F}(\myuline{Y}))
 \end{align*}
 are linear isomorphisms. Then, Equation~\eqref{E:commutator_def} implies that
 \[
  ((\gamma_{\myuline{F}})_{F'})_{\myuline{Y}} = (\eta_{\smash{\myuline{F'}}})_{\myuline{F}(\myuline{Y})} \circ ((\eta_{\smash{\myuline{F'} \circ \myuline{F}}})_{\myuline{Y}})^{-1}
 \]
 is a linear isomorphism too.

 Now, let $\myuline{C}$ be the full linear subcategory of $\myuline{A}$ whose objects $\myuline{X}$ have invertible component
 \[
  ((\gamma_{\myuline{F}})_{F'})_{\myuline{X}} : L_{\myuline{A}}(\myuline{F'} \circ \myuline{F})(\myuline{X}) \to L_{\myuline{A'}}(\myuline{F'})(\myuline{F}(\myuline{X})).
 \]
 Then, the linear category $\myuline{C}$ is closed under separate finite limits, because both $L_{\myuline{A}}(\myuline{F'} \circ \myuline{F}) : \myuline{A} \to \Vect_\Bbbk$ and $L_{\myuline{A'}}(\myuline{F'}) \circ \myuline{F} : \myuline{A} \to \Vect_\Bbbk$ are separately left exact. Furthermore, $\myuline{C}$ contains $\myuline{B}$, by assumption. Therefore, $\myuline{C} = \myuline{A}$. This implies that $\gamma_{\myuline{F}} : L_{\myuline{A}} \circ H_{\myuline{F}} \Rightarrow S_{\myuline{F}} \circ L_{\myuline{A'}}$ is a natural isomorphism.

 Thus, if a finite diagram $D : J \to \bfHom_\Bbbk(\myuline{A'},\Vect_\Bbbk)$ admits a colimit, we have the following chain of natural isomorphisms:
 \begin{align*}
  S_{\myuline{F}} \left( \colim_{j \in J} L_{\myuline{A'}}(D(j)) \right)
  &\cong S_{\myuline{F}} \left( L_{\myuline{A'}} \left( \colim_{j \in J} D(j) \right) \right) \\*
  &\cong L_{\myuline{A}} \left( H_{\myuline{F}} \left( \colim_{j \in J} D(j) \right) \right) \\*
  &\cong \colim_{j \in J} L_{\myuline{A}} (H_{\myuline{F}}(D(j))) \\*
  &\cong \colim_{j \in J} S_{\myuline{F}}(L_{\myuline{A'}}(D(j))). \qedhere
 \end{align*}
\end{proof}

\begin{lemma}\label{L:fc_equivalence_5}
 Let $\myuline{F} : \myuline{A} \to \myuline{A'}$ be a separately left exact separate linear functor between small separate linear categories $\myuline{A}$ and $\myuline{A'}$. If $\myuline{F}$ is fully faithful and separate finite limit dense, and if $S_{\myuline{F}}$ is right exact, then the restriction $Q_{\myuline{F}}$ of its opposite to $\myuline{\hat{A}'}$ is a quasi-inverse of $\myuline{\hat{F}}$.
\end{lemma}

\begin{proof}
 First of all, $Q_{\myuline{F}}$ is left exact, because $S_{\myuline{F}}$ is right exact. Furthermore, the natural transformation $\alpha_{\myuline{F}} : Q_{\myuline{F}} \circ K_{\myuline{A'}} \circ \myuline{F} \Rightarrow K_{\myuline{A}}$ whose component at $\myuline{X} \in \myuline{A}$ is given by the natural transformation
whose component at $\myuline{Y} \in \myuline{A}$ is given by the linear map
 \begin{align}
  ((\alpha_{\myuline{F}})_{\myuline{X}})_{\myuline{Y}} : \myuline{A}(\myuline{X},\myuline{Y}) &\to \myuline{A'}(\myuline{F}(\myuline{X}),\myuline{F}(\myuline{Y})) \label{E:fully_faithful_iso} \\*
  \myuline{h} &\mapsto \myuline{F}(\myuline{h}) \nonumber
 \end{align}
 is an isomorphism, because $\myuline{F}$ is fully faithful.

 Therefore, the image of $Q_{\myuline{F}}$ is contained in $\myuline{\hat{A}}$. Indeed, let $\myuline{B'}$ be the full linear subcategory of $\myuline{A'}$ whose objects $\myuline{X'}$ satisfy
 \begin{align*}
  Q_{\myuline{F}}(K_{\myuline{A'}}(\myuline{X'})) \in \myuline{\hat{A}}.
 \end{align*}
 Then, the linear category $\myuline{B'}$ is closed under separate finite limits, because both $Q_{\myuline{F}}$ and $K_{\myuline{A'}}$ are separately left exact and $\myuline{\hat{A}}$ is finitely complete. Furthermore, $\myuline{B'}$ contains the image of $\myuline{F}$, because, for every $\myuline{X} \in \myuline{A}$, we have
 \begin{align*}
  Q_{\myuline{F}} ( K_{\myuline{A'}} ( \myuline{F}(\myuline{X}) ) )
  &\cong K_{\myuline{A}}(\myuline{X}).
 \end{align*}
 Therefore, $\myuline{B'} = \myuline{A'}$. Similarly, let $\myuline{C'}$ be the full linear subcategory of $\myuline{\hat{A}'}$ whose objects $\myuline{F'}$ satisfy
 \begin{align*}
  Q_{\myuline{F}}(\myuline{F'}) \in \myuline{\hat{A}}.
 \end{align*}
 Then, the linear category $\myuline{C'}$ is closed under finite limits, because $Q_{\myuline{F}}$ is left exact and $\myuline{\hat{A}}$ is finitely complete. Furthermore, $\myuline{C'}$ contains all the representable linear functors $\myuline{A'}(\myuline{X'},\_)$ for $\myuline{X'} \in \myuline{A'}$, as we just showed. Therefore, $\myuline{C'} = \myuline{\hat{A}'}$.

 Now $Q_{\myuline{F}}$ is left adjoint to $\myuline{\hat{F}}$. Indeed, thanks to the Yoneda embedding, the linear map
 \[
  (\Phi_{\myuline{F\smash{'}}})_{\myuline{X}} : \myuline{\hat{A}'}(\myuline{F'},\myuline{\hat{F}}(K_{\myuline{A}}(\myuline{X}))) \to \myuline{\hat{A}}(Q_{\myuline{F}}(\myuline{F'}),K_{\myuline{A}}(\myuline{X}))
 \]
 defined by
 \[
  (\Phi_{\myuline{F\smash{'}}})_{\myuline{X}} := (\psi_{\myuline{A}})_{\myuline{F'} \circ \myuline{F},\myuline{X}} \circ (\phi_{\myuline{A'}})_{\myuline{F'},\myuline{F}(\myuline{X})} \circ (\theta_{\myuline{F}})_{\myuline{X}}
 \]
 is an isomorphism that is natural in $\myuline{F'} \in \myuline{\hat{A}'}$ and $\myuline{X} \in \myuline{A}$.
 In particular, since both $\myuline{\hat{A}'}(\myuline{F'},\myuline{\hat{F}}(\_)) : \myuline{\hat{A}} \to \Vect_\Bbbk$ and $\myuline{\hat{A}}(Q_{\myuline{F}}(\myuline{F'}),\_) : \myuline{\hat{A}} \to \Vect_\Bbbk$ are left exact, the universal property satisfied by finite completions yields a unique natural isomorphism
 \[
  \hat{\Phi}_{\myuline{F\smash{'}}} : \myuline{\hat{A}'}(\myuline{F'},\myuline{\hat{F}}(\_)) \Rightarrow \myuline{\hat{A}}(Q_{\myuline{F}}(\myuline{F'}),\_)
 \]
 satisfying
 \[
  \hat{\Phi}_{\myuline{F\smash{'}}} \triangleleft K_{\myuline{A}} = \Phi_{\myuline{F\smash{'}}}.
 \]
 Full faithfulness of the restriction along $K_{\myuline{A}}$ implies that $\hat{\Phi}_{\myuline{F\smash{'}}}$ is natural in $\myuline{F'} \in \myuline{\hat{A}'}$.

 Notice that, for every natural transformation $\alpha' \in \Nat(\myuline{A'}(\myuline{F}(\myuline{X}),\_),\myuline{F'})$, Equation~\eqref{E:fully_faithful_iso} yields
 \begin{equation}\label{E:adjunction_computation}
  (\Phi_{\myuline{F\smash{'}}})_{\myuline{X}}(\alpha') = (\alpha' \triangleleft \myuline{F}) \ast (\alpha_{\myuline{F}})_{\myuline{X}} \in \Nat(\myuline{A}(\myuline{X},\_),\myuline{F'} \circ \myuline{F}).
 \end{equation}
 Indeed, thanks to Equations~\eqref{E:Yoneda} and \eqref{E:Yoneda_inverse}, $(\Phi_{\myuline{F\smash{'}}})_{\myuline{X}}$ sends every natural transformation $\alpha' : \myuline{A'}(\myuline{F}(\myuline{X}),\_) \Rightarrow \myuline{F'}$ to the natural transformation whose component at $\myuline{Y} \in \myuline{A}$ is given by the linear map
 \begin{align*}
  ((\Phi_{\myuline{F\smash{'}}})_{\myuline{X}}(\alpha'))_{\myuline{Y}} : \myuline{A}(\myuline{X},\myuline{Y}) &\to \myuline{F'}(\myuline{F}(\myuline{Y})) \\*
  \myuline{h} &\mapsto \myuline{F'}(\myuline{F}(\myuline{h}))(\alpha'_{\myuline{F}(\myuline{X})}(\id_{\myuline{F}(\myuline{X})})) = \alpha'_{\myuline{F}(\myuline{Y})}(\myuline{F}(\myuline{h})).
 \end{align*}

 Then, the adjunction counit $\hat{\epsilon} : Q_{\myuline{F}} \circ \myuline{\hat{F}} \Rightarrow \id_{\myuline{\hat{A}}}$ is a natural isomorphism. Indeed, the universal property of finite completions implies that, since both $Q_{\myuline{F}} \circ \myuline{\hat{F}} : \myuline{\hat{A}} \to \myuline{\hat{A}}$ and $\id_{\myuline{\hat{A}}} : \myuline{\hat{A}} \to \myuline{\hat{A}}$ are left exact, then $\hat{\epsilon}$ is a natural isomorphism if and only if $\epsilon := \hat{\epsilon} \triangleleft K_{\myuline{A}}$ is. Now, Equations~\eqref{E:fully_faithful_iso} and \eqref{E:adjunction_computation} yield $\epsilon = \alpha_{\myuline{F}}$, because its component at $\myuline{X} \in \myuline{A}$ is the natural transformation whose component at $\myuline{Y} \in \myuline{A}$ is the linear map
 \begin{align*}
  (\epsilon_{\myuline{X}})_{\myuline{Y}} : \myuline{A}(\myuline{X},\myuline{Y}) &\to \myuline{A'}(\myuline{F}(\myuline{X}),\myuline{F}(\myuline{Y})) \\*
  \myuline{h} &\mapsto ((\Phi_{\myuline{\hat{F}}(K_{\myuline{A}}(\myuline{X}))})_{\myuline{X}}(\id_{\myuline{\hat{F}}(K_{\myuline{A}}(\myuline{X}))}))_{\myuline{Y}}(\myuline{h})
  = \myuline{F}(\myuline{h}).
 \end{align*}

 Finally, the adjunction unit $\hat{\eta} : \id_{\myuline{\hat{A}'}} \Rightarrow \myuline{\hat{F}} \circ Q_{\myuline{F}}$ is a natural isomorphism too. Indeed, the universal property of finite completions implies again that, since both $\smash{\id_{\myuline{\hat{A}'}} : \myuline{\hat{A}'} \to \myuline{\hat{A}'}}$ and $\myuline{\hat{F}} \circ Q_{\myuline{F}} : \myuline{\hat{A}'} \to \myuline{\hat{A}'}$ are left exact, then $\hat{\eta}$ is a natural isomorphism if and only if $\eta := \hat{\eta} \triangleleft K_{\myuline{A'}}$ is.
 Now, let $\myuline{D'}$ be the full linear subcategory of $\myuline{A'}$ whose objects $\myuline{X'}$ have invertible component
 \[
  \eta_{\myuline{X'}} : K_{\myuline{A'}}(\myuline{X'}) \to \myuline{\hat{F}}(Q_{\myuline{F}}(K_{\myuline{A'}}(\myuline{X'}))).
 \]
 Then, the linear category $\myuline{D'}$ is closed under separate finite limits, because both $K_{\myuline{A'}} : \myuline{A'} \to \myuline{\hat{A}'}$ and $\myuline{\hat{F}} \circ Q_{\myuline{F}} \circ K_{\myuline{A'}} : \myuline{A'} \to \myuline{\hat{A}'}$ are separately left exact. Furthermore, $\myuline{D'}$ contains the image of $\myuline{F}$, because, since $\hat{\epsilon}$ is invertible, then also $\myuline{\hat{F}} \triangleright \hat{\epsilon}$ is, and so the zig-zag identity
 \[
  (\myuline{\hat{F}} \triangleright \hat{\epsilon}) \ast (\hat{\eta} \triangleleft \myuline{\hat{F}}) = \id_{\myuline{\hat{F}}}
 \]
 implies that $\hat{\eta} \triangleleft \myuline{\hat{F}}$ is invertible too. Therefore, $\myuline{D'} = \myuline{A'}$. It follows then that $\eta : K_{\myuline{A'}} \Rightarrow \myuline{\hat{F}} \circ Q_{\myuline{F}} \circ K_{\myuline{A'}}$ is a natural isomorphism. \qedhere
\end{proof}

This completes the proof of Proposition~\ref{P:fc_equivalence}.

\section{Dualities, adjunctions, and actions in cobordism \texorpdfstring{$2$-categories}{2-categories}}\label{A:dualities_adjunctions}

In this appendix, we fix some notations for operations in $\bfCob_\calC$ that are used throughout the construction.

Let us start from dualities. The \textit{dual} $\bfGamma^*$ of an object $\bfGamma$ of $\bfCob_\calC$ is the \dmnsnl{1} manifold $\Gamma^*$ obtained from $\Gamma$ by reversing its orientation. Clearly $(\bfGamma^*)^* = \bfGamma$.

The \textit{right evaluation}
\[
 \rev_{\bfGamma} : \bfGamma \disjun \bfGamma^* \to \varnothing
\]
of an object $\bfGamma$ of $\bfCob_\calC$ is
\begin{equation}\label{E:right_evaluation}
 \rev_{\bfGamma} := (\rev_\Gamma,H_1(I \times \Gamma)),
\end{equation}
where $\rev_\Gamma$ is the \dmnsnl{2} cobordism obtained from $I \times \Gamma$ by declaring all boundary components to be incoming. The \textit{right coevaluation}
\[
 \rcoev_{\bfGamma} : \varnothing \to \bfGamma^* \disjun \bfGamma
\]
of an object $\bfGamma$ of $\bfCob_\calC$ is
\begin{equation}\label{E:right_coevaluation}
 \rcoev_{\bfGamma} := (\rcoev_\Gamma,H_1(I \times \Gamma)),
\end{equation}
where $\rcoev_\Gamma$ is the \dmnsnl{2} cobordism obtained from $I \times \Gamma$ by declaring all boundary components to be outgoing. The \textit{left evaluation}
\[
 \lev_{\bfGamma} : \bfGamma^* \disjun \bfGamma \to \varnothing
\]
of an object $\bfGamma$ of $\bfCob_\calC$ is
\begin{equation}\label{E:left_evaluation}
 \lev_{\bfGamma} := \rev_{\bfGamma^*}.
\end{equation}
The \textit{left coevaluation}
\[
 \lcoev_{\bfGamma} : \varnothing \to \bfGamma \disjun \bfGamma^*
\]
of an object $\bfGamma$ of $\bfCob_\calC$ is
\begin{equation}\label{E:left_coevaluation}
 \lcoev_{\bfGamma} := \rcoev_{\bfGamma^*}.
\end{equation}
Right evaluation and coevaluation make $\bfGamma^*$ into a right dual of $\bfGamma$, and left evaluation and coevaluation make $\bfGamma^*$ into a left dual of $\bfGamma$. Indeed, we have
\begin{align*}
 (\id_{\bfGamma} \disjun \lev_{\bfGamma}) \circ (\lcoev_{\bfGamma} \disjun \id_{\bfGamma}) &= \id_{\bfGamma}, &
 (\lev_{\bfGamma} \disjun \id_{\bfGamma^*}) \circ (\id_{\bfGamma^*} \disjun \lcoev_{\bfGamma}) &= \id_{\bfGamma^*}, \\*
 (\rev_{\bfGamma} \disjun \id_{\bfGamma}) \circ (\id_{\bfGamma} \disjun \rcoev_{\bfGamma}) &= \id_{\bfGamma}, &
 (\id_{\bfGamma^*} \disjun \rev_{\bfGamma}) \circ (\rcoev_{\bfGamma} \disjun \id_{\bfGamma^*}) &= \id_{\bfGamma^*}.
\end{align*}
Notice that we can formulate these identities as equalities, instead of isomorphisms, because we can assume $\bfCob_\calC$ to be quasi-strict, see \cite[Definition~2.28 \& Theorem~2.96]{S11}.

The \textit{dual}
\[
 \bfSigma^* : (\bfGamma')^* \to \bfGamma^*
\]
of a \mrphsm{1} $\bfSigma : \bfGamma \to \bfGamma'$ of $\bfCob_\calC$ is
\begin{equation}\label{E:dual_1-morphism}
 \bfSigma^* := (\Sigma^*,P,\Lagr),
\end{equation}
where $\Sigma^*$ is the \dmnsnl{2} cobordism from $(\Gamma')^*$ to $\Gamma^*$ obtained from $\Sigma$ by exchanging incoming and outgoing boundary identifications. Right evaluation and coevaluation make $\bfSigma^*$ into a right dual of $\bfSigma$, and left evaluation and coevaluation make $\bfSigma^*$ into a left dual of $\bfSigma$. Indeed, we have
\begin{align*}
 \bfSigma^* &= (\lev_{\bfGamma'} \disjun \id_{\bfGamma^*}) \circ (\id_{(\bfGamma')^*} \disjun \bfSigma \disjun \id_{\bfGamma^*}) \circ (\id_{(\bfGamma')^*} \disjun \lcoev_{\bfGamma}) \\*
 &= (\id_{\bfGamma^*} \disjun \rev_{\bfGamma'}) \circ (\id_{\bfGamma^*} \disjun \bfSigma \disjun \id_{(\bfGamma')^*}) \circ (\rcoev_{\bfGamma} \disjun \id_{(\bfGamma')^*}).
\end{align*}

The \textit{dual}
\[
 \bfM^* : \bfSigma^* \to (\bfSigma')^*
\]
of a \mrphsm{2} $\bfM : \bfSigma \to \bfSigma'$ of $\bfCob_\calC$ is the equivalence class of
\begin{equation}\label{E:dual_2-morphism}
 \bfM^* := (M^*,T,\sig),
\end{equation}
where $M^*$ is the \dmnsnl{3} cobordism with corners from $\Sigma^*$ to $(\Sigma')^*$ obtained from $M$ by exchanging incoming and outgoing vertical boundary identifications. Right evaluation and coevaluation make $\bfM^*$ into a right dual of $\bfM$, and left evaluation and coevaluation make $\bfM^*$ into a left dual of $\bfM$. Indeed, we have
\begin{align*}
 \bfM^* &= (\id_{\lev_{\bfGamma'}} \disjun \id_{\id_{\bfGamma^*}}) \circ (\id_{\id_{(\bfGamma')^*}} \disjun \bfM \disjun \id_{\id_{\bfGamma^*}}) \circ (\id_{\id_{(\bfGamma')^*}} \disjun \id_{\lcoev_{\bfGamma}}) \\*
 &= (\id_{\id_{\bfGamma^*}} \disjun \id_{\rev_{\bfGamma'}}) \circ (\id_{\id_{\bfGamma^*}} \disjun \bfM \disjun \id_{\id_{(\bfGamma')^*}}) \circ (\id_{\rcoev_{\bfGamma}} \disjun \id_{\id_{(\bfGamma')^*}}).
\end{align*}

The functor
\[
 \_^* : \bfCob_\calC(\bfGamma,\bfGamma') \to \bfCob_\calC((\bfGamma')^*,\bfGamma^*).
\]
is an equivalence, whose inverse is the functor
\[
 \_^* : \bfCob_\calC((\bfGamma')^*,\bfGamma^*) \to \bfCob_\calC(\bfGamma,\bfGamma').
\]

Next, let us move on to adjunctions. The \textit{adjoint}
\[
 \bfSigma^\dagger : \bfGamma' \to \bfGamma
\]
to a \mrphsm{1} $\bfSigma : \bfGamma \to \bfGamma'$ of $\bfCob_\calC$ is
\begin{equation}\label{E:adjoint_1-morphism}
 \bfSigma^\dagger := (\Sigma^\dagger,P^\dagger,\Lagr),	
\end{equation}
where $\Sigma^\dagger$ is the \dmnsnl{2} cobordism from $\Gamma'$ to $\Gamma$ obtained from $\Sigma$ by reversing its orientation and exchanging incoming and outgoing boundary identifications, and $P^\dagger \subset \Sigma^\dagger$ is the $\calC$-labeled blue set obtained from $P \subset \Sigma$ by reversing its orientation. Clearly $(\bfSigma^\dagger)^\dagger = \bfSigma$.

The \textit{right adjunction unit}
\[
 \radun_{\bfSigma} : \id_{\bfGamma} \Rightarrow \bfSigma^\dagger \circ \bfSigma
\]
of a \mrphsm{1} $\bfSigma : \bfGamma \to \bfGamma'$ of $\bfCob_\calC$ is the equivalence class of
\begin{equation}\label{E:right_adjunction_unit}
 \radun_{\bfSigma} := (\radun_\Sigma,\radun_P,0),
\end{equation}
where $\radun_\Sigma$ is the \dmnsnl{3} cobordism with corners from $I \times \Gamma$ to $\Sigma^\dagger \circ \Sigma$ obtained from $\Sigma \times I$ by appropriately decomposing and rearranging boundary identifications, and $\radun_P$ denotes the $\calC$-labeled blue framed tangle from $\varnothing$ to $P \cup P^\dagger$ obtained from the identity $P \times I$ by declaring all boundary vertices to be outgoing. The \textit{right adjunction counit}
\[
 \radcoun_{\bfSigma} : \bfSigma \circ \bfSigma^\dagger \Rightarrow \id_{\bfGamma'}
\]
of a \mrphsm{1} $\bfSigma : \bfGamma \to \bfGamma'$ of $\bfCob_\calC$ is the equivalence class of
\begin{equation}\label{E:right_adjunction_counit}
 \radcoun_{\bfSigma} := (\radcoun_\Sigma,\radcoun_P,0),
\end{equation}
where $\radcoun_\Sigma$ is the \dmnsnl{3} cobordism with corners from $\Sigma \circ \Sigma^\dagger$ to $I \times \Gamma'$ obtained from $\Sigma \times I$ by appropriately decomposing and rearranging boundary identifications, and $\radcoun_P$ denotes the $\calC$-labeled blue framed tangle from $P^\dagger \cup P$ to $\varnothing$ obtained from the identity $P \times I$ by declaring all boundary vertices to be incoming. The \textit{left adjunction unit}
\[
 \ladun_{\bfSigma} : \id_{\bfGamma'} \Rightarrow \bfSigma \circ \bfSigma^\dagger
\]
of a \mrphsm{1} $\bfSigma : \bfGamma \to \bfGamma'$ of $\bfCob_\calC$ is
\begin{equation}\label{E:left_adjunction_unit}
 \ladun_{\bfSigma} := \radun_{\bfSigma^\dagger}.
\end{equation}
The \textit{left adjunction counit}
\[
 \ladcoun_{\bfSigma} : \bfSigma^\dagger \circ \bfSigma \Rightarrow \id_{\bfGamma}
\]
of a \mrphsm{1} $\bfSigma : \bfGamma \to \bfGamma'$ of $\bfCob_\calC$ is
\begin{equation}\label{E:left_adjunction_counit}
 \ladcoun_{\bfSigma} := \radcoun_{\bfSigma^\dagger}.
\end{equation}
Right adjunction unit and counit make $\bfSigma^\dagger$ right adjoint to $\bfSigma$, and left adjunction unit and counit make $\bfSigma^\dagger$ left adjoint to $\bfSigma$. Indeed, we have
\begin{align*}
 (\bfSigma \triangleright \ladcoun_{\bfSigma}) \ast (\ladun_{\bfSigma} \triangleleft \bfSigma) &= \id_{\bfSigma}, &
 (\ladcoun_{\bfSigma} \triangleleft \bfSigma^\dagger) \ast (\bfSigma^\dagger \triangleright \ladun_{\bfSigma}) &= \id_{\bfSigma^\dagger}, \\*
 (\radcoun_{\bfSigma} \triangleleft \bfSigma) \ast (\bfSigma \triangleright \radun_{\bfSigma}) &= \id_{\bfSigma}, &
 (\bfSigma^\dagger \triangleright \radcoun_{\bfSigma}) \ast (\radun_{\bfSigma} \triangleleft \bfSigma^\dagger) &= \id_{\bfSigma^\dagger}.
\end{align*}

The \textit{mate}
\[
 \bfM^\dagger : (\bfSigma')^\dagger \to \bfSigma^\dagger
\]
of a \mrphsm{2} $\bfM : \bfSigma \to \bfSigma'$ of $\bfCob_\calC$ is the equivalence class of
\begin{equation}\label{E:adjoint_2-morphism}
 (M^\dagger,T^\dagger,\sig),
\end{equation}
where $M^\dagger$ is the \dmnsnl{3} cobordism with corners from $(\Sigma')^\dagger$ to $\Sigma^\dagger$ obtained from $M$ by exchanging incoming and outgoing horizontal and vertical boundary identifications, and $T^\dagger \subset M^\dagger$ is the $\calC$-labeled bichrome graph obtained from $T \subset M$ by exchanging incoming and outgoing vertices. Right adjunction unit and counit make $\bfM^\dagger$ into a right mate of $\bfM$, and left adjunction unit and counit make $\bfM^\dagger$ into a left mate of $\bfM$. Indeed, we have
\begin{align*}
 \bfM^\dagger
 &= (\ladcoun_{\bfSigma'} \triangleleft \bfSigma^\dagger) \ast ((\bfSigma')^\dagger \triangleright \bfM \triangleleft \bfSigma^\dagger) \ast ((\bfSigma')^\dagger \triangleright \ladun_{\bfSigma}) \\*
 &= (\bfSigma^\dagger \triangleright \radcoun_{\bfSigma'}) \ast (\bfSigma^\dagger \triangleright \bfM \triangleleft (\bfSigma')^\dagger) \ast (\radun_{\bfSigma} \triangleleft (\bfSigma')^\dagger).
\end{align*}
Furthermore, we have
\begin{align*}
 \radun_{\bfSigma^*} &= (\ladun_{\bfSigma})^*, &
 \radcoun_{\bfSigma^*} &= (\ladcoun_{\bfSigma})^*.
\end{align*}

The functor
\[
 \_^\dagger : \bfCob_\calC(\bfGamma,\bfGamma')^\op \to \bfCob_\calC(\bfGamma',\bfGamma).
\]
is an equivalence, whose inverse is the opposite of the functor
\[
 \_^\dagger : \bfCob_\calC(\bfGamma',\bfGamma)^\op \to \bfCob_\calC(\bfGamma,\bfGamma').
\]

Combining dualities and adjunctions, we obtain transpositions. The \textit{transpose}
\[
 \bfSigma^\rmT : \bfGamma^* \to (\bfGamma')^*
\]
of a \mrphsm{1} $\bfSigma : \bfGamma \to \bfGamma'$ of $\bfCob_\calC$ is
\begin{equation}\label{E:transpose_1-morphism}
 \bfSigma^\rmT := (\bfSigma^*)^\dagger = (\bfSigma^\dagger)^*.
\end{equation}

The \textit{transpose}
\[
 \bfM^\rmT : (\bfSigma')^\rmT \to \bfSigma^\rmT
\]
of a \mrphsm{2} $\bfM : \bfSigma \to \bfSigma'$ of $\bfCob_\calC$ is
\begin{equation}\label{E:transpose_2-morphism}
 \bfM^\rmT := (\bfM^*)^\dagger = (\bfM^\dagger)^*.
\end{equation}

The functor
\[
 \_^\rmT : \bfCob_\calC(\bfGamma,\bfGamma')^\op \to \bfCob_\calC(\bfGamma^*,(\bfGamma')^*).
\]
is an equivalence, whose inverse is the opposite of the functor
\[
 \_^\rmT : \bfCob_\calC(\bfGamma^*,(\bfGamma')^*)^\op \to \bfCob_\calC(\bfGamma,\bfGamma').
\]

Let us finish with actions. The \textit{left action}
\[
 f \cdot \bfSigma : \bfGamma \to \bfGamma''
\]
of a diffeomorphism $f : \bfGamma' \to \bfGamma''$ on a \mrphsm{1} $\bfSigma : \bfGamma \to \bfGamma'$ of $\bfCob_\calC$ is
\begin{equation}\label{E:left_action_1-morphism}
 f \cdot \bfSigma := (f \cdot \Sigma,P,\Lagr),	
\end{equation}
where $f \cdot \Sigma$ is the \dmnsnl{2} cobordism from $\Gamma$ to $\Gamma''$ obtained from $\Sigma$ by precomposing its outgoing boundary identification with $f^{-1}$.

The \textit{right action}
\[
 \bfSigma \cdot f : \bfGamma \to \bfGamma''
\]
of a diffeomorphism $f : \bfGamma \to \bfGamma'$ on a \mrphsm{1} $\bfSigma : \bfGamma' \to \bfGamma''$ of $\bfCob_\calC$ is
\begin{equation}\label{E:right_action_1-morphism}
 \bfSigma \cdot f := (\Sigma \cdot f,P,\Lagr),	
\end{equation}
where $\Sigma \cdot f$ is the \dmnsnl{2} cobordism from $\Gamma$ to $\Gamma''$ obtained from $\Sigma$ by precomposing its incoming boundary identification with $f$.

If $\bfSigma, \bfSigma' : \bfGamma \to \bfGamma'$ are \mrphsms{1} of $\bfCob_\calC$, a \textit{compatible diffeomorphism} $f : \bfSigma \to \bfSigma'$ is an isomorphism of cobordisms $f : \Sigma \to \Sigma'$ satisfying $f(P) = P'$ and $f_*(\Lagr) = \Lagr'$.

The \textit{left action}
\[
 f \cdot \bfM : \bfSigma \to \bfSigma''
\]
of a compatible diffeomorphism $f : \bfSigma' \to \bfSigma''$ on a \mrphsm{2} $\bfM : \bfSigma \to \bfSigma'$ of $\bfCob_\calC$ is
\begin{equation}\label{E:left_action_2-morphism}
 f \cdot \bfM := (f \cdot M,T,n),	
\end{equation}
where $f \cdot M$ is the \dmnsnl{3} cobordism with corners from $\Sigma$ to $\Sigma''$ obtained from $M$ by precomposing its outgoing horizontal boundary identification with $f^{-1}$.

The \textit{right action}
\[
 \bfM \cdot f : \bfSigma \to \bfSigma''
\]
of a compatible diffeomorphism $f : \bfSigma \to \bfSigma'$ on a \mrphsm{2} $\bfM : \bfSigma' \to \bfSigma''$ of $\bfCob_\calC$ is
\begin{equation}\label{E:right_action_2-morphism}
 \bfM \cdot f := (M \cdot f,T,n),	
\end{equation}
where $M \cdot f$ is the \dmnsnl{3} cobordism with corners from $\Sigma$ to $\Sigma''$ obtained from $M$ by precomposing its incoming horizontal boundary identification with $f$.

\section{Proof of Proposition~\ref{P:KL_TQFT}}\label{A:proof_KL_TQFT}

\begin{proof}[Proof of Proposition~\ref{P:KL_TQFT}]
 Before showing that $J_\eend : \KTan \to \calC$ is a well-defined functor, let us start by showing that every Kirby tangle $T : g \to g'$ admits a top tangle presentation $\tilde{T} : g \to g''$. In order to define it, let $q_1, \ldots, q_h$ denote a family of points uniformly distributed to the right of the outgoing boundary of $T$ on $D^1 \times \{ (0,1) \} \subset D^2 \times I$, ordered from left to right, where $g'' = g' + h$. Let us choose an ordering of the components of $C(T)$ (from $C(T)_1$ to $C(T)_h$), and let us choose an orientation and a basepoint $p_j$ on the $j$th component $C(T)_j$ of $C(T)$ for every integer $1 \leqs j \leqs h$. Then, let us consider a family $\gamma$ of pairwise disjoint framed arcs $\gamma_j$ joining $p_j$ to $q_j$ for every integer $1 \leqs j \leqs h$, whose interiors are disjoint from $T$. We denote by $\tilde{T}$ the green tangle obtained from $T$ by pulling $C(T)_j$ along $\gamma_j$ and above $q_j$ for every $1 \leqs j \leqs h$, as shown
 \[
  \pic{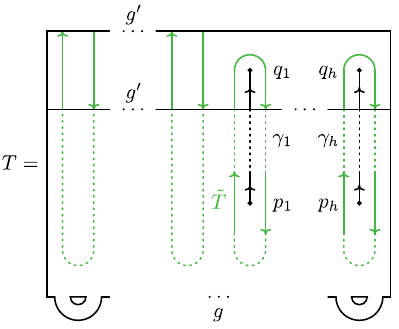}
 \]
 Then, by construction, $\tilde{T}$ is a top tangle presentation of $T$. In order to prove that $J_\eend$ is well-defined, we need to show that it is independent of all the choices made (paths, order, orientations and basepoints), and that it is invariant under signature-preserving Kirby moves (slam-dunks and slides).

 \textit{Independence of paths.} First, we claim that $J_\eend$ is independent of the choice of the framed path $\gamma_j$ for every integer $1 \leqs j \leqs h$. In order to prove this, it is sufficient to show that $\gamma_j$ can be exchanged for another framed path $\gamma'_j$. Up to isotopy, the exchange of $\gamma_j$ with $\gamma'_j$ is represented by the following picture.
 \[
  \pic{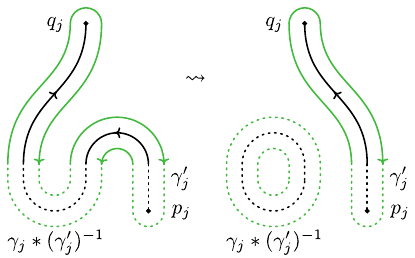}
 \]
 This means that, on one hand,
 \[
  \pic{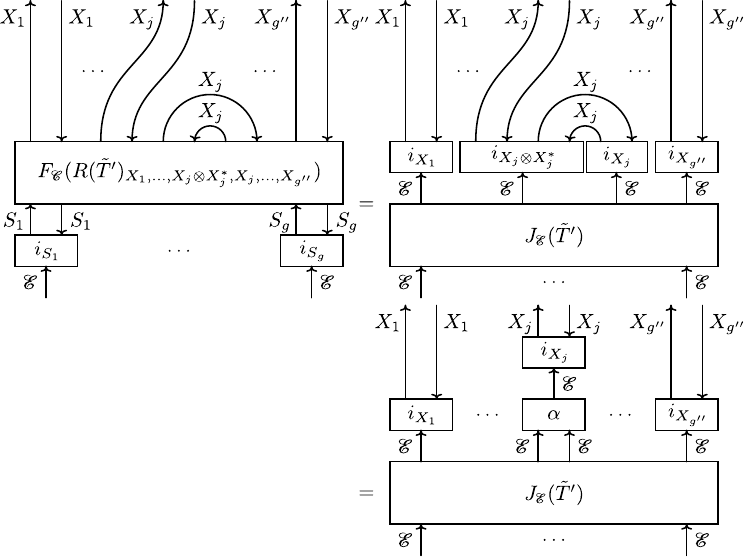}
 \]
 while, on the other hand,
 \[
  \pic{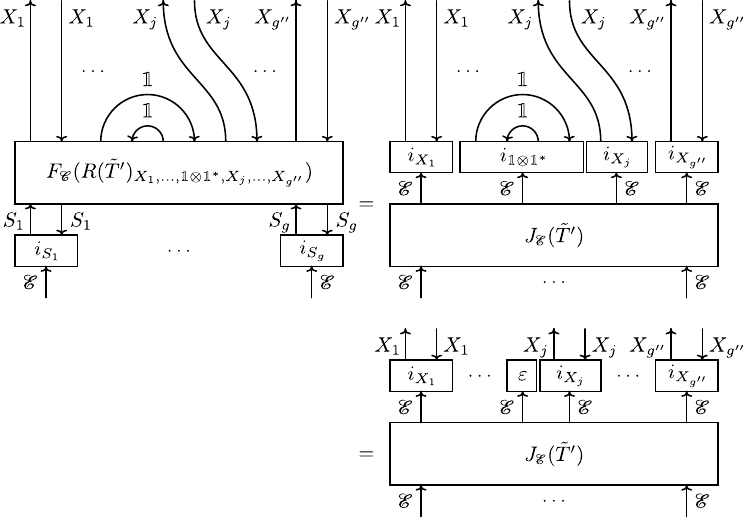}
 \]
 Then the claim follows from the fact that the integral $\lambda$ intertwines the adjoint action $\alpha$ of $\eend$, which implies
 \begin{equation}\label{E:path}
  \pic{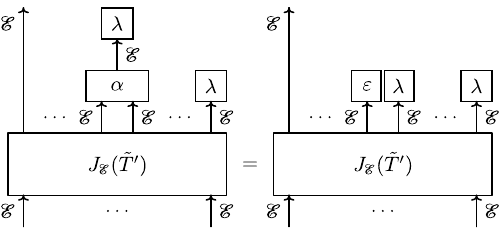}
 \end{equation}

 \textit{Independence of order.} Next, we claim that $J_\eend$ is independent of the choice of the ordering of the components of $C(T)$. In order to prove this, it is sufficient to show that the components $C(T)_j$ and $C(T)_{j+1}$ can be transposed for every integer $1 \leqs j < h$. Up to isotopy, this operation is represented by the following picture.
 \[
  \pic{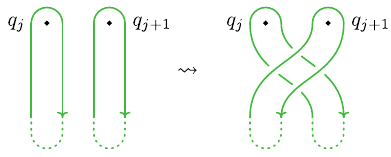}
 \]
 This means that
 \[
  \pic{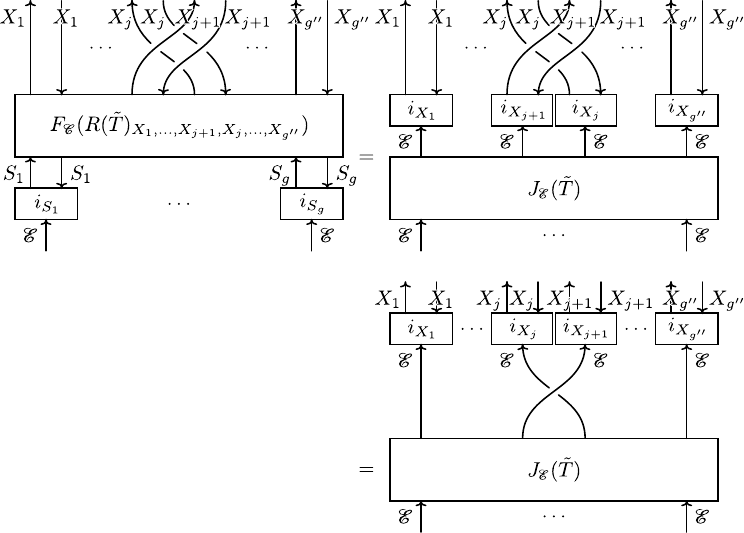}
 \]
 Then the claim follows from the fact that $\lambda$ is a form on $\eend$, which implies
 \begin{equation}\label{E:order}
  \pic{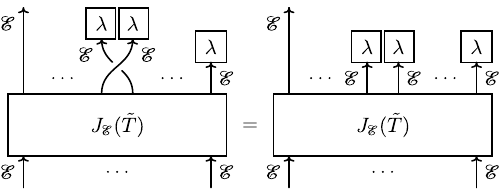}
 \end{equation}

 \textit{Independence of orientations.} Next, we claim that $J_\eend$ is independent of the choice of the orientation of $C(T)_j$ for every integer $1 \leqs j \leqs h$. Indeed, up to isotopy, an orientation reversal on $C(T)_j$ is represented by the following picture.
 \[
  \pic{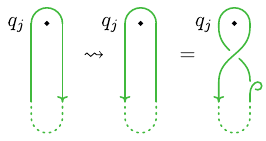}
 \]
 This means that
 \[
  \pic{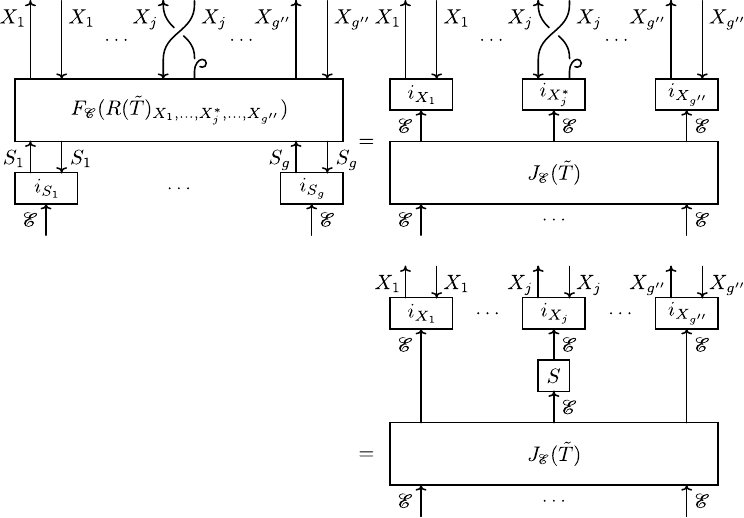}
 \]
 Then the claim follows from the fact that $\lambda$ is an\-ti\-pode-in\-var\-i\-ant, which implies
 \begin{equation}\label{E:orientation}
  \pic{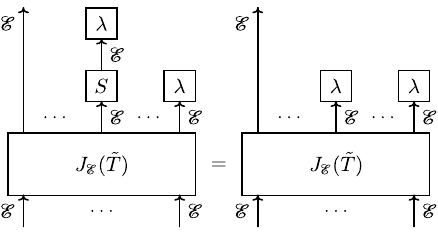}
 \end{equation}

 \textit{Independence of basepoints.} Next, we claim that $J_\eend$ is independent of the choice of the basepoint $p_j$ for every integer $1 \leqs j \leqs h$. Indeed, let $\tilde{p}_j$ be another possible choice. Up to isotoping portions of $C(T)_j$ containing $p_j$ and $\tilde{p}_j$ to the top of the diagram, making sure the one containing $\tilde{p}_j$ is nested inside the one containing $p_j$ with opposite orientation, this operation is represented by the following picture.
 \[
  \pic{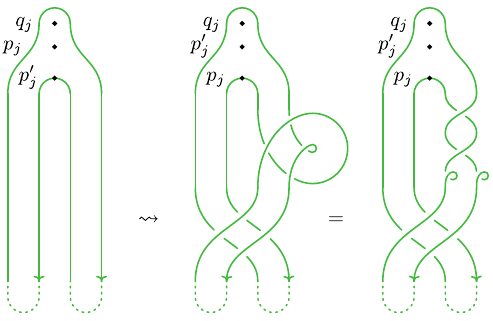}
 \]
 This means that, on one hand,
 \[
  \pic{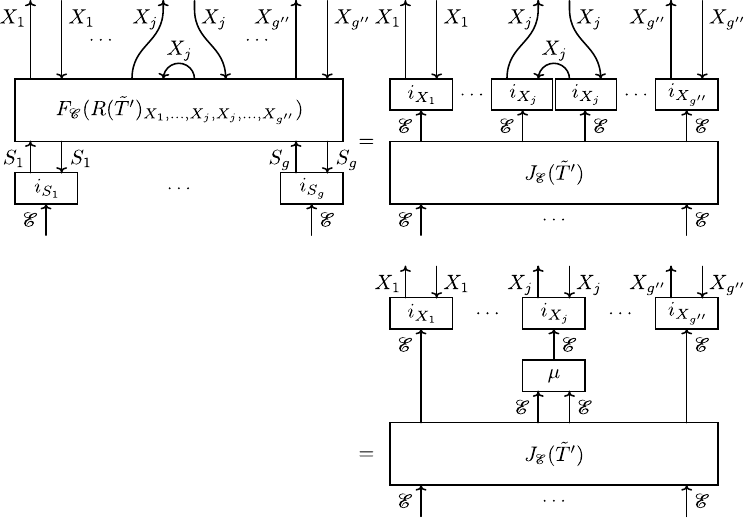}
 \]
 while, on the other hand,
 \[
  \pic{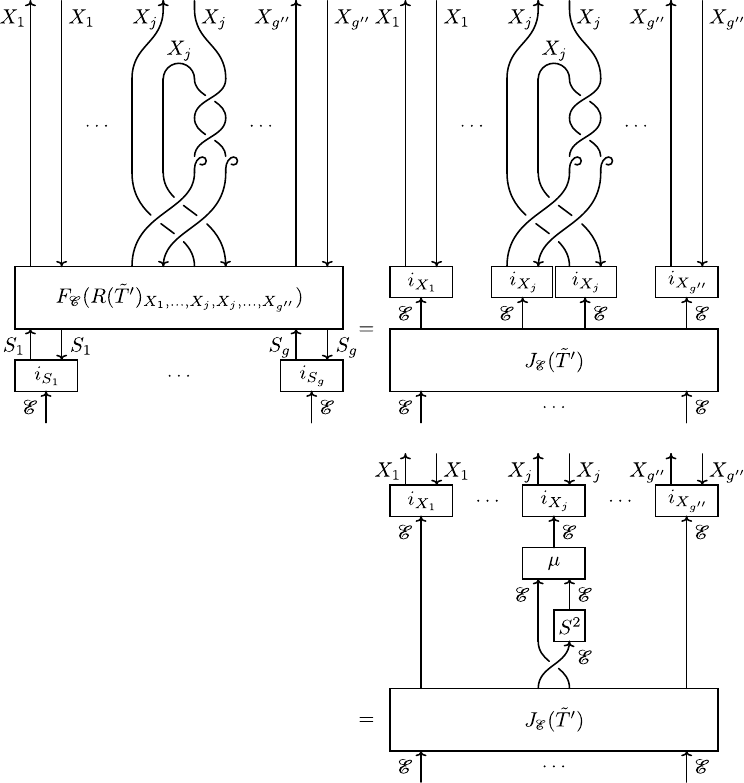}
 \]
 Then the claim follows from the fact that $\lambda$ is a quantum character on $\eend$, which implies
 \begin{equation}\label{E:basepoint}
  \pic{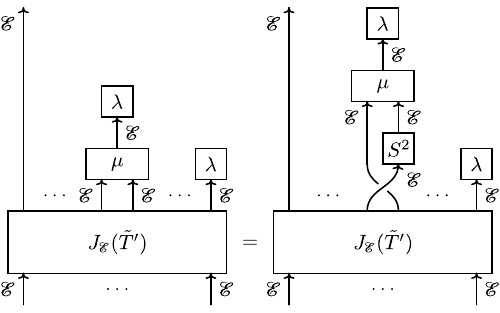}
 \end{equation}

 \textit{Invariance under slam-dunks.} Next, we claim that $J_\eend$ is invariant under the removal of a pair of closed components $C(T)_j$ and $C(T)_{j+1}$, if $C(T)_j$ is a $0$-framed meridian of $C(T)_{j+1}$. Up to isotopy, this operation is represented by the following picture.
 \[
  \pic{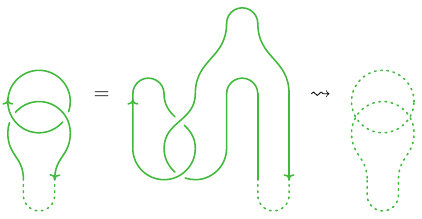}
 \]
 This means that, on one hand,
 \[
  \pic{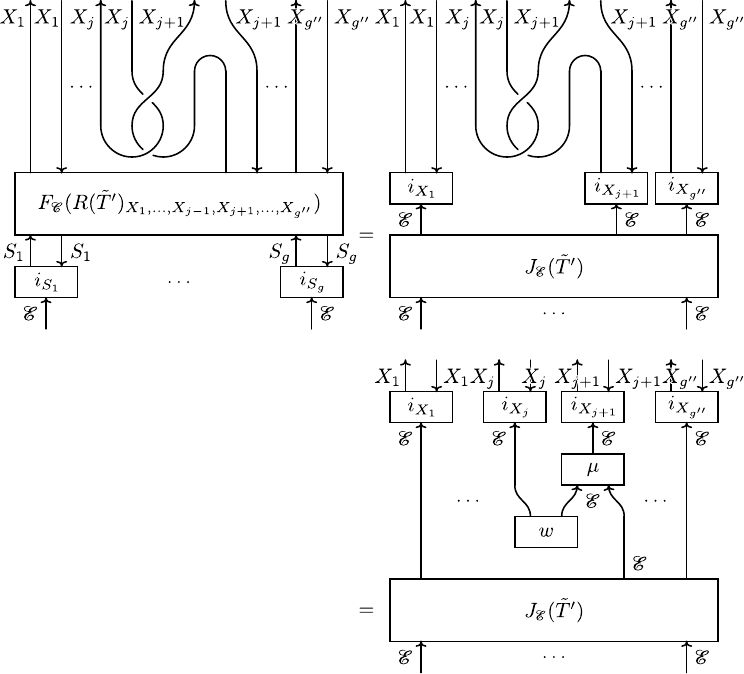}
 \]
 while, on the other hand,
 \[
  \pic{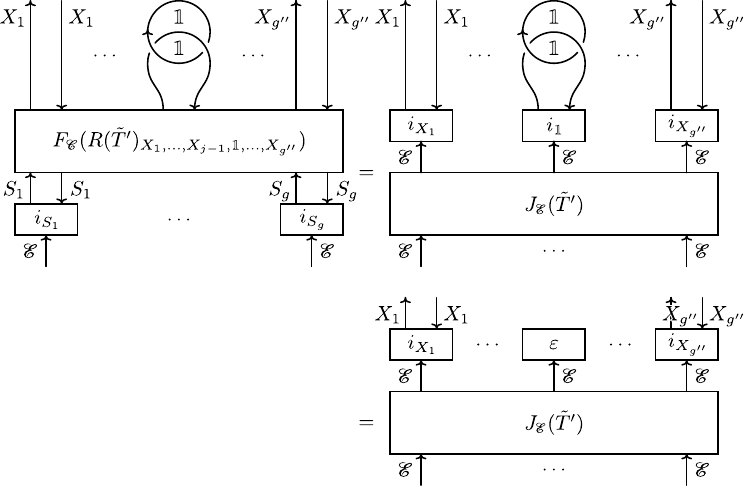}
 \]
 Then the claim follows from the fact that $\calC$ is factorizable, which implies
 \begin{equation}\label{E:slam-dunk}
  \pic{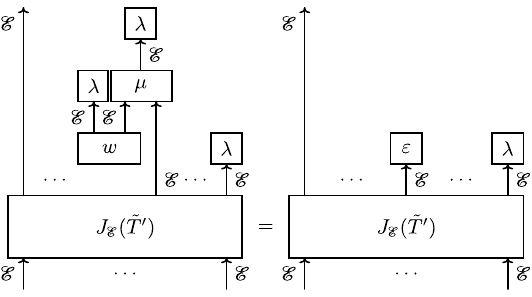}
 \end{equation}

 \textit{Invariance under slides.} Next, we claim that $J_\eend$ is invariant under the slide of an arbitrary component $T_j$ over a closed component $C(T)_{j+1}$. Up to isotopy, this operation is represented by the following picture.
 \[
  \pic{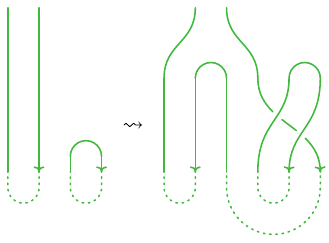}
 \]
 This means that
 \[
  \pic{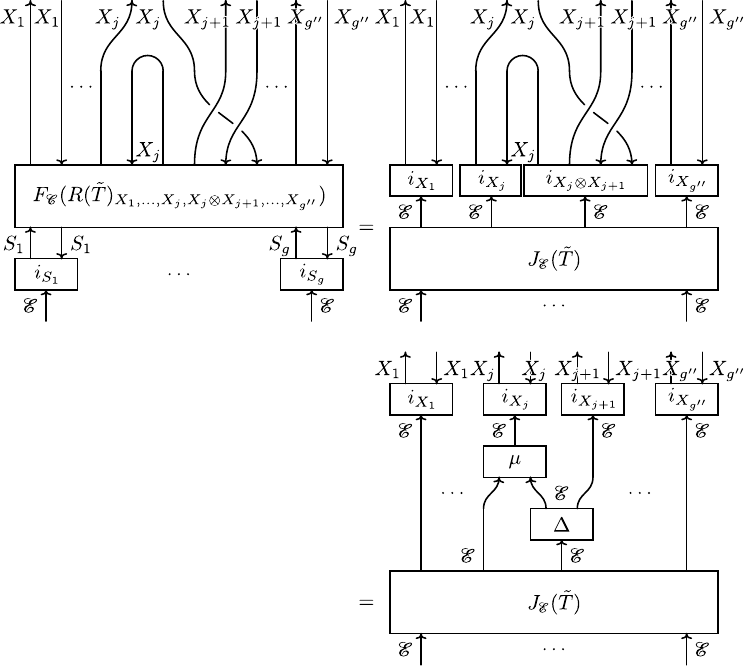}
 \]
 Then the claim follows from the fact that $\lambda$ is an integral on $\eend$, which implies
 \begin{equation}\label{E:slide}
  \pic{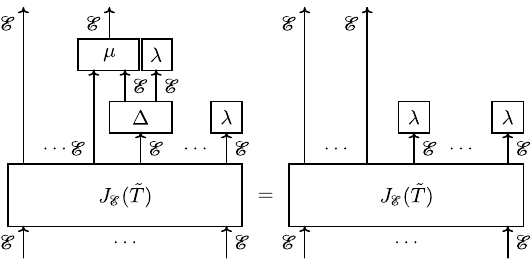}
 \end{equation}

 \textit{Functoriality.} If $g$ is an object of $\KTan$, then we clearly have
 \[
  J_\eend(\id_g) = \id_{\eend^{\otimes g}} = \id_{J_\eend(g)}.
 \]
 If $T : g \to g'$ and $T' : g' \to g''$ are morphisms of $\KTan$, if $\tilde{T} : g \to g'+h$ is a top tangle presentation of $T$, and if $\tilde{T}' : g' \to g''+h'$ is a top tangle presentation of $T'$, then a top tangle presentation of $T' \circ T$ is represented by the following picture.
 \[
  \pic{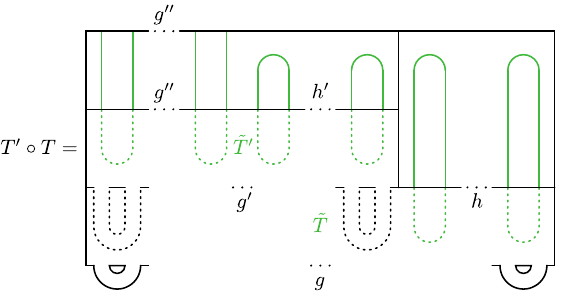}
 \]
 Then the claim follows from the functoriality of $J_\eend : \TTan \to \calC$, together with the fact that $\lambda$ is a form on $\eend$, which implies
 \begin{equation}\label{E:functoriality}
  \pic{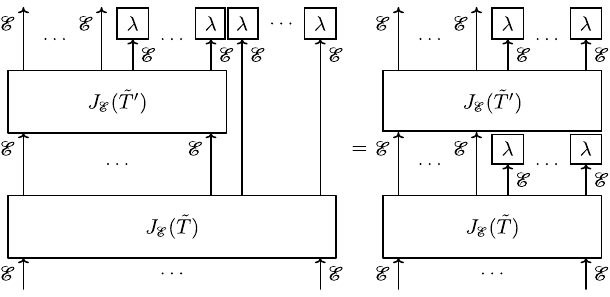}
 \end{equation}

 \textit{Monoidality.} If $g$ and $g'$ are objects of $\KTan$, then we clearly have
 \[
  J_\eend(g \bcs g') = \eend^{\otimes (g+g')} = J_\eend(g) \otimes J_\eend(g').
 \]
 If $T : g \to g''$ and $T' : g' \to g'''$ are morphisms of $\KTan$, if $\tilde{T} : g \to g''+h$ is a top tangle presentation of $T$, and if $\tilde{T}' : g' \to g'''+h'$ is a top tangle presentation of $T'$, then a top tangle presentation of $T \bcs T'$ is represented by the following picture.
 \[
  \pic{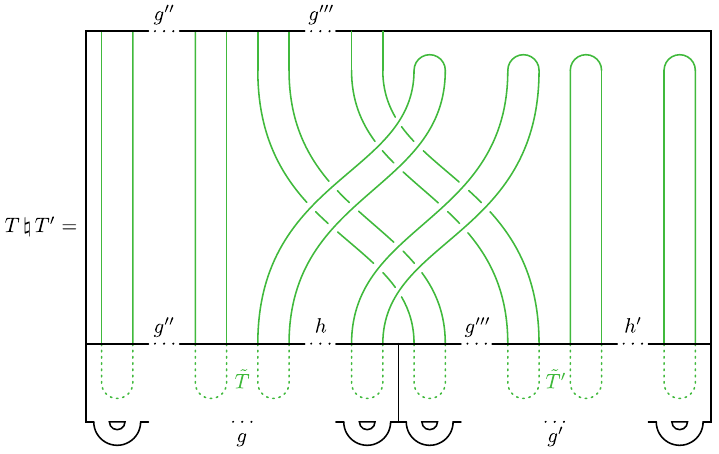}
 \]
 Then the claim follows from the monoidality of $J_\eend : \TTan \to \calC$, together with the fact that $\lambda$ is a form on $\eend$, which implies
 \begin{equation}\label{E:monoidality}
  \pic{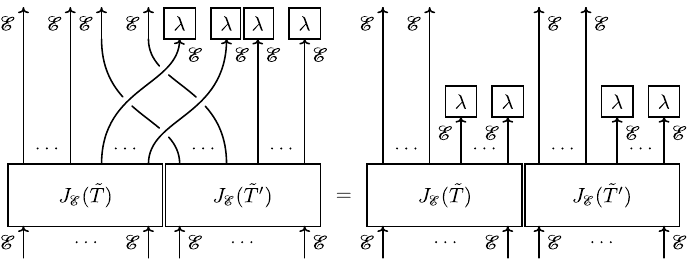}
 \end{equation}

 \textit{Braidings.} Since the braided structure of $\KTan$ coincides with that of $\TTan$, and since $J_\eend : \TTan \to \calC$ is a braided monoidal functor, then $J_\eend$ is a braided monoidal functor too. \qedhere
\end{proof}

\end{document}